\documentclass[a4paper]{article}
\usepackage{geometry}
\usepackage{hyperref}
\usepackage{graphicx}
\usepackage{amsthm,amsmath,amsfonts,amssymb,stmaryrd,mathrsfs,mathdots,amsbsy,float,cmll,xcolor,mathpartir,lmodern,centernot,colonequals,stackrel}
\usepackage{bm}
\usepackage{tikz}
\usepackage{tikz-cd}
\newsavebox{\pullback}
\sbox\pullback{%
\begin{tikzpicture}%
\draw (0,0) -- (1ex,0ex);%
\draw (1ex,0ex) -- (1ex,1ex);%
\end{tikzpicture}}
\usetikzlibrary{tikzmark}
\usepackage[small]{diagrams}
\usepackage{float}
\usetikzlibrary{positioning,calc}
\usepackage{bussproof}

\makeatletter
\DeclareRobustCommand{\notsquare}{\mathord{\mathpalette\generic@not\square}}
\newcommand{\generic@not}[2]{%
  \sbox\z@{$\m@th#1/$}%
  \sbox\tw@{$\m@th#1#2$}%
  \sbox\z@{\raisebox{\dimexpr(\ht\tw@-\dp\tw@-\ht\z@+\dp\z@)/2\relax}{$\m@th#1/$}}%
  \vphantom{\usebox{\z@}}%
  \ooalign{\hidewidth\usebox{\z@}\hidewidth\cr$\m@th#1#2$\cr}%
}
\makeatother

\theoremstyle{definition}
\newtheorem{definition}{Definition}[section]

\newtheorem{remark}[definition]{Remark}
\newtheorem{notation}[definition]{Notation}
\newtheorem{convention}[definition]{Convention}
\theoremstyle{theorem}
\newtheorem{theorem}[definition]{Theorem}
\newtheorem*{theorem*}{Theorem}

\newtheorem{lemma}[definition]{Lemma}
\newtheorem{corollary}[definition]{Corollary}
\newtheorem*{corollary*}{Corollary}
\theoremstyle{remark}

\newarrow{Corresponds}<--->
\newarrow {Dashto} {}{dash}{}{dash}>
\newarrow{Equals}=====

\DeclareFontFamily{U}{MnSymbolC}{}
\DeclareFontShape{U}{MnSymbolC}{m}{n}{
    <-6>  MnSymbolC5
   <6-7>  MnSymbolC6
   <7-8>  MnSymbolC7
   <8-9>  MnSymbolC8
   <9-10> MnSymbolC9
  <10-12> MnSymbolC10
  <12->   MnSymbolC12}{}
\DeclareSymbolFont{MnSyC}{U}{MnSymbolC}{m}{n}
\DeclareMathSymbol{\dotminus}{\mathbin}{MnSyC}{24}

\begin{document}
\title{\Huge \bfseries Sequent calculi for a unity of logic \\ \huge \bfseries Classicality is symmetric to non-linearity}
\author{\Large \bfseries Norihiro Yamada \\ \\
{\tt yamad041@umn.edu} \\
University of Minnesota
}

\maketitle

\begin{abstract}
We present a novel \emph{unity of logic}, viz., a single sequent calculus that embodies classical, intuitionistic and linear logics. 
Concretely, we define \emph{classical linear logic negative (CLL$^-$)}, a new logic that is \emph{classical} and \emph{linear} yet discards the \emph{polarities} and the \emph{strict De Morgan laws} in classical linear logic (CLL).
Then, we define \emph{unlinearisation} and \emph{classicalisation} on sequent calculi such that unlinearisation maps CLL$^-$ (resp. intuitionistic linear logic (ILL)) to classical logic (CL) (resp. intuitionistic logic (IL)), and classicalisation maps IL (resp. ILL) to CL (resp. CLL$^-$) modulo conservative extensions. 
By these two maps, only a sequent calculus for a conservative extension of ILL suffices for ILL, IL, CLL$^-$ and CL.
This result achieves a \emph{simple}, highly \emph{systematic} unity of logic by discarding the polarities and the strict De Morgan laws, which (arguably) \emph{only CLL has}, and consisting of the \emph{uniform} classicalisation and unlinearisation, which \emph{commute}.
Previous methods do not satisfy these points.
Our unity also clarifies the dichotomies between \emph{intuitionisity} and \emph{classicality}, and between \emph{linearity} and \emph{non-linearity} of logic, which are \emph{symmetric}. 
\end{abstract}

\tableofcontents

\section{Introduction}
\label{Introduction}
\subsection{Foreword}
\label{Foreword}
The present work \emph{unifies} classical, intuitionistic and linear logics into a single, most primitive logic, viz., a conservative extension of intuitionistic linear logic, in terms of sequent calculi. 
Although many mathematicians are only concerned with classical logic, this \emph{systematic} unity of logic provides a deeper analysis on their favourite logic too: Classical logic is the logic that permits unrestricted \emph{premise consumptions} and \emph{reasoning do-overs}, and is \emph{unaware} of either.

More generally, this unity of logic clarifies the dichotomies between \emph{intuitionisity} and \emph{classicality}, and between \emph{linearity} and \emph{non-linearity} of logic in a \emph{uniform} sense.
Informally, classicality is the unrestraint and the unawareness to reasoning do-overs, and non-linearity is those to resource consumptions; conversely, intuitionisity and linearity are sensitivities to the former and the latter, respectively. 
For instance, classical logic is classical and non-linear, intuitionistic logic is intuitionistic and non-linear, and so on. 
The algebraic beauty of this approach is the complete \emph{symmetry} between the two dichotomies.


In addition to this systematic nature, another advantage of our approach is its \emph{simplicity}. 
For instance, Girard's pioneering work on a unity of logic \cite{girard1993unity} employs the gigantic system \emph{\textsf{LU}}, which splits each of the logical constants and connectives into several cases, and contexts into several zones, depending on \emph{polarities}.
In contrast, our unified sequent calculus is a modest variant of the one for linear logic, which does not split constants, connectives or contexts.

\subsection{Background}
\label{Background}
\emph{(Mathematical) logic} refers to a method or law of `reasonings' or \emph{proofs} in mathematics, and it is traditionally formalised by \emph{formal systems} \cite{troelstra2000basic}.
Today, \emph{classical logic (CL)} \cite{frege1879begriffsschrift,troelstra2000basic}, \emph{intuitionistic logic (IL)} \cite{heyting1930formalen,troelstra1988constructivism,troelstra2000basic} and \emph{linear logic (LL)} \cite{girard1987linear} are arguably three among the most established logics.
In the present article, we focus on the \emph{propositional} fragments of these logics \cite{troelstra2000basic}.

CL is the most traditional and best-known logic, and even today it is still the (implicitly) official logic for most working mathematicians.
The view of CL regards the truth value (i.e., \emph{true} or \emph{false}) of every mathematical statement or \emph{formula} as a priori determined, and CL aims to find all and only `true' formulas in this predetermined sense \emph{regardless of whether we can `witness' or `construct' their truths}.
Thus, CL permits \emph{the law of excluded middle (LEM)}:
\begin{quote}
\textsc{(LEM)} The disjunction of any formula $A$ and its negation is true,
\end{quote}
where CL defines that the negation $\mathbin{\sim} A$ of $A$ is true if and only if $A$ is false, and the disjunction $A \vee B$ of formulas $A$ and $B$ is true if and only if at least one of them is true.
That is, CL allows LEM since the disjunction $A \vee \mathbin{\sim} A$ is true regardless of the truth value of $A$ \emph{even though in general we cannot compute it} \cite{rogers1987theory} (i.e., one does not have to compute the truth value for justifying LEM). 

In contrast, some mathematicians and logicians regard a formula as `true' \emph{exactly when we can `witness' or `construct' its truth} \cite{troelstra1988constructivism}.
For instance, they claim that the disjunction $A \vee B$ is `true' only if we can decide \emph{which} (of $A$ and $B$) is `true.'
IL is motivated by this computational or \emph{constructive} viewpoint on logic, and hence it refutes computationally infeasible laws such as LEM.

Finally, LL regards formulas as \emph{resources} and introduces the \emph{linearity} constraint on proofs: Every proof of a formula $A$ must consume each premise of $A$ exactly once to produce the conclusion of $A$ \cite{lafont1988introduction}. 
Strictly speaking, LL is divided into \emph{classical linear logic (CLL)} and \emph{intuitionistic linear logic (ILL)}. 
As their names indicate, CLL and ILL are intended to serve as the `classical' and the `intuitionistic' variants of LL, respectively.


Today, CL, IL, CLL and ILL are established precisely by respective formal systems \cite{troelstra2000basic,girard1995linear} that embody the aforementioned intuitions on them.
However, fundamental questions on these logics remain unsolved. 
For instance, there has been no precise formalisation of \emph{classicality} and \emph{intuitionisity} (resp. \emph{linearity} and \emph{non-linearity}) of logic applied \emph{uniformly} to the linear and non-linear (resp. classical and intuitionistic) cases; see \S\ref{PolaritiesAndConcurrencyInLogicAndGames} and \S\ref{LinearLogics}.
In other words,
\begin{quote}
What is the dichotomy between \emph{classicality} and \emph{intuitionisity}, (resp. \emph{linearity} and \emph{non-linearity}) of logic in a precise, \emph{uniform} sense?
\end{quote}
More generally, the interrelations between the four logics are not fully clarified. 

Also, one may find it frustrating to have \emph{more than one} logic since logic formalises mathematics, and mathematical truths should be `absolute.'
Hence,
\begin{quote}
Can we reduce CL, IL, CLL and ILL to a \emph{single} logic?
\end{quote}
Moreover, we want such a single logic to be relatively \emph{finer} or \emph{more primitive} so that it would deepen our understanding of the other logics. 
For instance, \emph{Girard's translation} \cite{girard1987linear} achieves such an analysis on IL by reducing it to ILL.  

The present work is motivated by these fundamental questions.
We work on \emph{sequent calculi}, a particular class of formal systems invented by Gentzen \cite{gentzen1935untersuchungen}, as they are suited to our aim.
We assume that the reader is familiar with the formal languages and the sequent calculi for CL, IL, CLL and ILL \cite{troelstra2000basic,girard1995linear}.

\subsection{Main results}
\label{MainResults}
Towards answering the aforementioned questions, our first result is:
\begin{theorem*}[Commutativity, informally]
There are two maps on sequents, unlinearisation $(\_)_\oc : \Delta \vdash \Gamma \mapsto \oc \Delta \vdash \Gamma$ and classicalisation $(\_)_\wn : \Theta \vdash \Xi \mapsto \Theta \vdash \wn \Xi$, extended to sequent calculi in the evident way, such that the diagram
\begin{mathpar}
\begin{tikzcd}
\arrow[d, hook, "\text{conservative extension}"'] \text{ILL} \arrow[rrrr, "\text{Girard's translation}"] &&&& \text{IL} \arrow[d, hook, "\text{conservative extension}"] \\
\arrow[d, "\text{classicalisation $(\_)_\wn$}"'] \text{ILL$^{\text{e}}_\iota$} \arrow[rrrr, "\text{unlinearisation $(\_)_\oc$}"'] &&&& \arrow[d, "\text{classicalisation $(\_)_\wn$}"] \text{IL$^{\text{e}}$} \\
\text{CLL$^-$} \arrow[rrrr, "\text{unlinearisation $(\_)_\oc$}"'] &&&& \text{CL}
\end{tikzcd}
\end{mathpar}
commutes modulo permuting axioms and rules in formal proofs, where $\oc$ and $\wn$ are of-course and why-not in LL, respectively \cite{girard1987linear}, $\oc \Delta \colonequals \oc D_1, \oc D_2, \dots, \oc D_n$ if $\Delta = D_1, D_2, \dots, D_n$, and similarly for $\wn \Xi$, ILL$^{\text{e}}_\iota$ and IL$^{\text{e}}$ are conservative extensions of ILL and IL, respectively, CLL$^-$ is the logic dual to IL$^{\text{e}}$, and the logics in the diagram are embodied by respective sequent calculi. 
\end{theorem*}

Precisely, the maps $(\_)_\oc$ and $(\_)_\wn$ on sequent calculi are defined by: 
\begin{definition}[Unlinearisation and classicalisation]
\label{Defunlinearisation}
Given a sequent calculus \textsf{C} that has of-course $\oc$ (resp. why-not $\wn$) \cite{girard1987linear}, the \emph{unlinearisation} (resp. \emph{classicalisation}) of \textsf{C} is its substructural sequent calculus $\textsf{C}_\oc$ (resp. $\textsf{C}_\wn$), whose sequents are those in \textsf{C}, and formal proofs are those in \textsf{C} whose roots (or conclusions) are sequents of the form $\oc \Delta \vdash \Gamma$ (resp. $\Delta \vdash \wn \Gamma$).
\end{definition}

Let us make a few remarks on this theorem. 
First, the extensions of ILL and IL to \emph{intuitionistic linear logic $\iota$-extended (ILL$^{\text{e}}_\iota$)} and \emph{intuitionistic logic extended (IL$^{\text{e}}$)}, respectively, are necessary for classicalisation $(\_)_\wn$ since neither ILL nor IL embraces why-not $\wn$.
The subscript $(\_)_\iota$ is to distinguish ILL$^{\text{e}}_\iota$ from another extension of ILL introduced below.
Next, we replace CLL with \emph{classical linear logic negative (CLL$^-$)}, which is \emph{dual} to IL$^{\text{e}}$. 
Note that CLL$^-$, not CLL, is exactly \emph{classical} and \emph{linear} in the sense of the commutative diagram. 
As explained below, CLL is what prohibits us from answering the aforementioned questions, and thus we replace it with CLL$^-$, which is more suited to our aim.
(Accordingly, we slightly modify the second question for this replacement.)
Finally, the compromise of the commutativity \emph{modulo permuting axioms and rules} is harmless since the permutations are inessential details of formal proofs as we shall see; they should be ignored in the \emph{category-theoretic} view \cite{lambek1988introduction}.

The main breakthrough made by the theorem is the \emph{simple}, highly \emph{systematic} unity of logic in the sense that it dispenses with the \emph{polarities} and the \emph{(strict) De Morgan laws}, which arguably \emph{only CLL has} (\S\ref{PolaritiesAndConcurrencyInLogicAndGames} and \S\ref{LinearLogics}), and achieves unlinearisation (resp. classicalisation) applicable \emph{uniformly} to the intuitionistic and classical (resp. linear and non-linear) cases. 
In addition, these maps \emph{commute}.
We explain these points in the rest of this introduction. 

Such a unity of logic was not achieved before.
For instance, Girard's translation works as the unlinearisation $\text{ILL} \mapsto \text{IL}$, but not $\text{CLL} \mapsto \text{CL}$.
Also, the \emph{negative translation} \cite{godel1933intuitionistischen,gentzen1969widerspruchsfreiheit,kolmogorov1925principe} translates CL into IL, but it does not work for linear logics (for the \emph{involution} of linear negation $(\_)^\bot$).
Finally, existing unities of logic \cite{girard1993unity,mellies2010resource,laurent2003translations} employ the polarities and/or the De Morgan laws, and their unlinearisation and classicalisation are not uniform or commutative.

In addition, our theorem shows that sequents $\Delta \vdash \Gamma$ in IL$^{\text{e}}$, CLL$^-$ and CL are representable by those in ILL$^{\text{e}}_\iota$ of the forms $\oc \Delta \vdash \Gamma$, $\Delta \vdash \wn \Gamma$ and $\oc \Delta \vdash \wn \Gamma$, respectively, modulo translations of logical constants and connectives.
Hence, 
\begin{quote}
\emph{Non-linearity} (resp. \emph{classicality}) is the \emph{implicit} placement of of-course $\oc$ (resp. why-not $\wn$) on elements on the left- (resp. right-) hand side of sequents; \emph{linearity} (resp. \emph{intuitionisity}) is the absence of this placement. 
\end{quote}

\begin{remark}
\label{RemarkOnLinearity}
The original \emph{linearity} of logic \cite{girard1987linear} is on the structural constraint not only on the left- but also the right-hand sides of sequents.
Our proposal is a departure from this tradition: We regard the constraint on the right-hand side instead as responsible for the dichotomy between \emph{classicality} and \emph{intuitionisity}.
Nevertheless, the conservativity of ILL$^{\text{e}}_\iota$ (resp. IL$^{\text{e}}$) over ILL (resp. IL) means that our proposal is \emph{compatible} with the traditional dichotomies except CLL.
See Remark~\ref{RemarkOnBacktracks} on why we oppose the view of CLL and regard the structural rules on the right-hand side as \emph{classicality} rather than non-linearity.
\end{remark}

We can explain our proposal in terms of the \emph{game semantics} of of-course $\oc$ and why-not $\wn$ \cite{laurent2002polarized} as follows. 
It interprets formulas $\oc A$ and $\wn A$ as \emph{an arbitrary number of copies of the formula $A$} and \emph{$A$ itself except that a formal proof of $\wn A$ consists of an arbitrary number of do-overs of formal proofs of $A$}, respectively.
Our proposed non-linearity (resp. classicality) is to regard $\oc A$ (resp. $\wn A$) as $A$.
Thus, \emph{classicality is the unrestraint and the unawareness to reasoning do-overs}, and \emph{non-linearity is those to resource consumptions}; conversely, intuitionisity and linearity are sensitivities to the former and the latter, respectively. 

\begin{remark}
\label{RemarkOnBacktracks}
\emph{Only one} of the do-overs of formal proofs of $A$ in a formal proof of $\wn A$ completes in the game semantics \cite{laurent2002polarized}, and par $\invamp$ is the binary version of why-not $\wn$. 
Hence, from this game-semantic viewpoint, the implicit placement of why-not $\wn$ on the right-hand side of sequents does \emph{not} destroy the linearity of resources; rather, it enables reasoning do-overs that \emph{preserve} the linearity.
This game-semantic analysis leads us to cast doubt on the perspective of CLL that the implicit placement of why-not $\wn$ on the right-hand side makes logic non-linear (Remark~\ref{RemarkOnLinearity}) and instead propose that the placement of why-not $\wn$ actually makes logic \emph{classical} without turning it into non-linear.
This semantic justification of our proposal is made precise in a forthcoming article (\S\ref{OurContributionsAndRelatedWork}).
\end{remark}

The novelty here is the dichotomy between intuitionisity and classicality by why-not $\wn$ (in contrast, the other dichotomy by of-course $\oc$ is already visible in Girard's translation).
In comparison to the negative translation, our method is more direct.  
For instance, LEM, $A \vee \mathbin{\sim} A$, in CL is translated roughly by $\wn (A \vee \mathbin{\sim} A)$ in IL$^{\text{e}}$, which is constructively valid \cite{laurent2002polarized} as we can have do-overs between $A$ and $\mathbin{\sim} A$. 
Note that LEM itself is invalid in IL$^{\text{e}}$.
In this way, IL$^{\text{e}}$ is aware of reasoning do-overs by the \emph{explicit} use of why-not $\wn$ (i.e., intuitionistic), but CL is not by the \emph{implicit} use (i.e., classical).
This classicality via reasoning do-overs intuitively matches Coquand's semantics of \emph{classical} arithmetic by games with \emph{backtracks} \cite{coquand1995semantics}.
Thus, although our linearity and classicality depart from those proposed by CLL (Remark~\ref{RemarkOnLinearity}), they make sense in terms of game semantics.
In particular, this game-semantic view suggests that the right structural rules are responsible for classicality, rather than linearity, of logic (Remark~\ref{RemarkOnBacktracks}).

Again, the proposed dichotomy between linearity and non-linearity (resp. intuitionisity and classicality) is applicable \emph{uniformly} to the intuitionistic and classical (resp. linear and non-linear) cases. 
Algebraically, the two dichotomies are completely \emph{symmetric} thanks to the symmetry between of-course $\oc$ and why-not $\wn$.
These dichotomies do not work if we did not replace CLL with CLL$^-$ (\S\ref{PolaritiesAndConcurrencyInLogicAndGames}). 

However, the theorem is not completely satisfactory since the maps $(\_)_\oc : \text{ILL}^{\text{e}}_\iota \mapsto \text{IL}^{\text{e}}$ and $(\_)_\wn : \text{ILL}^{\text{e}}_\iota \mapsto \text{CLL}^-$ are \emph{not conservative}, i.e., $(\text{ILL}^{\text{e}}_\iota)_\oc$ (resp. $(\text{ILL}^{\text{e}}_\iota)_\wn$) proves more formal theorems than $\text{IL}^{\text{e}}$ (resp. $\text{CLL}^-$).
Another, related problem is that our sequent calculus for ILL$^{\text{e}}_\iota$ \emph{does not enjoy cut-elimination}. 
We overcome these problems by introducing a substructural logic of ILL$^{\text{e}}_\iota$, \emph{intuitionistic linear logic $\rho$-extended (ILL$^{\text{e}}_\rho)$} and its sequent calculus, and proving:
\begin{corollary*}[Conservativity, informally]
By restricting ILL$^{\text{e}}_\iota$ into ILL$^{\text{e}}_\rho$ in the diagram of the theorem, we obtain another commutative diagram
\begin{mathpar}
\begin{tikzcd}
\arrow[d, hook, "\text{conservative extension}"'] \text{ILL} \arrow[rrrr, "\text{Girard's translation}"] &&&& \text{IL} \arrow[d, hook, "\text{conservative extension}"] \\
\arrow[d, "\text{classicalisation $(\_)_\wn$}"'] \text{ILL$^{\text{e}}_\rho$} \arrow[rrrr, "\text{unlinearisation $(\_)_\oc$}"'] &&&& \arrow[d, "\text{classicalisation $(\_)_\wn$}"] \text{IL$^{\text{e}}$} \\
\text{CLL$^-$} \arrow[rrrr, "\text{unlinearisation $(\_)_\oc$}"'] &&&& \text{CL}
\end{tikzcd}
\end{mathpar}
where unlinearisation $(\_)_\wn$ and classicalisation $(\_)_\wn$ are all conservative. 
\end{corollary*}

The sequent calculus for ILL$^{\text{e}}_\rho$ enjoys cut-elimination, from which the conservativity of the maps $(\_)_\oc$ and $(\_)_\wn$ follows. 
By this conservativity, only (the sequent calculus of) ILL$^{\text{e}}_\rho$ suffices to capture all the logics in the diagram \emph{precisely} at the level of provability, overcoming the deficiency of the theorem. 

\subsection{Polarities and De Morgan laws in classical linear logic}
\label{PolaritiesAndConcurrencyInLogicAndGames}
As mentioned above, we dispense with the \emph{polarities} and the \emph{De Morgan laws} in CLL. 
This point makes our unity of logic \emph{simple} and \emph{systematic} as follows. 

Girard introduced the \emph{polarities} of formulas, \emph{positive} and \emph{negative} ones, in his pioneering work on a unity of logic \cite{girard1993unity}.
However, it is possible to see that ILL, IL and CL have only negative formulas, and so these logics do not have the dimension of polarities (\S\ref{LinearLogics}).
Also, polarities are foreign to the standard categorical semantics, e.g., see \cite[p.~256]{girard2011blind} and \cite[Remark~34 on p.~19]{laurent2002polarized}.

In addition, recall that the De Morgan laws in CL refer to the following dualities between conjunction $\wedge$ and disjunction $\vee$ through negation $\mathbin{\sim}$:
\begin{mathpar}
\mathbin{\sim} (A \wedge B) \Leftrightarrow \mathbin{\sim} A \vee \mathbin{\sim} B
\and
\mathbin{\sim} (A \vee B) \Leftrightarrow \mathbin{\sim} A \wedge \mathbin{\sim} B.
\end{mathpar}
Note that these laws are \emph{derived logical equivalences}. 
Similarly, CLL has certain De Morgan laws, but in contrast they are rather \emph{defined in terms of equalities} between formulas \cite{girard1987linear}, which are completely exotic to ILL, IL and CL.

Hence, the polarities and the De Morgan laws in CLL seem to be irrelevant to ILL, IL and CL, and so have nothing to do with classicality or linearity of logic.
In other words, the polarities and the De Morgan laws prohibit us from understanding the two dichotomies in logic in a uniform, systematic way.

From this observation, our key idea is to replace CLL with CLL$^-$ that dispenses with the polarities and the De Morgan laws, so that we attain a simpler, more systematic unity of logic (as given by the corollary) than otherwise. 
Also, our method is compatible with the standard mathematical semantics (\S\ref{OurContributionsAndRelatedWork}). 

\begin{remark}
Arguments on polarities have not been resolved yet.
For instance, Girard himself confesses that polarisation is rather of a \emph{pragmatic} nature, and he personally \emph{hesitates} to give it a status \cite[p.~251]{girard2011blind}.
Our method does not use polarities, but we do not claim that it resolves the arguments on polarities; our aim is to just present another, \emph{simpler} unity than existing ones with polarities.
\end{remark}

\subsection{Our contributions and related work}
\label{OurContributionsAndRelatedWork}
Our main contribution is the unity of logic given by the corollary. 
This unity reduces CL, IL, CLL$^-$ and ILL to the most primitive one, viz., ILL$^{\text{e}}_\rho$, in a systematic way by the \emph{uniform} unlinearisation and classicalisation.
Also, we propose the \emph{uniform} dichotomies between intuitionisity and classicality, and between linearity and non-linearity.
Our proposal is compatible with ILL, IL and CL, but not CLL (Remark~\ref{RemarkOnLinearity}); however, the proposal makes sense from the game-semantic view (Remark~\ref{RemarkOnBacktracks}).
This method also uncovers the \emph{symmetry} between the two dichotomies. 
Such a \emph{simple}, highly \emph{systematic} unity of logic with a formulation of the dichotomies has been missing in the literature. 

Our result answers the questions proposed in \S\ref{Background}, except that we replace CLL with CLL$^-$, and more generally advances our understanding of logic.
For instance, our classicalisation is intuitively and uniformly understandable as \emph{reasoning do-overs}, which is impossible by the negative translation.
Also, we identify the general mechanism underlying the \emph{ad-hoc} intuitionistic restriction on CL to obtain IL (Definition~\ref{DefLJ}); see Remark~\ref{RemarkOnIntuitionisity}.
Moreover, the symmetry and the commutativity between our unlinearisation and classicalisation enable us to see the interrelations between logics in a simple, highly systematic way.

Similarly to other computational interpretations of CL \cite{filinski1989declarative,griffin1989formulae,murthy1991evaluation,girard1991new,parigot1992lambdamu,danos1997new,herbelin2000duality,wadler2003call}, our method enables us to understand CL in a \emph{proof-relevant} way like IL: A formula is true in CL \emph{if and only if it has a computational proof} (which may use reasoning do-overs and/or non-linear resource consumptions). 
This view on CL stands in contrast to the traditional one by truth values (\S\ref{Background}).

Our technical breakthrough is to define the substructural logic ILL$^{\text{e}}_\rho$ of ILL$^{\text{e}}_\iota$ and its sequent calculus \textsf{ILC}$_\rho$ in such a way that it \emph{enjoys cut-elimination}, and our translations become \emph{conservative}.
In particular, the cut-elimination poses a technical challenge, for which we invent a novel technique (Theorem~\ref{ThmCutEliminationForILC_Delta}). 

The work closest in spirit is the aforementioned work on a unity of logic \cite{girard1993unity} by Girard.
We can, to some extent, relate this work with our approach by translating sequents $\Delta ; \Delta' \vdash \Gamma' ; \Gamma$ in his unified formal system \emph{\textsf{LU}} into those $\Delta, \oc \Delta' \vdash \wn \Gamma', \Gamma$ in our unified sequent calculus \textsf{ILC}$_\rho$.
His approach splits contexts into several zones by the use of semicolon, while ours does not. 

Our approach stands in sharp contrast to polarised approaches to a unity of logic \cite{girard1993unity,mellies2010resource,laurent2003translations}.
For instance, the gigantic system \textsf{LU} embodies CL, IL and LL as its \emph{fragments}, for which polarities play crucial roles.
In contrast, we \emph{reduce} those logics, except that we replace CLL with CLL$^-$, into the most primitive one (ILL$^{\text{e}}_\rho$) without referring to the notion of polarities.
Also, \textsf{LU} has to split logical constants and connectives into several cases, and contexts into several zones as mentioned above, depending on the polarities of formulas. 
In contrast, our method does not have to split constants, connectives or contexts. 
In these respects, our method is much simpler. 
For another example, Laurent and Regnier \cite{laurent2003translations} achieve a commutative diagram that unifies CL, IL and certain fragments of LL, but they use \emph{polarised linear logic} instead of CLL$^-$ and the \emph{CPS-translation} for classicalisation between non-linear logics, which differ from ours.
They employ yet another translation for classicalisation between linear logics.
Similarly, another commutative diagram \cite[Figure~7]{danos1997new} given by Danos et al. does not achieve uniform unlinearisation or classicalisation. 
In contrast, our classicalisation (resp. unlinearisation) is \emph{uniformly} applicable to both of the linear and non-linear (resp. intuitionistic and classical) cases. 

Last but not least, there are categorical and game-semantic counterparts of the present work (the drafts are in preparation).
They extend the standard categorical and game semantics of Girard's translation \cite{seely1987linear,mccusker1998games} to the classical case, giving semantics to unlinearisation $(\_)_\oc$.
Categorically, it is a \emph{comonad}, and its dual, a \emph{monad}, gives semantics to classicalisation $(\_)_\wn$.
Also, this game semantics is a modest variant of the polarised one \cite{laurent2002polarized}, but a crucial difference is that the former only uses \emph{negative} games, i.e., the game-semantic counterpart of negative formulas. 
This implies that there are only negative formulas in our unity of logic, i.e., there is no dimension of polarities, which seems to explain why our approach does not have to handle polarities.
We leave it as future work to extend the present work and the semantics to \emph{predicate logics}.

\subsection{Structure of the present article}
The rest of the present article is structured as follows. 
We first review sequent calculi for CL, IL and LL in \S\ref{SequentCalculi}. 
We then introduce our new logics and their sequent calculi, and prove the theorem and the corollary mentioned above in \S\ref{CommutativeUnityOfLogic}. 
Appendices~\ref{ProofOnMu}--\ref{CutElimination} collect some simple yet lengthy details in our proofs.


%
%


\section{Review: sequent calculi for existing logics}
\label{SequentCalculi}
We first recall existing logics: CL and IL in \S\ref{SequentCalculiForClassicalAndIntuitionisticLogics}, and ILL and CLL in \S\ref{LinearLogics}.

\begin{notation}
Capital letters $A$, $B$, $C$, etc. range over formulas, and Greek capital letters $\Delta$, $\Sigma$, $\Theta$, etc. over finite sequences of formulas. 
For each $n \in \mathbb{N}$, we define $\overline{n} \colonequals \{ \, 1, 2, \dots, n \, \}$ (n.b., $\overline{0} = \emptyset$).
In formulas, every unary operation precedes any binary operation; every binary operation except (any kind of) implication is left associative, while implication is right associative. 
Given a finite sequence $\boldsymbol{s}$, we write $|\boldsymbol{s}| \in \mathbb{N}$ for its \emph{length}, i.e., the number of elements of $\boldsymbol{s}$.
We often use parentheses to clear ambiguity in formulas.
\end{notation}

\subsection{Sequent calculi for classical and intuitionistic logics}
\label{SequentCalculiForClassicalAndIntuitionisticLogics}
Let us begin with recalling the standard sequent calculi \emph{\textsf{LK}} and \emph{\textsf{LJ}} for CL and IL, respectively \cite{gentzen1935untersuchungen,troelstra2000basic}.
For our systematic approach, we define negation by implication and falsity, and include truth and the right-rule on falsity though they are only minor points.
In addition, we notationally distinguish between classical and intuitionistic truths, classical and intuitionistic conjunctions, classical and intuitionistic implications, and classical and intuitionistic negations, respectively, whose convenience will be clear shortly. 

\begin{definition}[Formulas of CL \cite{troelstra2000basic}]
\label{DefFormulasOfCLAndIL}
Formulas $A, B$ of \emph{classical logic (CL)} are defined by (the grammar of the \emph{Backus–Naur form} \cite{backus1959syntax})
\begin{mathpar}
A, B \colonequals X \mid \mathrm{tt} \mid \mathrm{ff} \mid A \wedge B \mid A \vee B \mid A \Rrightarrow B
\end{mathpar}
where $X$ ranges over propositional variables \cite{troelstra2000basic}, and we define $\mathbin{\sim} A \colonequals A \Rrightarrow \mathrm{ff}$.
We call $\mathrm{tt}$ \emph{(classical) truth}, $\mathrm{ff}$ \emph{(non-linear) falsity}, $\wedge$ \emph{(classical) conjunction}, $\vee$ \emph{(non-linear) disjunction}, $\Rrightarrow$ \emph{(classical) implication}, and $\sim$ \emph{(classical) negation}.
\end{definition}

\begin{remark}
The adjectives (in the parentheses) on truth $\mathrm{tt}$, falsity $\mathrm{ff}$, conjunction $\wedge$, disjunction $\vee$, implication $\Rrightarrow$ and negation $\sim$ make sense by the unity of logic given in \S\ref{CommutativeUnityOfLogic}, but for now it is better to simply ignore them.
\end{remark}

\begin{definition}[\textsf{LK} for CL \cite{gentzen1935untersuchungen,troelstra2000basic}]
\label{DefLK}
The sequent calculus \emph{\textsf{LK}} for CL consists of the axioms and the rules displayed in Figure~\ref{FigLK}.
\begin{figure}
\begin{mathpar}
\AxiomC{$\Delta, A, A', \Delta' \vdash \Gamma$}
\LeftLabel{\textsc{(XL)}}
\UnaryInfC{$\Delta, A', A, \Delta' \vdash \Gamma$}
\DisplayProof \and
\AxiomC{$\Delta \vdash \Gamma, B, B', \Gamma'$}
\LeftLabel{\textsc{(XR)}}
\UnaryInfC{$\Delta \vdash \Gamma, B', B, \Gamma'$}
\DisplayProof \\
\AxiomC{$\Delta \vdash \Gamma$} 
\LeftLabel{\textsc{(WL)}}
\UnaryInfC{$\Delta, A \vdash \Gamma$}
\DisplayProof \and
\AxiomC{$\Delta \vdash \Gamma$}
\LeftLabel{\textsc{(WR)}}
\UnaryInfC{$\Delta \vdash B, \Gamma$}
\DisplayProof \\
\AxiomC{$\Delta, A, A \vdash \Gamma$}
\LeftLabel{\textsc{(CL)}}
\UnaryInfC{$\Delta, A \vdash \Gamma$}
\DisplayProof \and
\AxiomC{$\Delta \vdash B, B, \Gamma$}
\LeftLabel{\textsc{(CR)}}
\UnaryInfC{$\Delta \vdash B, \Gamma$}
\DisplayProof \\
\AxiomC{}
\LeftLabel{\textsc{(Id)}}
\UnaryInfC{$A \vdash A$}
\DisplayProof \and
\AxiomC{$\Delta \vdash B, \Gamma$}
	\AxiomC{$\Delta', B \vdash \Gamma'$}
	\LeftLabel{\textsc{(Cut)}}
\BinaryInfC{$\Delta, \Delta' \vdash \Gamma, \Gamma'$}
\DisplayProof \\
\AxiomC{$\Delta \vdash \Gamma$}
\LeftLabel{\textsc{($\mathrm{tt}$L)}}
\UnaryInfC{$\Delta, \mathrm{tt} \vdash \Gamma$}
\DisplayProof \and
\AxiomC{}
\LeftLabel{\textsc{($\mathrm{tt}$R)}}
\UnaryInfC{$\vdash \mathrm{tt}$} 
\DisplayProof \and
\AxiomC{}
\LeftLabel{\textsc{($\mathrm{ff}$L)}}
\UnaryInfC{$\mathrm{ff} \vdash$}
\DisplayProof \and
\AxiomC{$\Delta \vdash \Gamma$}
\LeftLabel{\textsc{($\mathrm{ff}$R)}}
\UnaryInfC{$\Delta \vdash \mathrm{ff}, \Gamma$}
\DisplayProof \\
\AxiomC{$\Delta, A_i \vdash \Gamma$}
\LeftLabel{\textsc{($\wedge$L)}}
\RightLabel{\textsc{($i \in \overline{2}$)}}
\UnaryInfC{$\Delta, A_1 \wedge A_2 \vdash \Gamma$}
\DisplayProof \and
\AxiomC{$\Delta \vdash B_1, \Gamma$}
		\AxiomC{$\Delta \vdash B_2, \Gamma$}
	\LeftLabel{\textsc{($\wedge$R)}}
\BinaryInfC{$\Delta \vdash B_1 \wedge B_2, \Gamma$} 
\DisplayProof \\
\AxiomC{$\Delta, A_1 \vdash \Gamma$}
		\AxiomC{$\Delta, A_2 \vdash \Gamma$}
	\LeftLabel{\textsc{($\vee$L)}}
\BinaryInfC{$\Delta, A_1 \vee A_2 \vdash \Gamma$}
\DisplayProof \and
\AxiomC{$\Delta \vdash B_i, \Gamma$}
\LeftLabel{\textsc{($\vee$R)}}
\RightLabel{\textsc{($i \in \overline{2}$)}}
\UnaryInfC{$\Delta \vdash B_1 \vee B_2, \Gamma$}
\DisplayProof \\
\AxiomC{$\Delta \vdash A, \Gamma$}
	\AxiomC{$\Delta, B \vdash \Gamma$}
	\LeftLabel{\textsc{($\Rrightarrow$L)}}
\BinaryInfC{$\Delta, A \Rrightarrow B \vdash \Gamma$}
\DisplayProof \and
\AxiomC{$\Delta, A \vdash B, \Gamma$}
\LeftLabel{\textsc{($\Rrightarrow$R)}}
\UnaryInfC{$\Delta \vdash A \Rrightarrow B, \Gamma$}
\DisplayProof
\end{mathpar}
\caption{Sequent calculus \textsf{LK} for CL}
\label{FigLK}
\end{figure}
\end{definition}

\begin{definition}[Formulas of IL \cite{troelstra2000basic}]
\label{DefFormulasOfIL}
Formulas $A, B$ of \emph{intuitionistic logic (IL)} are defined by
\begin{mathpar}
A, B \colonequals X \mid \top \mid \mathrm{ff} \mid A \mathbin{\&} B \mid A \vee B \mid A \Rightarrow B
\end{mathpar}
where $X$ ranges over propositional variables, and we define $A^\star \colonequals A \Rightarrow \mathrm{ff}$.
We call $\top$ \emph{(intuitionistic) truth} or \emph{top}, $\&$ \emph{(intuitionistic) conjunction} or \emph{with}, $\Rightarrow$ \emph{(intuitionistic) implication}, and $(\_)^\star$ \emph{(intuitionistic) negation}.
\end{definition}

\begin{remark}
Again, it is better to ignore the adjectives (in the parentheses) on truth $\top$, conjunction $\&$, implication $\Rightarrow$ and negation $(\_)^\star$ until \S\ref{CommutativeUnityOfLogic}. 
\end{remark}

\begin{definition}[\textsf{LJ} for IL \cite{gentzen1935untersuchungen,troelstra2000basic}]
\label{DefLJ}
The sequent calculus \emph{\textsf{LJ}} for IL consists of the axioms and the rules of \textsf{LK} that contain only \emph{intuitionistic} (i.e., the length of the right-hand side is at most one) sequents, where truth $\mathrm{tt}$, conjunction $\wedge$ and implication $\Rrightarrow$ are replaced with the ones $\top$, $\&$ and $\Rightarrow$, respectively.
\end{definition}

For instance, there is a formal proof of the \emph{law of excluded middle (LEM)} in \textsf{LK}, i.e., $\vdash \mathbin{\sim} A \vee A$ for each formula $A$ is provable in CL:
\begin{mathpar}
\AxiomC{}
\LeftLabel{\textsc{(Id)}}
\UnaryInfC{$A \vdash A$}
\LeftLabel{\textsc{($\mathrm{ff}$R)}}
\UnaryInfC{$A \vdash \mathrm{ff}, A$}
\LeftLabel{\textsc{($\Rrightarrow$R)}}
\UnaryInfC{$\vdash \mathbin{\sim} A, A$}
\LeftLabel{\textsc{($\vee$R)}}
\UnaryInfC{$\vdash \mathbin{\sim} A \vee A, A$}
\LeftLabel{\textsc{(XR)}}
\UnaryInfC{$\vdash A, \mathbin{\sim} A \vee A$}
\LeftLabel{\textsc{($\vee$R)}}
\UnaryInfC{$\vdash \mathbin{\sim} A \vee A, \mathbin{\sim} A \vee A$}
\LeftLabel{\textsc{(CR)}}
\UnaryInfC{$\vdash \mathbin{\sim} A \vee A$}
\DisplayProof
\end{mathpar}
Note that some sequents in this formal proof have two elements on the right-hand side. 
Thus, this formal proof is invalid in \textsf{LJ} (even if we replace classical implication $\Rrightarrow$ and negation $\sim$ with the intuitionistic counterparts $\Rightarrow$ and $(\_)^\star$, respectively).
In this way, sequent calculi \emph{unify} CL and IL.

However, this unity of CL and IL is \emph{not} applicable to the linear case: The sequent calculus for ILL (Definition~\ref{DefLLJ}) is not obtained out of that for CLL (Definition~\ref{DefLLK}) by the same restriction since we must discard certain logical constants and connectives of CLL as well, e.g., \emph{linear negation} $(\_)^\bot$ (\S\ref{LinearLogics}).\footnote{We can certainly use linear negation $(\_)^\bot$ in a way that keeps the intuitionistic restriction on sequents, but (without a justification) it is forcibly prohibited in ILL. Therefore, the intuitionistic restriction per se does not map CLL to ILL unlike the non-linear case $\text{CL} \mapsto \text{IL}$.} 

Finally, let us recall the following fundamental theorem in proof theory, which was originally established by Gentzen: 
\begin{theorem}[Cut-elimination for \textsf{LK} and \textsf{LJ} \cite{gentzen1935untersuchungen,troelstra2000basic}]
\label{ThmCutEliminationForLKAndLJ}
Given a formal proof of a sequent in \textsf{LK} (resp. \textsf{LJ}), there is a formal proof of this sequent in the sequent calculus that does not use Cut.
\end{theorem}
\begin{proof}
See the original article \cite{gentzen1935untersuchungen} or a standard textbook \cite[\S4.1]{troelstra2000basic}.
\end{proof}


\subsection{Sequent calculi for classical and intuitionistic linear logics}
\label{LinearLogics}
Let us call the standard, two-sided sequent calculi for CLL and ILL \cite{girard1987linear,troelstra2000basic} \emph{\textsf{LLK}} and \emph{\textsf{LLJ}}, respectively.
For completeness of this article, we recall them too:

\begin{notation}
We write $\top$ and $1$ for the units of tensor $\otimes$ and with $\&$, respectively, i.e., we swap the traditional notations for the units \cite{girard1987linear} (similarly to \cite[\S 2.7]{troelstra1991lectures}), because we find it notationally more systematic. 
We define $e (A_1, A_2, \dots, A_k) \colonequals e A_1, e A_2, \dots, e A_k$ for each exponential $e \in \{ \oc, \wn \}$.
\end{notation}

\begin{definition}[Formulas of linear logics \cite{girard1987linear,troelstra2000basic}]
\label{DefFormulasOfLinearLogics}
Formulas $A, B$ of \emph{classical linear logic (CLL)} are defined by
\begin{mathpar}
A, B \colonequals X \mid X^\bot \mid \top \mid \bot \mid 1 \mid 0 \mid A \otimes B \mid A \invamp B \mid A \mathbin{\&} B \mid A \oplus B \mid \oc A \mid \wn A
\end{mathpar}
where $X$ ranges over propositional variables. 
We call $\top$ \emph{top}, $\bot$ \emph{bottom}, $1$ \emph{one}, $0$ \emph{zero}, $\otimes$ \emph{tensor}, $\invamp$ \emph{par}, $\&$ \emph{with}, $\oplus$ \emph{plus}, $\oc$ \emph{of-course}, and $\wn$ \emph{why-not}, and extend \emph{linear negation} $(\_)^\bot$ on propositional variables to all formulas by $(X^\bot)^\bot \colonequals X$, $\top^\bot \colonequals \bot$, $\bot^\bot \colonequals \top$, $1^\bot \colonequals 0$, $0^\bot \colonequals 1$, $(A \otimes B)^\bot \colonequals A^\bot \invamp B^\bot$, $(A \invamp B)^\bot \colonequals A^\bot \otimes B^\bot$, $(A \mathbin{\&} B)^\bot \colonequals A^\bot \oplus B^\bot$, $(A \oplus B)^\bot \colonequals A^\bot \mathbin{\&} B^\bot$, $(\oc A)^\bot \colonequals \wn A^\bot$, and $(\wn A)^\bot \colonequals \oc A^\bot$.
We call these equations the \emph{(strict) De Morgan laws} in CLL.
We define \emph{linear implication} $\multimap$ by $A \multimap B \colonequals A^\bot \invamp B$.

Formulas $A, B$ of \emph{intuitionistic linear logic (ILL)} are defined by
\begin{mathpar}
A, B \colonequals X \mid \top \mid A \otimes B \mid A \mathbin{\&} B \mid A \oplus B \mid A \rightarrowtriangle B \mid \oc A
\end{mathpar}
where again $X$ ranges over propositional variables. 
We call $\rightarrowtriangle$ \emph{unpolarised (up-) linear implication}.
\end{definition}

\begin{remark}
Unlike the standard convention, we write $\rightarrowtriangle$ and call it \emph{up-linear implication} in order to distinguish it from linear implication $\multimap$ since they are a priori unrelated.
Let us also remark that it slightly varies among authors \cite{abramsky1993computational,mellies2009categorical} which logical constants and connectives to include in ILL. 
\end{remark}

\begin{definition}[\textsf{LLK} for CLL \cite{girard1987linear}]
\label{DefLLK}
The sequent calculus \emph{\textsf{LLK}} for CLL consists of the axioms and the rules displayed in Figure~\ref{FigLLK}.
\begin{figure}
\begin{mathpar}
\AxiomC{$\Delta, A, A', \Delta' \vdash \Gamma$}
\LeftLabel{\textsc{(XL)}}
\UnaryInfC{$\Delta, A', A, \Delta' \vdash \Gamma$}
\DisplayProof \and
\AxiomC{$\Delta \vdash \Gamma, B, B', \Gamma'$}
\LeftLabel{\textsc{(XR)}}
\UnaryInfC{$\Delta \vdash \Gamma, B', B, \Gamma'$}
\DisplayProof \\
\AxiomC{$\Delta \vdash \Gamma$}
\LeftLabel{\textsc{($\oc$W)}}
\UnaryInfC{$\Delta, \oc A \vdash \Gamma$}
\DisplayProof \and
\AxiomC{$\Delta \vdash \Gamma$}
\LeftLabel{\textsc{($\wn$W)}}
\UnaryInfC{$\Delta \vdash \wn B, \Gamma$}
\DisplayProof \\
\AxiomC{$\Delta, \oc A, \oc A \vdash \Gamma$}
\LeftLabel{\textsc{($\oc$C)}}
\UnaryInfC{$\Delta, \oc A \vdash \Gamma$}
\DisplayProof \and
\AxiomC{$\Delta \vdash \wn B, \wn B, \Gamma$}
\LeftLabel{\textsc{($\wn$C)}}
\UnaryInfC{$\Delta \vdash \wn B, \Gamma$}
\DisplayProof \\
\AxiomC{$\Delta, A \vdash \Gamma$}
\LeftLabel{\textsc{($\oc$D)}}
\UnaryInfC{$\Delta, \oc A \vdash \Gamma$}
\DisplayProof \and
\AxiomC{$\Delta \vdash B, \Gamma$}
\LeftLabel{\textsc{($\wn$D)}}
\UnaryInfC{$\Delta \vdash \wn B, \Gamma$}
\DisplayProof 
\\
\AxiomC{$\oc \Delta, A \vdash \wn \Gamma$}
\LeftLabel{\textsc{($\wn$L$^{\oc\wn}$)}}
\UnaryInfC{$\oc \Delta, \wn A \vdash \wn \Gamma$}
\DisplayProof \and
\AxiomC{$\oc \Delta \vdash B, \wn \Gamma$}
\LeftLabel{\textsc{($\oc$R$^{\oc\wn}$)}}
\UnaryInfC{$\oc \Delta \vdash \oc B, \wn \Gamma$}
\DisplayProof \\
\AxiomC{}
\LeftLabel{\textsc{(Id)}}
\UnaryInfC{$A \vdash A$}
\DisplayProof \and
\AxiomC{$\Delta \vdash B, \Gamma$}
	\AxiomC{$\Delta', B \vdash \Gamma'$}
	\LeftLabel{\textsc{(Cut)}}
\BinaryInfC{$\Delta, \Delta' \vdash \Gamma, \Gamma'$}
\DisplayProof 
\\
\AxiomC{}
\LeftLabel{\textsc{($1$R)}}
\UnaryInfC{$\Delta \vdash 1, \Gamma$}
\DisplayProof 
\and
\AxiomC{}
\LeftLabel{\textsc{($0$L)}}
\UnaryInfC{$\Delta, 0 \vdash \Gamma$}
\DisplayProof  
\\
\AxiomC{$\Delta \vdash \Gamma$}
\LeftLabel{\textsc{($\top$L)}}
\UnaryInfC{$\Delta, \top \vdash \Gamma$}
\DisplayProof \and
\AxiomC{}
\LeftLabel{\textsc{($\top$R)}}
\UnaryInfC{$\vdash \top$}
\DisplayProof \and
\AxiomC{}
\LeftLabel{\textsc{($\bot$L)}}
\UnaryInfC{$\bot \vdash$}
\DisplayProof \and
\AxiomC{$\Delta \vdash \Gamma$}
\LeftLabel{\textsc{($\bot$R)}}
\UnaryInfC{$\Delta \vdash \bot, \Gamma$}
\DisplayProof \\
\AxiomC{$\Delta, A_1, A_2 \vdash \Gamma$}
\LeftLabel{\textsc{($\otimes$L)}}
\UnaryInfC{$\Delta, A_1 \otimes A_2 \vdash \Gamma$}
\DisplayProof \and
\AxiomC{$\Delta_1 \vdash B_1, \Gamma_1$}
		\AxiomC{$\Delta_2 \vdash B_2, \Gamma_2$}
	\LeftLabel{\textsc{($\otimes$R)}}
\BinaryInfC{$\Delta_1, \Delta_2 \vdash B_1 \otimes B_2, \Gamma_1, \Gamma_2$}
\DisplayProof \\ 
\AxiomC{$\Delta, A_i \vdash \Gamma$}
	\LeftLabel{\textsc{($\&$L)}}
	\RightLabel{($i \in \overline{2}$)}
\UnaryInfC{$\Delta, A_1 \mathbin{\&} A_2 \vdash \Gamma$}
\DisplayProof \and
\AxiomC{$\Delta \vdash B_1, \Gamma$}
		\AxiomC{$\Delta \vdash B_2, \Gamma$}
	\LeftLabel{\textsc{($\&$R)}}
\BinaryInfC{$\Delta \vdash B_1 \mathbin{\&} B_2, \Gamma$}
\DisplayProof \\
\AxiomC{$\Delta_1, A_1 \vdash \Gamma_1$}
		\AxiomC{$\Delta_2, A_2 \vdash \Gamma_2$}
	\LeftLabel{\textsc{($\invamp$L)}}
\BinaryInfC{$\Delta_1, \Delta_2, A_1 \invamp A_2 \vdash \Gamma_1, \Gamma_2$}
\DisplayProof \and
\AxiomC{$\Delta \vdash B_1, B_2, \Gamma$}
\LeftLabel{\textsc{($\invamp$R)}}
\UnaryInfC{$\Delta \vdash B_1 \invamp B_2, \Gamma$}
\DisplayProof \\
\AxiomC{$\Delta, A_1 \vdash \Gamma$}
		\AxiomC{$\Delta, A_2 \vdash \Gamma$}
	\LeftLabel{\textsc{($\oplus$L)}}
\BinaryInfC{$\Delta, A_1 \oplus A_2 \vdash \Gamma$}
\DisplayProof \and
\AxiomC{$\Delta \vdash B_i, \Gamma$}
	\LeftLabel{\textsc{($\oplus$R)}}
	\RightLabel{($i \in \overline{2}$)}
\UnaryInfC{$\Delta \vdash B_1 \oplus B_2, \Gamma$}
\DisplayProof \\
\AxiomC{$\Delta \vdash B, \Gamma$}
\LeftLabel{\textsc{($(\_)^\bot$L)}}
\UnaryInfC{$\Delta, B^\bot \vdash \Gamma$}
\DisplayProof \and
\AxiomC{$\Delta, A \vdash \Gamma$}
\LeftLabel{\textsc{($(\_)^\bot$R)}}
\UnaryInfC{$\Delta \vdash A^\bot, \Gamma$}
\DisplayProof
\end{mathpar}
\caption{Sequent calculus \textsf{LLK} for CLL}
\label{FigLLK}
\end{figure}
\end{definition}

\begin{definition}[\textsf{LLJ} for ILL \cite{girard1987linear,abramsky1993computational,mellies2009categorical}]
\label{DefLLJ}
The sequent calculus \emph{\textsf{LLJ}} for ILL consists of the axioms and the rules of \textsf{LLK} on exchange, identity, cut, top $\top$, tensor $\otimes$, with $\&$, plus $\oplus$ and of-course $\oc$ that contain only intuitionistic sequents, as well as the following rules on up-linear implication:
\begin{mathpar}
\AxiomC{$\Delta \vdash A$}
\AxiomC{$\Gamma, B \vdash C$}
\LeftLabel{\textsc{($\rightarrowtriangle$L)}}
\BinaryInfC{$\Delta, \Gamma, A \rightarrowtriangle B \vdash C$}
\DisplayProof \and
\AxiomC{$\Delta, A \vdash B$}
\LeftLabel{\textsc{($\rightarrowtriangle$R)}}
\UnaryInfC{$\Delta \vdash A \rightarrowtriangle B$}
\DisplayProof
\end{mathpar}
\end{definition}

CLL embodies the De Morgan laws \emph{by the definition} of linear negation $(\_)^\bot$ and \emph{in terms of equalities} between formulas, which stands in contrast to the De Morgan laws in CL (\S\ref{PolaritiesAndConcurrencyInLogicAndGames}).
These equalities do not hold in CL, IL or ILL. 

Another peculiarity of linear negation $(\_)^\bot$ is that it switches the \emph{polarities} of formulas, which consist of \emph{positive} and \emph{negative} ones \cite{girard1993unity}, while (classical) negation $\mathbin{\sim} (\_)$ in CL does not.
Furthermore, it is possible to see that ILL, IL and CL have only negative formulas \cite{girard1993unity,laurent2003translations}.
Even \emph{game semantics} \cite{hyland1997game}, which interprets positive (resp. negative) formulas by \emph{positive} (resp. \emph{negative}) \emph{games} \cite{laurent2002polarized}, supports this viewpoint: The game semantics of ILL, IL and CL \cite{mccusker1998games,laird1999semantic,herbelin1997games,blot2017realizability,laurent2002polarized} employs only negative games.
For these points, we replace CLL with a logic (viz., CLL$^-$) without linear negation $(\_)^\bot$ in \S\ref{ClassicalLinearLogicNegative}.
This replacement enables us to obtain a simple, highly systematic unity of logic, which dispenses with the polarities and the De Morgan laws, in \S\ref{CommutativeUnityOfLogic}.
It is also worth noting that the game semantics of this unity mentioned in \S\ref{OurContributionsAndRelatedWork} uses only negative games, i.e., semantically our unity of logic does not have the dimension of polarities. 


Finally, recall a well-known translation of sequents $\Delta \vdash B$ in \textsf{LJ} into those $\oc \Delta \vdash B$ in \textsf{LLJ}, called \emph{Girard's translation} \cite{girard1987linear}, which is abstracted by category theory too \cite{seely1987linear}.
Hence, one may wonder if Girard's translation is to be called \emph{unlinearisation}.
However, it does not seem to work for the classical case since existing translations of CL into (variants of) CLL \cite{girard1987linear,girard1993unity,laurent2003translations} are all different from Girard's translation. 
This is another motivation of ours to replace CLL. 

As a summary of this section, we restate that universal operations on logic to be called \emph{unlinearisation} and \emph{classicalisation} in the sense that the diagram
\begin{mathpar}
\begin{tikzcd}
\arrow[d, "\text{classicalisation}"'] \text{ILL} \arrow[rrrr, "\text{unlinearisation}"] &&&& \arrow[d, "\text{classicalisation}"] \text{IL} \\
\text{CLL} \arrow[rrrr, "\text{unlinearisation}"'] &&&& \text{CL}
\end{tikzcd}
\end{mathpar}
commutes have not been established.
For example, Girard's translation works as the unlinearisation $\text{ILL} \mapsto \text{IL}$, but not $\text{CLL} \mapsto \text{CL}$; the \emph{negative translation} \cite{troelstra2000basic} works for the classicalisation $\text{IL} \mapsto \text{CL}$, but not $\text{ILL} \mapsto \text{CLL}$ since linear negation $(\_)^\bot$ is involutive.
As the converse of classicalisation, the intuitionistic restriction works for obtaining IL out of CL, but not ILL out of CLL; thus, the manipulation of the number of elements on sequents does not work either.
The lack of such \emph{uniform} unlinearisation and classicalisation on logic is a main problem we address in the next section. 
For this task, our key idea is to modify CLL by discarding its rather separate polarities and De Morgan laws. 

\section{Commutative unity of logic and conservative translations}
\label{CommutativeUnityOfLogic}
Having reviewed the existing logics and their sequent calculi, we are now ready to present the main contributions of the present work.

We first introduce a conservative extension of ILL, called \emph{intuitionistic linear logic $\iota$-extended (ILL$^{\text{e}}_{\iota}$)}, to which other logics in this article are reduced, and a sequent calculus \emph{\textsf{ILC$_{\iota}$}} for ILL$^{\text{e}}_{\iota}$ in \S\ref{ConservativeExtensionOfIntuitionisticLinearLogic}.
Similarly, we introduce a conservative extension of IL, called \emph{intuitionistic logic extended (IL$^{\text{e}}$)}, and a sequent calculus \emph{\textsf{INC}} for IL$^{\text{e}}$ in \S\ref{ConservativeExtensionOfIntuitionisticLogic}. 
Further, we define \emph{classical linear logic negative (CLL$^-$)} and a sequent calculus \emph{\textsf{CLC}} for CLL$^-$ in \S\ref{ClassicalLinearLogicNegative}; CLL$^-$ and \textsf{CLC} are \emph{dual} to IL$^{\text{e}}$ and \textsf{INC}, respectively. 
We then prove the theorem stated in \S\ref{MainResults} in terms of these sequent calculi in \S\ref{SubsectionCommutativeUnityOfLogic}.
Finally, we refine this theorem into the corollary formulated in \S\ref{MainResults} by replacing ILL$^{\text{e}}_{\iota}$ and \textsf{ILC$_{\iota}$} with their substructural alternatives, ILL$^{\text{e}}_{\rho}$ and \textsf{ILC$_{\rho}$} in \S\ref{ConservativeTranslations}, respectively, so that the translations given in the theorem become all \emph{conservative}. 

\begin{notation}
The three digits in the naming of our sequent calculi signify whether the logic embodied by a calculus is \underline{I}ntuitionistic (\textsf{I}) or \underline{C}lassical (\textsf{C}), \underline{L}inear (\textsf{L}) or \underline{N}on-linear (\textsf{N}), and the point that it is a certain \underline{C}alculus (\textsf{C}), respectively (e.g., \textsf{ILC$_\iota$} is an \underline{I}ntuitionistic \underline{L}inear \underline{C}alculus). 
Moreover, the subscript on \textsf{ILC} signifies whether cut-elimination for the calculus is \underline{I}mpossible ($\iota$) or \underline{P}ossible ($\rho$) (e.g., cut-elimination for \textsf{ILC$_\iota$} is \underline{I}mpossible). 
\end{notation}

\subsection{Conservative extension of intuitionistic linear logic}
\label{ConservativeExtensionOfIntuitionisticLinearLogic}
Let us begin with defining the formal language of ILL$^{\text{e}}_{\iota}$. 
The idea is as follows. 
As explained in \S\ref{LinearLogics}, neither the polarities nor the strict De Morgan laws of linear negation $(\_)^\bot$ in CLL holds in ILL, IL or CL.
Hence, linear negation $(\_)^\bot$ seems to prohibit us from obtaining the unity of logic articulated in \S\ref{Introduction}. 

Then, by discarding linear negation $(\_)^\bot$, more precisely its polarities and strict De Morgan laws, out of the formal language of CLL, we obtain:
\begin{definition}[Formulas of ILL$^{\text{e}}_{\boldsymbol{(\iota)}}$]
The formal language of \emph{intuitionistic linear logic extended (ILL$^{\text{e}}$)} and \emph{intuitionistic linear logic $\iota$-extended (ILL$^{\text{e}}_{\iota}$)} is obtained out of that of CLL by replacing linear negation $(\_)^\bot$ with \emph{unpolarised (up-) linear negation} $\neg (\_)$, i.e., formulas $A, B$ of ILL$^{\text{e}}_{(\iota)}$ are defined by
\begin{mathpar}
A, B \colonequals X \mid \top \mid \bot \mid 1 \mid 0 \mid A \otimes B \mid A \invamp B \mid A \mathbin{\&} B \mid A \oplus B \mid \neg A \mid \oc A \mid \wn A
\end{mathpar}
where $X$ ranges over propositional variables, and we define $A \rightarrowtriangle B \colonequals \neg A \invamp B$.

The naming of the logical constants and connectives other than up-linear negation follows that for CLL and ILL (Definition~\ref{DefFormulasOfLinearLogics}).
\end{definition}

ILL$^{\text{e}}$ is just an auxiliary concept as explained below.
Next, let us introduce a sequent calculus \emph{\textsf{ILC}} for ILL$^{\text{e}}$, which is the corresponding fragment of \textsf{LLK}: 
\begin{definition}[\textsf{ILC} for ILL$^{\text{e}}$]
\label{DefILC}
The sequent calculus \emph{\textsf{ILC}} for ILL$^{\text{e}}$ consists of the axioms and the rules displayed in Figure~\ref{FigILLK}. 
\begin{figure}
\begin{mathpar}
\AxiomC{$\Delta, A, A', \Delta' \vdash \Gamma$}
\LeftLabel{\textsc{(XL)}}
\UnaryInfC{$\Delta, A', A, \Delta' \vdash \Gamma$}
\DisplayProof \and
\AxiomC{$\Delta \vdash \Gamma, B, B', \Gamma'$}
\LeftLabel{\textsc{(XR)}}
\UnaryInfC{$\Delta \vdash \Gamma, B', B, \Gamma'$}
\DisplayProof \\
\AxiomC{$\Delta \vdash \Gamma$}
\LeftLabel{\textsc{($\oc$W)}}
\UnaryInfC{$\Delta, \oc A \vdash \Gamma$}
\DisplayProof \and
\AxiomC{$\Delta \vdash \Gamma$}
\LeftLabel{\textsc{($\wn$W)}}
\UnaryInfC{$\Delta \vdash \wn B, \Gamma$}
\DisplayProof \\
\AxiomC{$\Delta, \oc A, \oc A \vdash \Gamma$}
\LeftLabel{\textsc{($\oc$C)}}
\UnaryInfC{$\Delta, \oc A \vdash \Gamma$}
\DisplayProof 
\and
\AxiomC{$\Delta \vdash \wn B, \wn B, \Gamma$}
\LeftLabel{\textsc{($\wn$C)}}
\UnaryInfC{$\Delta \vdash \wn B, \Gamma$}
\DisplayProof \\
\AxiomC{$\Delta, A \vdash \Gamma$}
\LeftLabel{\textsc{($\oc$D)}}
\UnaryInfC{$\Delta, \oc A \vdash \Gamma$}
\DisplayProof \and
\AxiomC{$\Delta \vdash B, \Gamma$}
\LeftLabel{\textsc{($\wn$D)}}
\UnaryInfC{$\Delta \vdash \wn B, \Gamma$}
\DisplayProof 
\\
\AxiomC{$\oc \Delta, A \vdash \wn \Gamma$}
\LeftLabel{\textsc{($\wn$L$^{\oc\wn}$)}}
\UnaryInfC{$\oc \Delta, \wn A \vdash \wn \Gamma$}
\DisplayProof \and
\AxiomC{$\oc \Delta \vdash B, \wn \Gamma$}
\LeftLabel{\textsc{($\oc$R$^{\oc\wn}$)}}
\UnaryInfC{$\oc \Delta \vdash \oc B, \wn \Gamma$}
\DisplayProof 
\\
\AxiomC{}
\LeftLabel{\textsc{(Id)}}
\UnaryInfC{$A \vdash A$}
\DisplayProof 
\and
\AxiomC{$\Delta \vdash B, \Gamma$}
	\AxiomC{$\Delta', B \vdash \Gamma'$}
	\LeftLabel{\textsc{(Cut)}}
\BinaryInfC{$\Delta, \Delta' \vdash \Gamma, \Gamma'$}
\DisplayProof 
\\
\AxiomC{}
\LeftLabel{\textsc{($1$R)}}
\UnaryInfC{$\Delta \vdash 1, \Gamma$}
\DisplayProof 
\and
\AxiomC{}
\LeftLabel{\textsc{($0$L)}}
\UnaryInfC{$\Delta, 0 \vdash \Gamma$}
\DisplayProof  
\\
\AxiomC{$\Delta \vdash \Gamma$}
\LeftLabel{\textsc{($\top$L)}}
\UnaryInfC{$\Delta, \top \vdash \Gamma$}
\DisplayProof \and
\AxiomC{}
\LeftLabel{\textsc{($\top$R)}}
\UnaryInfC{$\vdash \top$}
\DisplayProof \and
\AxiomC{}
\LeftLabel{\textsc{($\bot$L)}}
\UnaryInfC{$\bot \vdash$}
\DisplayProof \and
\AxiomC{$\Delta \vdash \Gamma$}
\LeftLabel{\textsc{($\bot$R)}}
\UnaryInfC{$\Delta \vdash \bot, \Gamma$}
\DisplayProof \\
\AxiomC{$\Delta, A_1, A_2 \vdash \Gamma$}
\LeftLabel{\textsc{($\otimes$L)}}
\UnaryInfC{$\Delta, A_1 \otimes A_2 \vdash \Gamma$}
\DisplayProof \and
\AxiomC{$\Delta_1 \vdash B_1, \Gamma_1$}
		\AxiomC{$\Delta_2 \vdash B_2, \Gamma_2$}
	\LeftLabel{\textsc{($\otimes$R)}}
\BinaryInfC{$\Delta_1, \Delta_2 \vdash B_1 \otimes B_2, \Gamma_1, \Gamma_2$} 
\DisplayProof \\ 
\AxiomC{$\Delta, A_i \vdash \Gamma$}
	\LeftLabel{\textsc{($\&$L)}}
	\RightLabel{($i \in \overline{2}$)}
\UnaryInfC{$\Delta, A_1 \mathbin{\&} A_2 \vdash \Gamma$}
\DisplayProof \and
\AxiomC{$\Delta \vdash B_1, \Gamma$}
		\AxiomC{$\Delta \vdash B_2, \Gamma$}
	\LeftLabel{\textsc{($\&$R)}}
\BinaryInfC{$\Delta \vdash B_1 \mathbin{\&} B_2, \Gamma$}
\DisplayProof \\
\AxiomC{$\Delta_1, A_1 \vdash \Gamma_1$}
		\AxiomC{$\Delta_2, A_2 \vdash \Gamma_2$}
	\LeftLabel{\textsc{($\invamp$L)}}
\BinaryInfC{$\Delta_1, \Delta_2, A_1 \invamp A_2 \vdash \Gamma_1, \Gamma_2$}
\DisplayProof \and
\AxiomC{$\Delta \vdash B_1, B_2, \Gamma$}
\LeftLabel{\textsc{($\invamp$R)}}
\UnaryInfC{$\Delta \vdash B_1 \invamp B_2, \Gamma$}
\DisplayProof \\
\AxiomC{$\Delta, A_1 \vdash \Gamma$}
		\AxiomC{$\Delta, A_2 \vdash \Gamma$}
	\LeftLabel{\textsc{($\oplus$L)}}
\BinaryInfC{$\Delta, A_1 \oplus A_2 \vdash \Gamma$}
\DisplayProof \and
\AxiomC{$\Delta \vdash B_i, \Gamma$}
	\LeftLabel{\textsc{($\oplus$R)}}
	\RightLabel{($i \in \overline{2}$)}
\UnaryInfC{$\Delta \vdash B_1 \oplus B_2, \Gamma$}
\DisplayProof \\
\AxiomC{$\Delta \vdash B, \Gamma$}
\LeftLabel{\textsc{($\neg$L)}}
\UnaryInfC{$\Delta, \neg B \vdash \Gamma$}
\DisplayProof \and
\AxiomC{$\Delta, A \vdash \Gamma$}
\LeftLabel{\textsc{($\neg$R)}}
\UnaryInfC{$\Delta \vdash \neg A, \Gamma$}
\DisplayProof
\end{mathpar}
\caption{Sequent calculus \textsf{ILC} for ILL$^{\text{e}}$}
\label{FigILLK}
\end{figure}
\end{definition}

We then obtain our sequent calculus \emph{\textsf{ILC$_\iota$}} for ILL$^{\text{e}}_\iota$ out of  \emph{\textsf{ILC}} as follows:
\begin{definition}[\textsf{ILC$_{\boldsymbol{\iota}}$} for ILL$^{\text{e}}_{\boldsymbol{\iota}}$]
\label{DefILC_Mu}
The sequent calculus \emph{\textsf{ILC$_{\iota}$}} for ILL$^{\text{e}}_\iota$ consists of the axioms and the rules of \textsf{ILC}, and the following \emph{weakly distributive rules}:
\begin{mathpar}
\AxiomC{$\oc \Delta, \oc A \vdash \wn \Gamma$}
\LeftLabel{\textsc{($\oc\wn$L$^{\oc\wn}$)}}
\UnaryInfC{$\oc \Delta, \oc \wn A \vdash \wn \Gamma$}
\DisplayProof \and
\AxiomC{$\oc \Delta \vdash \wn B, \wn \Gamma$}
\LeftLabel{\textsc{($\wn\oc$R$^{\oc\wn}$)}}
\UnaryInfC{$\oc \Delta \vdash \wn \oc B, \wn \Gamma$}
\DisplayProof 
\end{mathpar}
\end{definition}

Hence, CLL (resp. \textsf{LLK}) and ILL$^{\text{e}}$ (resp. \textsf{ILC}) differ only in their negations: linear negation $(\_)^\bot$ and up-linear negation $\neg$. 
These negations differ only in the point that the former a priori satisfies the De Morgan laws by definition, but the latter does not.
Also, we show that up-linear negation $\neg$ does not bring polarities as follows.\footnote{A point more crucial than this argument is that by replacing linear negation $(\_)^\bot$ with up-linear negation $\neg$ our unity of logic does not have to handle polarities as we shall see.}
First, we can alternatively define up-linear negation $\neg$ by $\neg A \colonequals A \rightarrowtriangle \bot$ and replace the rules $\neg$L and $\neg$R of \textsf{ILC} with those $\rightarrowtriangle$L and $\rightarrowtriangle$R of \textsf{LLJ} if we allow any finite number of elements on the right-hand side of sequents in the latter two rules. 
The converse translation is also possible by $A \rightarrowtriangle B \colonequals \neg A \invamp B$.
We leave the details to the reader.
In this way, up-linear negation $\neg$ and up-linear implication $\rightarrowtriangle$ are derivable by each other inside \textsf{ILC}.
Since ILL contains only negative formulas, including up-linear implication $\rightarrowtriangle$ (\S\ref{LinearLogics}), and bottom $\bot$ is negative \cite{girard1993unity}, it follows from the equation $\neg A = A \rightarrowtriangle \bot$ that up-linear negation $\neg$ does not generate positive formulas.

On the other hand, at least one of the weakly distributive rules $\oc\wn$L$^{\oc\wn}$ and $\wn\oc$R$^{\oc\wn}$ is necessary for our translations of logics given later, and it is why our interests are in ILL$^{\text{e}}_\iota$ rather than ILL$^{\text{e}}$.
Our idea on these rules comes from the morphisms $\oc \wn A \rightarrow \wn \oc A$ used in the categorical reformulation of game semantics \cite{harmer2007categorical}, on which our game semantics (\S\ref{OurContributionsAndRelatedWork}) is based. 
In fact, the rules $\oc\wn$L$^{\oc\wn}$ and $\wn\oc$R$^{\oc\wn}$ provide the sequent $\oc \wn A \vdash \wn \oc A$ in \textsf{ILC} with trivially different formal proofs.

Although the formal languages of ILL$^{\text{e}}$ and ILL$^{\text{e}}_\iota$ coincide, ILL$^{\text{e}}_\iota$ is stronger than ILL$^{\text{e}}$ because, e.g., the sequent $\oc \wn A \vdash \wn \oc A$ for any formula $A$ is provable in \textsf{ILC$_\iota$} but not in \textsf{ILC}, where we leave the details to the reader. 

\begin{remark}
For the only slight modification of CLL, one may wonder if ILL$^{\text{e}}_{(\iota)}$ is \emph{intuitionistic} especially because we allow more than one formula to occur on the right-hand side of sequents in \textsf{ILC$_{(\iota)}$}.
We give a \emph{positive} answer to this question by Corollary~\ref{ILCAsConservativeExtensionOfLLJ}. 
Also, we regard LEM on par $\invamp$ as \emph{intuitionistically valid}; \textsf{ILC$_{(\iota)}$} clearly has a proof of the sequent $\vdash \neg A \invamp A$ for any formula $A$.
\end{remark}

Let us proceed to prove the following \emph{cut-elimination theorem} for \textsf{ILC}:
\begin{theorem}[Cut-elimination for \textsf{ILC}]
\label{ThmCutEliminationForILC}
Given a formal proof of a sequent in \textsf{ILC}, there is a formal proof of this sequent in \textsf{ILC} without Cut.
\end{theorem}
\begin{proof}
Based on a standard cut-elimination procedure given in \cite{bierman1994intuitionistic}; see \S\ref{CutElimination}.
\end{proof}

In contrast, \textsf{ILC$_\iota$} does \emph{not} enjoy cut-elimination since, e.g., the sequent $\oc (X \invamp X) \vdash \wn (X \otimes X)$ is derivable in \textsf{ILC$_\iota$}, as shown by the formal proof
\begin{mathpar}
\AxiomC{}
\LeftLabel{\textsc{(Id)}}
\UnaryInfC{$X \vdash X$}
\AxiomC{}
\LeftLabel{\textsc{(Id)}}
\UnaryInfC{$X \vdash X$}
\LeftLabel{\textsc{($\invamp$L)}}
\BinaryInfC{$X \invamp X \vdash X, X$}
\LeftLabel{\textsc{($\oc$D)}}
\UnaryInfC{$\oc (X \invamp X) \vdash X, X$}
\doubleLine
\LeftLabel{\textsc{($\wn$D$^\ast$, XR)}}
\UnaryInfC{$\oc (X \invamp X) \vdash \wn X, \wn X$}
\LeftLabel{\textsc{($\wn$C)}}
\UnaryInfC{$\oc (X \invamp X) \vdash \wn X$}
\LeftLabel{\textsc{($\oc$R$^{\oc\wn}$)}}
\UnaryInfC{$\oc (X \invamp X) \vdash \oc \wn X$}
\AxiomC{}
\LeftLabel{\textsc{(Id)}}
\UnaryInfC{$X \vdash X$}
\AxiomC{}
\LeftLabel{\textsc{(Id)}}
\UnaryInfC{$X \vdash X$}
\LeftLabel{\textsc{($\otimes$R)}}
\BinaryInfC{$X, X \vdash X \otimes X$}
\LeftLabel{\textsc{($\wn$D)}}
\UnaryInfC{$X, X \vdash \wn (X \otimes X)$}
\LeftLabel{\textsc{($\oc$D$^\ast$, XL)}}
\doubleLine
\UnaryInfC{$\oc X, \oc X \vdash \wn (X \otimes X)$}
\LeftLabel{\textsc{($\oc$C)}}
\UnaryInfC{$\oc X \vdash \wn (X \otimes X)$}
\LeftLabel{\textsc{($\oc \wn$L$^{\oc\wn}$)}}
\UnaryInfC{$\oc \wn X \vdash \wn (X \otimes X)$}
\LeftLabel{\textsc{(Cut)}}
\BinaryInfC{$\oc (X \invamp X) \vdash \wn (X \otimes X)$}
\DisplayProof
\end{mathpar}
where the superscript $(\_)^\ast$ denotes a finite number of applications of a rule, and the double line abbreviates multiple applications of the rules indicated on the left-hand side; however, it is not without Cut.
Note that the rule $\oc \wn$L$^{\oc\wn}$ plays a key role in this formal proof.
Note also that there is another formal proof of the sequent $\oc (X \invamp X) \vdash \wn (X \otimes X)$, which uses the rules $\wn\oc$R$^{\oc\wn}$ and $\wn$L$^{\oc\wn}$ instead of those $\oc$R$^{\oc\wn}$ and $\oc\wn$L$^{\oc\wn}$, respectively.
This problem on cut-elimination is technically challenging, but we overcome it by restricting the use of the rules $\wn \oc$R$^{\oc\wn}$ and $\oc \wn$L$^{\oc\wn}$ in \textsf{ILC}$_\iota$ in \S\ref{ConservativeTranslations}. 
To arrive at this solution, however, we first need to see in \S\ref{ConservativeExtensionOfIntuitionisticLogic}--\ref{SubsectionCommutativeUnityOfLogic} how \textsf{ILC}$_\iota$ achieves the theorem articulated in \S\ref{MainResults}.

An important corollary of Theorem~\ref{ThmCutEliminationForILC} is, as announced previously, that \textsf{ILC$_{(\iota)}$} is indeed a \emph{conservative} extension of \textsf{LLJ}:
\begin{corollary}[\textsf{ILC$_{\boldsymbol{(\iota)}}$} as a conservative extension of \textsf{LLJ}]
\label{ILCAsConservativeExtensionOfLLJ}
\textsf{ILC} (resp. \textsf{ILC$_\iota$}) is a conservative extension of \textsf{LLJ}, i.e., every provable sequent in \textsf{LLJ} is also provable in \textsf{ILC} (resp. \textsf{ILC$_\iota$}), and if a sequent $\Delta \vdash \Gamma$ is provable in \textsf{ILC} (resp. \textsf{ILC$_\iota$}), where only formulas in ILL occur in $\Delta, \Gamma$, then $\Gamma = \boldsymbol{\epsilon}$ or $\Gamma = B$ for some formula $B$, and the sequent $\Delta \vdash \Gamma$ is also provable in \textsf{LLJ}.
\end{corollary}
\begin{remark}
The conservativity of \textsf{ILC}$_{(\iota)}$ over \textsf{LLJ} does not merely follow from \cite[Proposition~3.8]{schellinx1991some} because \textsf{ILC}$_{(\iota)}$ allows up-linear negation $\neg$ to occur in a formal proof of the sequent $\Delta \vdash \Gamma$ in the corollary except the root $\Delta \vdash \Gamma$. 
Also, the conservativity would fail if we included one $1$ or zero $0$ in ILL.
\end{remark}
\begin{proof}
We show the corollary only for \textsf{ILC} since why-not $\wn$ does not occur in formulas in ILL, and thus it is not hard to extend the argument to \textsf{ILC$_\iota$}; see \S\ref{ProofOnMu}.
First, every provable sequent in \textsf{LLJ} is clearly provable in \textsf{ILC}, where $\rightarrowtriangle$L is replaced with the pair of $\neg$L and $\invamp$L, and $\rightarrowtriangle$R with that of $\neg$R and $\invamp$R.

Next, to show the conservativity, assume that a sequent $\Delta \vdash \Gamma$ has a formal proof $p$ in \textsf{ILC}, and only formulas in ILL occur in $\Delta, \Gamma$.
By Theorem~\ref{ThmCutEliminationForILC}, there is a formal proof $p'$ of the sequent $\Delta \vdash \Gamma$ in \textsf{ILC} without Cut.
We have to show that $\Gamma$ is $\boldsymbol{\epsilon}$ or $B$ for some formula $B$.
For this task, we claim that
\begin{quote}
Each sequent occurring in $p'$ is of the form $\Delta', \neg A \vdash \neg \Gamma'$ or $\Delta' \vdash A, \neg \Gamma'$ modulo XL and XR, where only formulas in ILL occur in $\Delta', A, \Gamma'$.
\end{quote}

In fact, because \emph{only $\invamp$L and $\invamp$R} can delete $\neg$ in $p'$ thanks to the equation $\neg A \invamp B = A \rightarrowtriangle B$, the claim is easily proved by induction on cut-free proofs of the sequent $\Delta \vdash \Gamma$ (n.b., it is crucial here that why-not $\wn$ does not occur in formulas in ILL), for which the \emph{subformula property} \cite{troelstra2000basic} of \textsf{ILC} implied by Theorem~\ref{ThmCutEliminationForILC} plays a crucial role. 
Then, because no $\neg$ occurs in the sequent $\Delta \vdash \Gamma$, the claim particularly implies that $\Gamma$ is $\boldsymbol{\epsilon}$ or $B$ for some formula $B$.

It remains to show that $\Delta \vdash \Gamma$ is provable in \textsf{LLJ}.
For this task, we have to handle non-cut rules that may decrease the number of elements on the right-hand side of sequents as they allow $p'$ to have sequents not provable in \textsf{LLJ} in the middle of the formal proof.
Among non-cut rules of \textsf{ILC}, only $\invamp$R and $\neg$L may decrease the number of elements on the right-hand side of sequents. 

Now, consider the \emph{last} application of $\invamp$R or $\neg$L in $p'$.
Let us replace this last application in $p'$ with a derived rule in \textsf{LLJ} as follows.
First, by the above claim, the last application of $\invamp$R or $\neg$L in $p'$ must be either of the following:
\begin{mathpar}
\AxiomC{$\vdots$}
\LeftLabel{\textsc{($\neg$R)}}
\UnaryInfC{$\vdots$}
\UnaryInfC{$\Theta \vdash \neg D, E$}
\LeftLabel{\textsc{($\invamp$R)}}
\UnaryInfC{$\Theta \vdash D \rightarrowtriangle E$}
\DisplayProof
\and
\AxiomC{$\vdots$}
\LeftLabel{\textsc{($\neg$L)}}
\UnaryInfC{$\vdots$}
\UnaryInfC{$\Theta, \neg A \vdash$}
\AxiomC{$\Xi, B \vdash [C]$}
\LeftLabel{\textsc{($\invamp$L)}}
\BinaryInfC{$\Theta, \Xi, A \rightarrowtriangle B \vdash [C]$}
\DisplayProof
\end{mathpar}
where $[C]$ denotes the empty sequence $\boldsymbol{\epsilon}$ or a singleton sequence $C$, and the last application of $\neg$R (resp. $\neg$L) in $p'$ generates the element $\neg$D (resp. $\neg A$).

For the left case, we delay the last application of $\neg$R that generates $\neg D$ until right before the application of $\invamp$R, which slightly modifies $p'$.
Similarly for the right case, we delay the last application of $\neg$L that generates $\neg A$ until right before the application of $\invamp$L.
As a result, we can focus on the following derived rules in $p'$ (where as mentioned above $p'$ may be slightly modified):
\begin{mathpar}
\AxiomC{$\Theta, D \vdash E$}
\LeftLabel{\textsc{($\neg$R)}}
\UnaryInfC{$\Theta \vdash \neg D, E$}
\LeftLabel{\textsc{($\invamp$R)}}
\UnaryInfC{$\Theta \vdash D \rightarrowtriangle E$}
\DisplayProof
\and
\AxiomC{$\Theta \vdash A$}
\LeftLabel{\textsc{($\neg$L)}}
\UnaryInfC{$\Theta, \neg A \vdash$}
\AxiomC{$\Xi, B \vdash [C]$}
\LeftLabel{\textsc{($\invamp$L)}}
\BinaryInfC{$\Theta, \Xi, A \rightarrowtriangle B \vdash [C]$}
\DisplayProof
\end{mathpar}
Then, the left (resp. right) one can be replaced with $\rightarrowtriangle$R (resp. $\rightarrowtriangle$L) of \textsf{LLJ}.

In this way, we inductively replace each application of $\invamp$R and $\neg$L in $p'$ (in the order from the last to the first) with a rule in \textsf{LLJ}, obtaining a formal proof $p''$ of $\Delta \vdash \Gamma$ out of $p'$.
It is easy to see that the sequents in $p''$ are all intuitionistic (i.e., the length of the right-hand side is at most one) and $\neg$-free. 
Thus, each rule in $p''$ is clearly derivable in \textsf{LLJ}, proving the conservativity. 
\end{proof}

Hence, ILL$^{\text{e}}_{(\iota)}$ is \emph{intuitionistic} in the conventional sense even though \textsf{ILC$_{(\iota)}$} allows more than one formula to occur on the right-hand side of sequents.



\subsection{Conservative extension of intuitionistic logic}
\label{ConservativeExtensionOfIntuitionisticLogic}
In this section, we present a translation $\mathscr{T}_{\oc\wn}$ of \textsf{LK} into $(\textsf{ILC$_\iota$})_{\oc\wn}$ (Definition~\ref{Defunlinearisation}) out of more primitive ones $\mathscr{T}_{\oc}$ and $\mathscr{T}_{\wn}$ by $\mathscr{T}_{\oc\wn} \colonequals \mathscr{T}_{\oc} \circ \mathscr{T}_{\wn}$.
For this decomposition of $\mathscr{T}_{\oc\wn}$, we need an intermediate logic between ILL$^{\text{e}}_\iota$ and CL.
For our unity of logic, this intermediate logic has to be a conservative extension of IL.\footnote{What follows in the next few pages is on how we have arrived at Definitions~\ref{DefFormulasOfILE}--\ref{DefINC_Mu}, and the reader who is not interested in this point can jump to these definitions immediately.}

For convenience, let us call such an extension of IL and a sequent calculus for it respectively \emph{intuitionistic logic extended (IL$^{\text{e}}$)} and \emph{\textsf{INC}} even before defining them.
Our initial idea on defining them is to apply unlinearisation $(\_)_\oc$ to \textsf{ILC$_\iota$} and regard sequents $\oc \Delta \vdash \Gamma$ in \textsf{ILC$_\iota$} as those $\Delta \vdash \Gamma$ in \textsf{INC} under the translation $\mathrm{ff} \colonequals \oc \bot$, $A \vee B \colonequals \oc A \oplus \oc B$ and $A \Rightarrow B \colonequals \oc A \rightarrowtriangle B$. 
This translation of non-linear disjunction $\vee$ and intuitionistic implication $\Rightarrow$ follows Girard's translation of IL into ILL \cite[\S 5.1]{girard1987linear}, and our choice on the translation of false $\mathrm{ff}$ is for the \emph{intuitionisity} of \textsf{INC} in the conventional sense as explained below. 

The weakening $\wn$W (resp. introduction of false $\mathrm{ff}$ on the right-hand side of sequents by the rules $\bot$R and $\oc$R$^{\oc\wn}$) in $(\textsf{ILC$_\iota$})_\oc$ necessitate why-not $\wn$ on the main (resp. side) formula(s) on the right-hand side of sequents.
It then follows, under the aforementioned translation of formulas, that $(\textsf{ILC$_\iota$})_\oc$ is \emph{intuitionistic} in the conventional sense: If $\oc \Delta \vdash \Gamma$ is derivable in $(\textsf{ILC$_\iota$})_\oc$, and only formulas in IL occur in $\Delta, \Gamma$, then $\Gamma$ consists of \emph{at most one} element.\footnote{The same argument plays a crucial role in the proof of Corollary~\ref{CorMainCorollary} as well.}

Therefore, so far $(\textsf{ILC$_\iota$})_\oc$ seems to be a good candidate for \textsf{INC}.
However, $(\textsf{ILC$_\iota$})_\oc$ is in some sense \emph{too unrestricted}.
For instance, the rule
\begin{mathpar}
\AxiomC{$\oc \Delta \vdash B_i, \Gamma$}
\LeftLabel{\textsc{($\vee$R)}}
\RightLabel{\textsc{($i \in \overline{2}$)}}
\UnaryInfC{$\oc \Delta \vdash B_1 \vee B_2, \Gamma$}
\DisplayProof
\end{mathpar}
is derivable in $(\textsf{ILC$_\iota$})_\oc$ only when $\Gamma$ is of the form $\Gamma = \wn \Gamma'$ so that
\begin{mathpar}
\AxiomC{$\oc \Delta \vdash B_i, \wn \Gamma'$}
\RightLabel{\textsc{($i \in \overline{2}$)}}
\LeftLabel{\textsc{($\oc$R$^{\oc\wn}$)}}
\UnaryInfC{$\oc \Delta \vdash \oc B_i, \wn \Gamma'$}
\LeftLabel{\textsc{($\oplus$R)}}
\UnaryInfC{$\oc \Delta \vdash \oc B_1 \oplus \oc B_2, \wn \Gamma'$}
\DisplayProof
\end{mathpar}
where as long as IL is concerned we can assume of-course $\oc$ on the left-hand side $\Delta$ of sequents in $(\textsf{ILC$_\iota$})_\oc$ since we can advance the applications of the rule $\oc$D in formal proofs (for which the insertion of $\oc$ in the translation of $\vee$ and $\Rightarrow$ given above is crucial).
We cannot simply require $\Gamma$ to be empty since, for the translation of CL into IL$^{\text{e}}$ by classicalisation $(\_)_\wn$ given later, we must allow the side formulas of the rule $\vee$R and other rules of \textsf{INC} to be \emph{nonempty}. 

This problem suggests us to restrict the right-hand side of sequents in $(\textsf{ILC$_\iota$})_\oc$ to those of the form $[B], \wn \Gamma$, where $[B]$ is the empty sequence $\boldsymbol{\epsilon}$ or a singleton sequence $B$, and require that the active formula(s) of each logical right rule in $(\textsf{ILC$_\iota$})_\oc$ must be the distinguished one(s) $B$.\footnote{Additives of this form are called \emph{tamed} in \cite[\S 2.3]{danos1997new}. Also, the sequence $[B], \wn \Gamma$ on the right-hand side of a sequent in (\textsf{ILC}$_\iota$)$_\oc$ corresponds to the one $\Gamma ; [B]$ in Girard's \textsf{LU} \cite{girard1993unity}.} 

\begin{remark}
\label{RemarkOnIntuitionisity}
The sequent calculus $(\textsf{ILC$_\iota$})_\oc$ is already intuitionistic as articulated above, and hence this restriction on the right-hand side of sequents is \emph{not} for obtaining intuitionisity.
In other words, the restriction is \emph{automatically satisfied} by $(\textsf{ILC$_\iota$})_\oc$ as long as we focus on the formulas of IL$^{\text{e}}$. 
Eventually, our main result (Corollary~\ref{CorMainCorollary}), specifically the \emph{conservativity} of the translation $\mathscr{T}_\oc : \textsf{INC}_\rho \rightarrow (\textsf{ILC}_\rho)_\oc$, makes this point explicit. 
This result also explains the general mechanism underlying the \emph{ad-hoc} intuitionistic restriction for obtaining IL out of CL (Definition~\ref{DefLJ}); see the paragraph right after Corollary~\ref{CorMainCorollary}.
\end{remark}

\begin{convention}
We write $\Gamma$ in place of $[B], \wn \Gamma$ for most of the left rules of $(\textsf{ILC$_\iota$})_\oc$ as they do not break down the required form of the right-hand side of sequents.
\end{convention}

Moreover, as a result of this restriction on the right-hand side of sequents in $(\textsf{ILC$_\iota$})_\oc$, some logical constants and connectives in ILL$^{\text{e}}_{\iota}$ become \emph{redundant}, which explains why those constants and connectives do not appear in IL$^{\text{(e)}}$.

For instance, we may substitute tensor $\otimes$ with with $\&$ inside $(\textsf{ILC$_\iota$})_\oc$ because the derived rule 
\begin{mathpar}
\AxiomC{$\oc \Delta, \oc A_1, \oc A_2 \vdash \Gamma$}
\LeftLabel{\textsc{($\otimes$L)}}
\UnaryInfC{$\oc \Delta, \oc A_1 \otimes \oc A_2 \vdash \Gamma$}
\LeftLabel{\textsc{($\oc$D)}}
\UnaryInfC{$\oc \Delta, \oc (\oc A_1 \otimes \oc A_2) \vdash \Gamma$}
\DisplayProof
\end{mathpar}
can be simulated by the derived rule
\begin{mathpar}
\AxiomC{}
\LeftLabel{\textsc{(Id)}}
\UnaryInfC{$A_1 \vdash A_1$}
\LeftLabel{\textsc{($\oc$D)}}
\UnaryInfC{$\oc A_1 \vdash A_1$}
\LeftLabel{\textsc{($\&$L)}}
\UnaryInfC{$\oc A_1 \mathbin{\&} \oc A_2 \vdash A_1$}
\LeftLabel{\textsc{($\oc$D)}}
\UnaryInfC{$\oc (\oc A_1 \mathbin{\&} \oc A_2) \vdash A_1$}
\LeftLabel{\textsc{($\oc$R$^{\oc\wn}$)}}
\UnaryInfC{$\oc (\oc A_1 \mathbin{\&} \oc A_2) \vdash \oc A_1$}
\AxiomC{}
\LeftLabel{\textsc{(Id)}}
\UnaryInfC{$A_2 \vdash A_2$}
\LeftLabel{\textsc{($\oc$D)}}
\UnaryInfC{$\oc A_2 \vdash A_2$}
\LeftLabel{\textsc{($\&$L)}}
\UnaryInfC{$\oc A_1 \mathbin{\&} \oc A_2 \vdash A_2$}
\LeftLabel{\textsc{($\oc$D)}}
\UnaryInfC{$\oc (\oc A_1 \mathbin{\&} \oc A_2) \vdash A_2$}
\LeftLabel{\textsc{($\oc$R$^{\oc\wn}$)}}
\UnaryInfC{$\oc (\oc A_1 \mathbin{\&} \oc A_2) \vdash \oc A_2$}
\AxiomC{$\oc \Delta, \oc A_1, \oc A_2 \vdash \Gamma$}
\LeftLabel{\textsc{(Cut)}}
\BinaryInfC{$\oc (\oc A_1 \mathbin{\&} \oc A_2), \oc \Delta, \oc A_1 \vdash \Gamma$}
\LeftLabel{\textsc{(Cut)}}
\BinaryInfC{$\oc (\oc A_1 \mathbin{\&} \oc A_2), \oc (\oc A_1 \mathbin{\&} \oc A_2), \oc \Delta \vdash \Gamma$}
\LeftLabel{\textsc{(XL$^\ast$)}}
\doubleLine
\UnaryInfC{$\oc \Delta, \oc (\oc A_1 \mathbin{\&} \oc A_2), \oc (\oc A_1 \mathbin{\&} \oc A_2) \vdash \Gamma$}
\LeftLabel{\textsc{($\oc$C)}}
\UnaryInfC{$\oc \Delta, \oc (\oc A_1 \mathbin{\&} \oc A_2) \vdash \Gamma$}
\DisplayProof
\end{mathpar}
and the rule 
\begin{mathpar}
\AxiomC{$\oc \Delta_1 \vdash B_1, \wn \Gamma_1$}
\AxiomC{$\oc \Delta_2 \vdash B_2, \wn \Gamma_2$}
\LeftLabel{\textsc{($\otimes$R)}}
\BinaryInfC{$\oc \Delta_1, \oc \Delta_2 \vdash B_1 \otimes B_2, \wn \Gamma_1, \wn \Gamma_2$}
\DisplayProof
\end{mathpar}
can be simulated by the derived rule
\begin{mathpar}
\AxiomC{$\oc \Delta_1 \vdash B_1, \wn \Gamma_1$}
\LeftLabel{\textsc{($\oc$W$^\ast$)}}
\doubleLine
\UnaryInfC{$\oc \Delta_1, \oc \Delta_2 \vdash B_1, \wn \Gamma_1$}
\LeftLabel{\textsc{($\wn$W$^\ast$)}}
\doubleLine
\UnaryInfC{$\oc \Delta_1, \oc \Delta_2 \vdash \wn \Gamma_2, B_1, \wn \Gamma_1$}
\LeftLabel{\textsc{(XR$^\ast$)}}
\doubleLine
\UnaryInfC{$\oc \Delta_1, \oc \Delta_2 \vdash B_1, \wn \Gamma_1, \wn \Gamma_2$}
\AxiomC{$\oc \Delta_2 \vdash B_2, \wn \Gamma_2$}
\LeftLabel{\textsc{($\oc$W$^\ast$)}}
\doubleLine
\UnaryInfC{$\oc \Delta_2, \oc \Delta_1 \vdash B_1, \wn \Gamma_2$}
\LeftLabel{\textsc{(XL$^\ast$)}}
\doubleLine
\UnaryInfC{$\oc \Delta_1, \oc \Delta_2 \vdash B_2, \wn \Gamma_2$}
\LeftLabel{\textsc{($\wn$W$^\ast$)}}
\doubleLine
\UnaryInfC{$\oc \Delta_1, \oc \Delta_2 \vdash \wn \Gamma_1, B_2, \wn \Gamma_2$}
\LeftLabel{\textsc{(XR$^\ast$)}}
\doubleLine
\UnaryInfC{$\oc \Delta_1, \oc \Delta_2 \vdash B_2, \wn \Gamma_1, \wn \Gamma_2$}
\LeftLabel{\textsc{($\&$R)}}
\BinaryInfC{$\oc \Delta_1, \oc \Delta_2 \vdash B_1 \mathbin{\&} B_2, \wn \Gamma_1, \wn \Gamma_2$}
\DisplayProof
\end{mathpar}

On the other hand, the matter on par $\invamp$ is slightly involved.
First, since of-course $\oc$ is placed on all elements on the left-hand side of the conclusion of each formal proof in $(\textsf{ILC$_\iota$})_\oc$, we have to restrict par $\invamp$ to the form $\oc A \invamp \oc B$ for the left rule as in the case of the modification of plus $\oplus$ into disjunction $\vee$.
In addition, the restriction on the right-hand side of sequents to those $[B], \wn \Gamma$ requires us to place why-not $\wn$ at least one of the two components of $\invamp$ (i.e., $\wn A \invamp B$ or $A \invamp \wn B$) for the right rule.
Then, one reasonable solution for these two problems is to restrict par $\invamp$ to the form $\wn \oc A \invamp \wn \oc B$, so that we have
\begin{mathpar}
\begin{small}
\AxiomC{$\oc \Delta_1, \oc A_1 \vdash \wn \Gamma_1$}
\LeftLabel{\textsc{($\wn$L$^{\oc\wn}$)}}
\UnaryInfC{$\oc \Delta_1, \wn \oc A_1 \vdash \wn \Gamma_1$}
		\AxiomC{$\oc \Delta_2, \oc A_2 \vdash \wn \Gamma_2$}
		\LeftLabel{\textsc{($\wn$L$^{\oc\wn}$)}}
		\UnaryInfC{$\oc \Delta_2, \wn \oc A_2 \vdash \wn \Gamma_2$}
	\LeftLabel{\textsc{($\invamp$L)}}
\BinaryInfC{$\oc \Delta_1, \oc \Delta_2, \wn \oc A_1 \invamp \wn \oc A_2 \vdash \wn \Gamma_1, \wn \Gamma_2$}
\LeftLabel{\textsc{($\oc$D)}}
\UnaryInfC{$\oc \Delta_1, \oc \Delta_2, \oc (\wn \oc A_1 \invamp \wn \oc A_2) \vdash \wn \Gamma_1, \wn \Gamma_2$}
\DisplayProof
\and
\AxiomC{$\oc \Delta \vdash \wn B_1, \wn B_2, \wn \Gamma$}
\LeftLabel{\textsc{($\wn\oc$R$^{\oc\wn}$)}}
\UnaryInfC{$\oc \Delta \vdash \wn \oc B_1, \wn B_2, \wn \Gamma$}
\LeftLabel{\textsc{(XR)}}
\UnaryInfC{$\oc \Delta \vdash \wn B_2, \wn \oc B_1, \wn \Gamma$}
\LeftLabel{\textsc{($\wn\oc$R$^{\oc\wn}$)}}
\UnaryInfC{$\oc \Delta \vdash \wn \oc B_2, \wn \oc B_1, \wn \Gamma$}
\LeftLabel{\textsc{(XR)}}
\UnaryInfC{$\oc \Delta \vdash \wn \oc B_1, \wn \oc B_2, \wn \Gamma$}
	\LeftLabel{\textsc{($\invamp$R)}}
\UnaryInfC{$\oc \Delta \vdash \wn \oc B_1 \invamp \wn \oc B_2, \wn \Gamma$}
\DisplayProof
\end{small}
\end{mathpar}
where why-not $\wn$ on $\Gamma_1$, $\Gamma_2$ and $\Gamma$ is crucial.
We can substitute this restricted par $\wn \oc A \invamp \wn \oc B$ with the formula $\wn (A \vee B) = \wn (\oc A \oplus \oc B)$ and recover these restricted rules on par $\invamp$ inside $(\textsf{ILC$_\iota$})_\oc$ similarly to the substitution of tensor $\otimes$ with with $\&$ demonstrated above; we leave the details to the reader.

Next, we consider up-linear implication $\rightarrowtriangle$.
Again, since of-course $\oc$ is placed on all elements on the left-hand side of conclusions in $(\textsf{ILC$_\iota$})_\oc$, we have to replace it with intuitionistic implication $\Rightarrow$.
Accordingly, we also have to replace up-linear negation $\neg A = A \rightarrowtriangle \bot$ with \emph{intuitionistic negation} $A^\star \colonequals A \Rightarrow \mathrm{ff}$.

At this point, let us consider one $1$.
However, the rule $1$R does not keep our restriction on the right-hand side of sequents.
To keep the restriction, we modify one $1$ into the formula $\oc 1$ so that its introduction by the rules $1$R and $\oc$R$^{\oc\wn}$ necessitates why-not $\wn$ on the side formulas on the right-hand side.
Then, we may substitute this formula $\oc 1$ with top $\top$ inside $(\textsf{ILC$_\iota$})_\oc$ as the derived rule 
\begin{mathpar}
\AxiomC{}
\LeftLabel{\textsc{($1$R)}}
\UnaryInfC{$\oc \Delta \vdash 1, \wn \Gamma$}
\LeftLabel{\textsc{($\oc$R$^{\oc\wn}$)}}
\UnaryInfC{$\oc \Delta \vdash \oc 1, \wn \Gamma$}
\DisplayProof
\end{mathpar}
can be simulated by the derived rule
\begin{mathpar}
\AxiomC{}
\LeftLabel{\textsc{($\top$R)}}
\UnaryInfC{$\vdash \top$}
\LeftLabel{\textsc{($\oc$W$^\ast$)}}
\doubleLine
\UnaryInfC{$\oc \Delta \vdash \top$}
\LeftLabel{\textsc{($\wn$W$^\ast$, XR$^\ast$)}}
\doubleLine
\UnaryInfC{$\oc \Delta \vdash \top, \wn \Gamma$}
\DisplayProof
\end{mathpar}

Similarly, the rule $0$L on zero $0$ does not follow our restriction on the right-hand side of sequents.
Thus, to keep the restriction again, we modify zero $0$ into the formula $\wn 0$ so that its introduction by the rules $0$L and $\wn$L$^{\oc\wn}$ necessitates why-not $\wn$ on the right-hand side of sequents.
Then, we may substitute this formula $\wn 0$ with false $\mathrm{ff} \colonequals \oc \bot$ inside $(\textsf{ILC$_\iota$})_\oc$ because the derived rule 
\begin{mathpar}
\AxiomC{}
\LeftLabel{\textsc{($0$L)}}
\UnaryInfC{$\oc \Delta, 0 \vdash \wn \Gamma$}
\LeftLabel{\textsc{($\wn$L$^{\oc\wn}$)}}
\UnaryInfC{$\oc \Delta, \wn 0 \vdash \wn \Gamma$}
\LeftLabel{\textsc{($\oc$D)}}
\UnaryInfC{$\oc \Delta, \oc \wn 0 \vdash \wn \Gamma$}
\DisplayProof
\end{mathpar}
can be simulated by the derived rule
\begin{mathpar}
\AxiomC{}
\LeftLabel{\textsc{($\bot$L)}}
\UnaryInfC{$\bot \vdash $}
\LeftLabel{\textsc{($\oc$D$^\ast$)}}
\doubleLine
\UnaryInfC{$\oc \mathrm{ff} \vdash $}
\LeftLabel{\textsc{($\oc$W$^\ast$, XL$^\ast$)}}
\doubleLine
\UnaryInfC{$\oc \Delta, \oc \mathrm{ff} \vdash$}
\LeftLabel{\textsc{($\wn$W$^\ast$)}}
\doubleLine
\UnaryInfC{$\oc \Delta, \oc \mathrm{ff} \vdash \wn \Gamma$}
\DisplayProof
\end{mathpar}

Further, we may dispense with of-course $\oc$ as long as IL is concerned since we can advance the applications of the rule $\oc$D in formal proofs in $(\textsf{ILC$_\iota$})_\oc$ so that of-course $\oc$ is on the left-hand side of sequents in $(\textsf{ILC$_\iota$})_\oc$ almost by default.
That is, we can focus on sequents of the form $\oc \Delta \vdash [B], \wn \Gamma$ in $(\textsf{ILC$_\iota$})_\oc$ and regard them as those $\Delta \vdash [B], \wn \Gamma$ in \textsf{INC}.
This \emph{implicit} placement of of-course $\oc$ on the left-hand side of sequents is what we propose as the definition of \emph{non-linearity} (\S\ref{MainResults}).
Dually, the \emph{implicit} placement of why-not $\wn$ on the right-hand side of sequents is what we propose as the definition of \emph{classicality} (\S\ref{MainResults}); see \S\ref{ClassicalLinearLogicNegative}.

The last missing piece is a \emph{cut-rule} for \textsf{INC}. 
For this task, we adopt the rule Cut of \textsf{LK} yet translated appropriately as in the proof of Theorem~\ref{LemTranslationOfLKIntoILK_mu}. 

Consequently, we define IL$^{\text{e}}$ and \textsf{INC} as follows:

\begin{definition}[Formulas of IL$^{\text{e}}$]
\label{DefFormulasOfILE}
Formulas $A, B$ of \emph{intuitionistic logic extended (IL$^{\text{e}}$)} are defined by
\begin{mathpar}
A, B \colonequals X \mid \top \mid \mathrm{ff} \mid A \mathbin{\&} B \mid A \vee B \mid A \Rightarrow B \mid \wn A
\end{mathpar}
where $X$ ranges over propositional variables, and we define $A^\star \colonequals A \Rightarrow \mathrm{ff}$.
\end{definition}

\begin{definition}[\textsf{INC} for IL$^{\text{e}}$]
\label{DefINC_Mu}
The sequent calculus \emph{\textsf{INC}} for IL$^{\text{e}}$ consists of the axioms and the rules displayed in Figure~\ref{FigILK_mu}. 
\begin{figure}
\begin{center}
\begin{mathpar}
\AxiomC{$\Delta, A, A', \Delta' \vdash \Gamma$}
\LeftLabel{\textsc{(XL)}}
\UnaryInfC{$\Delta, A', A, \Delta' \vdash \Gamma$}
\DisplayProof \and
\AxiomC{$\Delta \vdash \Gamma, B, B', \Gamma'$}
\LeftLabel{\textsc{(XR)}}
\UnaryInfC{$\Delta \vdash \Gamma, B', B, \Gamma'$}
\DisplayProof \\
\AxiomC{$\Delta \vdash \Gamma$}
\LeftLabel{\textsc{(WL)}}
\UnaryInfC{$\Delta, A \vdash \Gamma$}
\DisplayProof \and
\AxiomC{$\Delta\vdash \Gamma$}
\LeftLabel{\textsc{($\wn$W)}}
\UnaryInfC{$\Delta \vdash \wn B, \Gamma$}
\DisplayProof \\
\AxiomC{$\Delta, A, A \vdash \Gamma$}
\LeftLabel{\textsc{(CL)}}
\UnaryInfC{$\Delta, A \vdash \Gamma$}
\DisplayProof \and
\AxiomC{$\Delta \vdash \wn B, \wn B, \Gamma$}
\LeftLabel{\textsc{($\wn$C)}}
\UnaryInfC{$\Delta \vdash \wn B, \Gamma$}
\DisplayProof \\
\AxiomC{$\Delta \vdash B, \Gamma$}
\LeftLabel{\textsc{($\wn$D)}}
\UnaryInfC{$\Delta \vdash \wn B, \Gamma$}
\DisplayProof 
\and
\AxiomC{$\Delta, A \vdash \wn \Gamma$}
\LeftLabel{\textsc{($\wn$L$^\wn$)}}
\UnaryInfC{$\Delta, \wn A \vdash \wn \Gamma$}
\DisplayProof \\
\AxiomC{}
\LeftLabel{\textsc{(Id)}}
\UnaryInfC{$A \vdash A$}
\DisplayProof \and
\AxiomC{$\Delta \vdash \wn B, \wn \Gamma$}
	\AxiomC{$\Delta', B \vdash \wn \Gamma'$}
	\LeftLabel{\textsc{(Cut$^\wn$)}}
\BinaryInfC{$\Delta, \Delta' \vdash \wn \Gamma, \wn \Gamma'$}
\DisplayProof 
\\
\AxiomC{$\Delta \vdash \Gamma$}
\LeftLabel{\textsc{($\top$L)}}
\UnaryInfC{$\Delta, \top \vdash \Gamma$}
\DisplayProof \and
\AxiomC{}
\LeftLabel{\textsc{($\top$R)}}
\UnaryInfC{$\vdash \top$}
\DisplayProof 
\and
\AxiomC{}
\LeftLabel{\textsc{($\mathrm{ff}$L)}}
\UnaryInfC{$\mathrm{ff} \vdash$}
\DisplayProof \and
\AxiomC{$\Delta \vdash \wn \Gamma$}
\LeftLabel{\textsc{($\mathrm{ff}$R$^\wn$)}}
\UnaryInfC{$\Delta \vdash \mathrm{ff}, \wn \Gamma$}
\DisplayProof \\
\AxiomC{$\Delta, A_i \vdash \Gamma$}
	\LeftLabel{\textsc{($\&$L)}}
	\RightLabel{($i \in \overline{2}$)}
\UnaryInfC{$\Delta, A_1 \mathbin{\&} A_2 \vdash \Gamma$}
\DisplayProof \and
\AxiomC{$\Delta \vdash B_1, \wn \Gamma$}
		\AxiomC{$\Delta \vdash B_2, \wn \Gamma$}
	\LeftLabel{\textsc{($\&$R$^\wn$)}}
\BinaryInfC{$\Delta \vdash B_1 \mathbin{\&} B_2, \wn \Gamma$}
\DisplayProof \\
\AxiomC{$\Delta, A_1 \vdash \Gamma$}
		\AxiomC{$\Delta, A_2 \vdash \Gamma$}
	\LeftLabel{\textsc{($\vee$L)}}
\BinaryInfC{$\Delta, A_1 \vee A_2 \vdash \Gamma$}
\DisplayProof \and
\AxiomC{$\Delta \vdash B_i, \wn \Gamma$}
	\LeftLabel{\textsc{($\vee$R$^\wn$)}}
	\RightLabel{($i \in \overline{2}$)}
\UnaryInfC{$\Delta \vdash B_1 \vee B_2, \wn \Gamma$}
\DisplayProof \\
\AxiomC{$\Delta, B \vdash \Gamma$}
\AxiomC{$\Theta \vdash A, \wn \Xi$}
\LeftLabel{\textsc{($\Rightarrow$L$^\wn$)}}
\BinaryInfC{$\Delta, \Theta, A \Rightarrow B \vdash \Gamma, \wn \Xi$}
\DisplayProof \and
\AxiomC{$\Delta, A \vdash B, \wn \Gamma$}
\LeftLabel{\textsc{($\Rightarrow$R$^\wn$)}}
\UnaryInfC{$\Delta \vdash A \Rightarrow B, \wn \Gamma$}
\DisplayProof
\end{mathpar}
\end{center}
\caption{Sequent calculus \textsf{INC} for IL$^{\text{e}}$}
\label{FigILK_mu}
\end{figure}
\end{definition}

As expected, the cut-elimination theorem holds for \textsf{INC}:
\begin{lemma}[Cut-elimination for \textsf{INC}]
\label{LemCutEliminationForINC_mu}
Given a formal proof of a sequent in \textsf{INC}, there is a formal proof of this sequent in \textsf{INC} without Cut$^\wn$.
\end{lemma}
\begin{proof}
Similar to the case of Theorem~\ref{ThmCutEliminationForILC}.
\end{proof}

The following corollary of Lemma~\ref{LemCutEliminationForINC_mu} shows that \textsf{INC} is a \emph{conservative} extension of \textsf{LJ}, and therefore IL$^{\text{e}}$ is \emph{intuitionistic} in the conventional sense:
\begin{corollary}[\textsf{INC} as a conservative extension of \textsf{LJ}]
\label{CorINCAsAConservativeExtensionOfLJ}
\textsf{INC} is a conservative extension of \textsf{LJ}, i.e., each provable sequent in \textsf{LJ} is provable in \textsf{INC}, and if a sequent $\Delta \vdash \Gamma$ is provable in \textsf{INC}, where only formulas in IL occur in $\Delta, \Gamma$, then $\Gamma = \boldsymbol{\epsilon}$ or $\Gamma = B$ for some formula $B$, and the sequent $\Delta \vdash [B]$ is provable in \textsf{LJ}, where $[B]$ is the empty sequence $\boldsymbol{\epsilon}$ or a singleton sequence $B$.
\end{corollary}
\begin{proof}
First, every non-cut rule of \textsf{LJ} is also a rule of \textsf{INC}.
Therefore, by Theorem~\ref{ThmCutEliminationForLKAndLJ}, every sequent provable in \textsf{LJ} is clearly provable in \textsf{INC}.

Next, assume that there is a formal proof $p$ of a sequent $\Delta \vdash \Gamma$ in \textsf{INC}, and only formulas of IL occur in $\Delta, \Gamma$.
By Lemma~\ref{LemCutEliminationForINC_mu}, there is a cut-free formal proof $p'$ of $\Delta \vdash \Gamma$ in \textsf{INC}.
By induction on $p'$, we see that the right-hand side of each sequent in $p'$ is of the form $[B], \wn \Gamma'$.
Hence, in particular, $\Gamma = [B]$.
Moreover, by this form $[B], \wn \Gamma'$ on $p'$ and the subformula property of \textsf{INC} (by Lemma~\ref{LemCutEliminationForINC_mu}), we see by induction on $p'$ that $p'$ is a formal proof in \textsf{LJ}. 
\end{proof}

Let us then propose the translations $\textsf{LK} \stackrel{\mathscr{T}_{\wn}}{\rightarrow} \textsf{INC} \stackrel{\mathscr{T}_{\oc}}{\rightarrow} \textsf{ILC$_\iota$}$:

\begin{lemma}[Translation $\mathscr{T}_{\wn}$ of \textsf{LK} into $\textsf{INC}$]
\label{LemTranslationOfLKIntoILK_mu}
There is a map $\mathscr{T}_{\wn}$ that assigns, to each formal proof $p$ of a sequent $\Delta \vdash \Gamma$ in \textsf{LK}, a formal proof $\mathscr{T}_{\wn}(p)$ of the sequent $\mathscr{T}_{\wn}^\ast(\Delta) \vdash \wn \mathscr{T}_{\wn}^\ast(\Gamma)$ in \textsf{INC}, where 
\begin{mathpar}
\mathscr{T}_\wn(\mathrm{tt}) \colonequals \wn \top
\and
\mathscr{T}_\wn(\mathrm{ff}) \colonequals \mathrm{ff}
\and
\mathscr{T}_\wn(A \wedge B) \colonequals \wn \mathscr{T}_\wn(A) \mathbin{\&} \wn \mathscr{T}_\wn(B)
\and
\mathscr{T}_\wn (A \vee B) \colonequals \mathscr{T}_\wn (A) \vee \mathscr{T}_\wn(B)
\and
\mathscr{T}_\wn(A \Rrightarrow B) \colonequals \mathscr{T}_\wn(A) \Rightarrow \wn \mathscr{T}_\wn(B).
\end{mathpar}

Moreover, it is conservative: If $\Delta, \Gamma$ has only formulas of CL, and $\mathscr{T}_\wn^\ast(\Delta) \vdash \wn \mathscr{T}_\wn^\ast(\Gamma)$ is provable in \textsf{INC}, then $\Delta \vdash \Gamma$ is provable in \textsf{LK}.
\end{lemma}
\begin{proof}
We first translate the axioms and the rules of \textsf{LK} into derived ones in \textsf{INC}. 
Let us leave it to the reader to translate XL, XR, WL, WR, CL, CR, Id, $\mathrm{tt}$L, $\mathrm{tt}$R, $\mathrm{ff}$L and $\mathrm{ff}$R in \textsf{LK} into \textsf{INC} because it is straightforward. 

Cut in \textsf{LK} is translated into \textsf{INC} simply by
\begin{mathpar}
\AxiomC{$\Delta \vdash \wn B, \wn \Gamma$}
\AxiomC{$\Delta', B \vdash \wn \Gamma'$}
\LeftLabel{\textsc{(Cut$^\wn$)}}
\BinaryInfC{$\Delta, \Delta' \vdash \wn \Gamma, \wn \Gamma'$}
\DisplayProof
\end{mathpar}

$\wedge$L in \textsf{LK} is translated into \textsf{INC} by
\begin{mathpar}
\AxiomC{$\Delta, A_i \vdash \wn \Gamma$}
\RightLabel{($i \in \overline{2}$)}
\LeftLabel{\textsc{($\wn$L$^\wn$)}}
\UnaryInfC{$\Delta, \wn A_i \vdash \wn \Gamma$}
\LeftLabel{\textsc{($\&$L)}}
\UnaryInfC{$\Delta, \wn A_1 \mathbin{\&} \wn A_2 \vdash \wn \Gamma$}
\DisplayProof
\end{mathpar}
and $\wedge$R by
\begin{mathpar}
\AxiomC{$\Delta \vdash \wn B_1, \wn \Gamma$}
\AxiomC{$\Delta \vdash \wn B_2, \wn \Gamma$}
\LeftLabel{\textsc{($\&$R$^\wn$)}}
\BinaryInfC{$\Delta \vdash \wn B_1 \mathbin{\&} \wn B_2, \wn \Gamma$}
\LeftLabel{\textsc{($\wn$D)}}
\UnaryInfC{$\Delta \vdash \wn (\wn B_1 \mathbin{\&} \wn B_2), \wn \Gamma$}
\DisplayProof
\end{mathpar}

Dually, $\vee$L in \textsf{LK} is translated into \textsf{INC} simply by
\begin{mathpar}
\AxiomC{$\Delta, A_1 \vdash \wn \Gamma$}
\AxiomC{$\Delta, A_2 \vdash \wn \Gamma$}
\LeftLabel{\textsc{($\vee$L)}}
\BinaryInfC{$\Delta, A_1 \vee A_2 \vdash \wn \Gamma$}
\DisplayProof
\end{mathpar}
and $\vee$R by
\begin{mathpar}
\AxiomC{$\Delta \vdash \wn B_i, \wn \Gamma$}
\AxiomC{}
\LeftLabel{\textsc{(Id)}}
\RightLabel{($i \in \overline{2}$)}
\UnaryInfC{$B_i \vdash B_i$}
\LeftLabel{\textsc{($\vee$R$^\wn$)}}
\UnaryInfC{$B_i \vdash B_1 \vee B_2$}
\LeftLabel{\textsc{($\wn$D)}}
\UnaryInfC{$B_i \vdash \wn (B_1 \vee B_2)$}
\LeftLabel{\textsc{(Cut$^\wn$)}}
\BinaryInfC{$\Delta \vdash \wn \Gamma, \wn (B_1 \vee B_2)$}
\LeftLabel{\textsc{(XR$^\ast$)}}
\doubleLine
\UnaryInfC{$\Delta \vdash \wn (B_1 \vee B_2), \wn \Gamma$}
\DisplayProof
\end{mathpar}

Next, $\Rrightarrow$L in \textsf{LK} is translated into \textsf{INC} by
\begin{mathpar}
\AxiomC{}
\LeftLabel{\textsc{(Id)}}
\UnaryInfC{$\wn B \vdash \wn B$}
\AxiomC{}
\LeftLabel{\textsc{(Id)}}
\UnaryInfC{$A \vdash A$}
\LeftLabel{\textsc{($\Rightarrow$L$^\wn$)}}
\BinaryInfC{$A, A \Rightarrow \wn B \vdash \wn B$}
\LeftLabel{\textsc{(XL)}}
\UnaryInfC{$A \Rightarrow \wn B, A \vdash \wn B$}
\LeftLabel{\textsc{($\wn$L$^\wn$)}}
\UnaryInfC{$A \Rightarrow \wn B, \wn A \vdash \wn B$}
\LeftLabel{\textsc{($\Rightarrow$R$^\wn$)}}
\UnaryInfC{$A \Rightarrow \wn B \vdash \wn A \Rightarrow \wn B$}
\LeftLabel{\textsc{($\wn$D)}}
\UnaryInfC{$A \Rightarrow \wn B \vdash \wn (\wn A \Rightarrow \wn B)$}
\AxiomC{$\Delta, B \vdash \wn \Gamma$}
\LeftLabel{\textsc{($\wn$L$^\wn$)}}
\UnaryInfC{$\Delta, \wn B \vdash \wn \Gamma$}
\AxiomC{$\Delta \vdash \wn A, \wn \Gamma$}
\LeftLabel{\textsc{($\Rightarrow$L$^\wn$)}}
\BinaryInfC{$\Delta, \Delta, \wn A \Rightarrow \wn B \vdash \wn \Gamma, \wn \Gamma$}
\LeftLabel{\textsc{(CL$^\ast$)}}
\doubleLine
\UnaryInfC{$\Delta, \wn A \Rightarrow \wn B \vdash \wn \Gamma, \wn \Gamma$}
\LeftLabel{\textsc{($\wn$C$^\ast$)}}
\doubleLine
\UnaryInfC{$\Delta, \wn A \Rightarrow \wn B \vdash \wn \Gamma$}
\LeftLabel{\textsc{(Cut$^\wn$)}}
\BinaryInfC{$A \Rightarrow \wn B, \Delta \vdash \wn \Gamma$}
\LeftLabel{\textsc{(XL$^\ast$)}}
\doubleLine
\UnaryInfC{$\Delta, A \Rightarrow \wn B \vdash \wn \Gamma$}
\DisplayProof
\end{mathpar}
and $\Rrightarrow$R as:
\begin{mathpar}
\AxiomC{$\Delta, A \vdash \wn B, \wn \Gamma$}
\LeftLabel{\textsc{($\Rightarrow$R$^\wn$)}}
\UnaryInfC{$\Delta \vdash A \Rightarrow \wn B, \wn \Gamma$}
\LeftLabel{\textsc{($\wn$D)}}
\UnaryInfC{$\Delta \vdash \wn (A \Rightarrow \wn B), \wn \Gamma$}
\DisplayProof
\end{mathpar}

Given a formal proof $p$ in \textsf{LK}, we obtain the required formal proof $\mathscr{T}_{\wn}(p)$ in \textsf{INC} out of $p$ by applying these translations to the axioms and the rules occurring in $p$.
We see by induction on $p$ that $\mathscr{T}_{\wn}(p)$ is well-defined. 

Next, given a formal proof $q$ of a sequent $\mathscr{T}_{\wn}^\ast(\Delta) \vdash [\mathscr{T}_{\wn}(B)], \wn \mathscr{T}_{\wn}^\ast(\Gamma)$ in \textsf{INC}, we obtain by Lemma~\ref{LemCutEliminationForINC_mu} a formal proof $q'$ without Cut$^\wn$ of this sequent in \textsf{INC}. 
Further, we obtain out of $q'$ another formal proof $q''$ without Cut$^\wn$ of the same sequent in \textsf{INC} by advancing (resp. delaying) applications of $\wn$D (resp. $\wn$L$^\wn$) in $q'$ as much as possible except that we restrict an introduction of the formula $\wn \top$ on the left-hand side of sequents in $q'$ to consecutive applications of $\top$L and $\wn$L$^\wn$.
Thanks to the subformula property of \textsf{INC} (by Lemma~\ref{LemCutEliminationForINC_mu}), we see by induction on formal proofs, in which applications of $\wn$D (resp. $\wn$L$^\wn$) are advanced (resp. delayed) as much as possible except that $\top$L is immediately followed by $\wn$L$^\wn$, that $q''$ has a corresponding formal proof of the sequent $\Delta \vdash [B], \Gamma$ in \textsf{LK}.\footnote{For $\vee$R$^\wn$ and $\Rightarrow$L$^\wn$, which are the only nontrivial cases here, we need to take care of not only the sequents $\mathscr{T}_{\wn}^\ast(\Delta) \vdash \wn \mathscr{T}_{\wn}^\ast(\Gamma)$ but also the ones $\mathscr{T}_{\wn}^\ast(\Delta) \vdash \mathscr{T}_{\wn}(B), \wn \mathscr{T}_{\wn}^\ast(\Gamma)$ in \textsf{INC}.} 
Hence, in particular, the translation $\mathscr{T}_{\wn}$ is conservative.
\end{proof}

\begin{lemma}[Translation $\mathscr{T}_{\oc}$ of \textsf{INC} into $\textsf{ILC$_{\boldsymbol{\iota}}$}$]
\label{LemTranslationOfILK_muIntoILLK_mu}
There is a map $\mathscr{T}_{\oc}$ that assigns, to each formal proof $p$ of a sequent $\Delta \vdash \Gamma$ in \textsf{INC}, a formal proof $\mathscr{T}_{\oc}(p)$ of the sequent $\oc \mathscr{T}_{\oc}^\ast(\Delta) \vdash \mathscr{T}_{\oc}^\ast(\Gamma)$ in \textsf{ILC$_\iota$}, where 
\begin{mathpar}
\mathscr{T}_\oc(\top) \colonequals \top
\and
\mathscr{T}_\oc(\mathrm{ff}) \colonequals \oc \bot
\\
\mathscr{T}_\oc(A \mathbin{\&} B) \colonequals \mathscr{T}_\oc(A) \mathbin{\&} \mathscr{T}_\oc(B)
\and 
\mathscr{T}_\oc (A \vee B) \colonequals \oc \mathscr{T}_\oc (A) \oplus \oc \mathscr{T}_\oc(B)
\and
\mathscr{T}_\oc(A \Rightarrow B) \colonequals \oc \mathscr{T}_\oc(A) \rightarrowtriangle \mathscr{T}_\oc(B)
\and
\mathscr{T}_\oc (\wn A) \colonequals \wn \mathscr{T}_\oc(A).
\end{mathpar}
\end{lemma}
\begin{proof}
We first translate the axioms and the rules of \textsf{INC} into derived ones in \textsf{ILC$_\iota$}. 
Let us leave it to the reader to translate XL, XR, WL, $\wn$W, CL, $\wn$C, $\wn$D, Id, $\top$L, $\top$R, $\mathrm{ff}$L, $\mathrm{ff}$R$^\wn$ and $\&$R$^\wn$ in \textsf{INC} into \textsf{ILC$_\iota$} since it is straightforward. 

$\&$L in \textsf{INC} is translated into \textsf{ILC$_\iota$} by
\begin{mathpar}
\AxiomC{}
\LeftLabel{\textsc{(Id)}}
\RightLabel{\textsc{($i \in \overline{2}$)}}
\UnaryInfC{$A_i \vdash A_i$}
\LeftLabel{\textsc{($\&$L)}}
\UnaryInfC{$A_1 \mathbin{\&} A_2 \vdash A_i$}
\LeftLabel{\textsc{($\oc$D)}}
\UnaryInfC{$\oc (A_1 \mathbin{\&} A_2) \vdash A_i$}
\LeftLabel{\textsc{($\oc$R$^{\oc\wn}$)}}
\UnaryInfC{$\oc (A_1 \mathbin{\&} A_2) \vdash \oc A_i$}
\AxiomC{$\oc \Delta, \oc A_i \vdash \Gamma$}
\LeftLabel{\textsc{(Cut)}}
\BinaryInfC{$\oc (A_1 \mathbin{\&} A_2), \oc \Delta \vdash \Gamma$}
\LeftLabel{\textsc{(XL$^\ast$)}}
\doubleLine
\UnaryInfC{$\oc \Delta, \oc (A_1 \mathbin{\&} A_2) \vdash \Gamma$}
\DisplayProof
\end{mathpar}

$\vee$L in \textsf{INC} is translated into \textsf{ILC$_\iota$} by
\begin{mathpar}
\AxiomC{$\oc \Delta, \oc A_1 \vdash \Gamma$}
\AxiomC{$\oc \Delta, \oc A_2 \vdash \Gamma$}
\LeftLabel{\textsc{($\oplus$L)}}
\BinaryInfC{$\oc \Delta, \oc A_1 \oplus \oc A_2 \vdash \Gamma$}
\LeftLabel{\textsc{($\oc$D)}}
\UnaryInfC{$\oc \Delta, \oc (\oc A_1 \oplus \oc A_2) \vdash \Gamma$}
\DisplayProof
\end{mathpar}
and $\vee$R$^\wn$ by
\begin{mathpar}
\AxiomC{$\oc \Delta \vdash B_i, \wn \Gamma$}
\RightLabel{\textsc{($i \in \overline{2}$)}}
\LeftLabel{\textsc{($\oc$R$^{\oc\wn}$)}}
\UnaryInfC{$\oc \Delta \vdash \oc B_i, \wn \Gamma$}
\LeftLabel{\textsc{($\oplus$R)}}
\UnaryInfC{$\oc \Delta \vdash \oc B_1 \oplus \oc B_2, \wn \Gamma$}
\DisplayProof
\end{mathpar}

Next, $\Rightarrow$L$^\wn$ in \textsf{INC} is translated into \textsf{ILC$_\iota$} by
\begin{mathpar}
\AxiomC{}
\LeftLabel{\textsc{(Id)}}
\UnaryInfC{$\oc A \vdash \oc A$}
\LeftLabel{\textsc{($\neg$L)}}
\UnaryInfC{$\oc A, \neg \oc A \vdash$}
\AxiomC{}
\LeftLabel{\textsc{(Id)}}
\UnaryInfC{$B \vdash B$}
\LeftLabel{\textsc{($\invamp$L)}}
\BinaryInfC{$\oc A, \oc A \rightarrowtriangle B \vdash B$}
\LeftLabel{\textsc{($\oc$D)}}
\UnaryInfC{$\oc A, \oc (\oc A \rightarrowtriangle B) \vdash B$}
\LeftLabel{\textsc{($\oc$R$^{\oc\wn}$)}}
\UnaryInfC{$\oc A, \oc (\oc A \rightarrowtriangle B) \vdash \oc B$}
\LeftLabel{\textsc{(XL)}}
\UnaryInfC{$\oc (\oc A \rightarrowtriangle B), \oc A \vdash \oc B$}
\LeftLabel{\textsc{($\neg$R)}}
\UnaryInfC{$\oc (\oc A \rightarrowtriangle B) \vdash \neg \oc A, \oc B$}
\LeftLabel{\textsc{($\invamp$R)}}
\UnaryInfC{$\oc (\oc A \rightarrowtriangle B) \vdash \oc A \rightarrowtriangle \oc B$}
\LeftLabel{\textsc{($\oc$R$^{\oc\wn}$)}}
\UnaryInfC{$\oc (\oc A \rightarrowtriangle B) \vdash \oc (\oc A \rightarrowtriangle \oc B)$}
\AxiomC{$\oc \Theta \vdash A, \wn \Xi$}
\LeftLabel{\textsc{($\oc$R$^{\oc\wn}$)}}
\UnaryInfC{$\oc \Theta \vdash \oc A, \wn \Xi$}
\LeftLabel{\textsc{($\neg$L)}}
\UnaryInfC{$\oc \Theta, \neg \oc A \vdash \wn \Xi$}
\LeftLabel{\textsc{($\invamp$L)}}
\AxiomC{$\oc \Delta, \oc B \vdash \Gamma$}
\BinaryInfC{$\oc \Theta, \oc \Delta, \oc A \rightarrowtriangle \oc B \vdash \wn \Xi, \Gamma$}
\LeftLabel{\textsc{($\oc$D)}}
\UnaryInfC{$\oc \Theta, \oc \Delta, \oc (\oc A \rightarrowtriangle \oc B) \vdash \wn \Xi, \Gamma$}
\LeftLabel{\textsc{(XL$^\ast$)}}
\doubleLine
\UnaryInfC{$\oc \Delta, \oc \Theta, \oc (\oc A \rightarrowtriangle \oc B) \vdash \wn \Xi, \Gamma$}
\LeftLabel{\textsc{(XR$^\ast$)}}
\doubleLine
\UnaryInfC{$\oc \Delta, \oc \Theta, \oc (\oc A \rightarrowtriangle \oc B) \vdash \Gamma, \wn \Xi$}
\LeftLabel{\textsc{(Cut)}}
\BinaryInfC{$\oc (\oc A \rightarrowtriangle B), \oc \Delta, \oc \Theta \vdash \Gamma, \wn \Xi$}
\LeftLabel{\textsc{(XL$^\ast$)}}
\doubleLine
\UnaryInfC{$\oc \Delta, \oc \Theta, \oc (\oc A \rightarrowtriangle B) \vdash \Gamma, \wn \Xi$}
\DisplayProof
\end{mathpar}
and $\Rightarrow$R$^\wn$ by
\begin{mathpar}
\AxiomC{$\oc \Delta, \oc A \vdash B, \wn \Gamma$}
\LeftLabel{\textsc{($\neg$R)}}
\UnaryInfC{$\oc \Delta \vdash \neg \oc A, B, \wn \Gamma$}
\LeftLabel{\textsc{($\invamp$R)}}
\UnaryInfC{$\oc \Delta \vdash \oc A \rightarrowtriangle B, \wn \Gamma$}
\DisplayProof
\end{mathpar}

Now, $\wn$L$^\wn$ in \textsf{INC} is translated into \textsf{ILC$_\iota$} simply by
\begin{mathpar}
\AxiomC{$\oc \Delta, \oc A \vdash \wn \Gamma$}
\LeftLabel{\textsc{($\oc\wn$L$^{\oc\wn}$)}}
\UnaryInfC{$\oc \Delta, \oc \wn A \vdash \wn \Gamma$}
\DisplayProof
\end{mathpar}
and Cut$^\wn$ in \textsf{INC} by
\begin{mathpar}
\AxiomC{$\oc \Delta, \vdash \wn B, \wn \Gamma$}
\LeftLabel{\textsc{($\oc$R$^{\oc\wn}$)}}
\UnaryInfC{$\oc \Delta, \vdash \oc \wn B, \wn \Gamma$}
\AxiomC{$\oc \Delta', \oc B \vdash \wn \Gamma'$}
\LeftLabel{\textsc{($\oc\wn$L$^{\oc\wn}$)}}
\UnaryInfC{$\oc \Delta', \oc \wn B \vdash \wn \Gamma'$}
\LeftLabel{\textsc{(Cut)}}
\BinaryInfC{$\oc \Delta, \oc \Delta' \vdash \wn \Gamma, \wn \Gamma'$}
\DisplayProof
\end{mathpar}

As in the case of the proof of Theorem~\ref{LemTranslationOfLKIntoILK_mu}, given a formal proof $p$ in \textsf{INC}, we obtain the required formal proof $\mathscr{T}_{\oc}(p)$ in \textsf{ILC$_\iota$} out of $p$ by applying the translations given above to the axioms and the rules occurring in $p$.
Again, it is easy to see by induction on $p$ that $\mathscr{T}_{\oc}(p)$ is well-defined. 
\end{proof}

Since this translation $\mathscr{T}_\oc$ of \textsf{INC} into \textsf{ILC$_\iota$} utilises the rule $\oc\wn$L$^{\oc\wn}$, it is not possible into \textsf{ILC}.
This is the main point of the extension of \textsf{ILC} to \textsf{ILC$_\iota$}.

However, unlike the translation $\mathscr{T}_\wn$ given in Lemma~\ref{LemTranslationOfLKIntoILK_mu}, we cannot show that this translation $\mathscr{T}_\oc$ is conservative as \textsf{ILC$_\iota$} does not enjoy cut-elimination. 
We fix this problem by replacing \textsf{ILC$_\iota$} with its substructural calculus in \S\ref{ConservativeTranslations}. 

\begin{remark}
Dually to the translation of the rule Cut$^\wn$ as given in the proof of Theorem~\ref{LemTranslationOfILK_muIntoILLK_mu}, we could instead translate Cut$^\wn$ in \textsf{ILC$_\iota$} by
\begin{mathpar}
\AxiomC{$\oc \Delta, \vdash \wn B, \wn \Gamma$}
\LeftLabel{\textsc{($\wn\oc$R$^{\oc\wn}$)}}
\UnaryInfC{$\oc \Delta, \vdash \wn \oc B, \wn \Gamma$}
\AxiomC{$\oc \Delta', \oc B \vdash \wn \Gamma'$}
\LeftLabel{\textsc{($\wn$L$^{\oc\wn}$)}}
\UnaryInfC{$\oc \Delta', \wn \oc B \vdash \wn \Gamma'$}
\LeftLabel{\textsc{(Cut)}}
\BinaryInfC{$\oc \Delta, \oc \Delta' \vdash \wn \Gamma, \wn \Gamma'$}
\DisplayProof
\end{mathpar}
\end{remark}

Finally, by composing the translations $\mathscr{T}_\wn : \textsf{LK} \rightarrow \textsf{INC}$ and $\mathscr{T}_\oc : \textsf{INC} \rightarrow \textsf{ILC$_\iota$}$, we get a translation $\mathscr{T}_{\oc\wn} \colonequals \mathscr{T}_\oc \circ \mathscr{T}_\wn : \textsf{LK} \rightarrow \textsf{ILC$_\iota$}$:
\begin{corollary}[Translation $\mathscr{T}_{\oc\wn}$ of \textsf{LK} into $\textsf{ILC$_{\boldsymbol{\iota}}$}$]
\label{CoroTranslationOfLKintoILLK_mu}
The composition $\mathscr{T}_{\oc\wn} \colonequals \mathscr{T}_\oc \circ \mathscr{T}_\wn$ assigns, to each formal proof $p$ of a sequent $\Delta \vdash \Gamma$ in \textsf{LK}, a formal proof $\mathscr{T}_{\oc\wn}(p)$ of the sequent $\oc \mathscr{T}_{\oc\wn}^\ast(\Delta) \vdash \wn \mathscr{T}_{\oc\wn}^\ast(\Gamma)$ in \textsf{ILC$_\iota$}, where 
\begin{mathpar}
\mathscr{T}_{\oc\wn}(\mathrm{tt}) \colonequals \wn \top
\and
\mathscr{T}_{\oc\wn}(\mathrm{ff}) \colonequals \oc \bot
\and
\mathscr{T}_{\oc\wn}(A \wedge B) \colonequals \wn \mathscr{T}_{\oc\wn}(A) \mathbin{\&} \wn \mathscr{T}_{\oc\wn}(B)
\and
\mathscr{T}_{\oc\wn} (A \vee B) \colonequals \oc \mathscr{T}_{\oc\wn} (A) \oplus \oc \mathscr{T}_{\oc\wn} (B)
\and
\mathscr{T}_{\oc\wn}(A \Rrightarrow B) \colonequals \oc \mathscr{T}_{\oc\wn}(A) \rightarrowtriangle \wn \mathscr{T}_{\oc\wn}(B).
\end{mathpar}
\end{corollary}
\begin{proof}
By Lemmata~\ref{LemTranslationOfLKIntoILK_mu} and \ref{LemTranslationOfILK_muIntoILLK_mu}.
\end{proof}


\subsection{Classical linear logic negative}
\label{ClassicalLinearLogicNegative}
Let us proceed to decompose again the translation $\mathscr{T}_{\oc\wn}$ into the more primitive ones $\mathscr{T}_{\oc}$ and $\mathscr{T}_{\wn}$ yet in the reverse order this time.
Accordingly, an intermediate logic between ILL$^{\text{e}}_\iota$ and CL for this new decomposition is the dual of IL$^{\text{e}}$:
\begin{definition}[Formulas of CLL$^{\boldsymbol{-}}$]
\label{DefFormulasOfCLL$^-$}
Formulas $A, B$ of \emph{classical linear logic negative (CLL$^-$)} are defined by
\begin{mathpar}
A, B \colonequals X \mid \mathrm{tt} \mid \bot \mid A \wedge B \mid A \oplus B \mid A \looparrowright B \mid \oc A
\end{mathpar}
where $X$ ranges over propositional variables, and $A_\star \colonequals A \looparrowright \bot$.

We call $\looparrowright$ \emph{(classical linear) implication}, and $(\_)_\star$ \emph{(classical linear) negation}.
\end{definition}

\begin{definition}[\textsf{CLC} for CLL$^{\boldsymbol{-}}$]
\label{DefCLLK_mu}
The sequent calculus \emph{\textsf{CLC}} for CLL$^-$ consists of the axioms and the rules displayed in Figure~\ref{FigCLLK_mu}. 
\begin{figure}
\begin{mathpar}
\AxiomC{$\Delta, A, A', \Delta' \vdash \Gamma$}
\LeftLabel{\textsc{(XL)}}
\UnaryInfC{$\Delta, A', A, \Delta' \vdash \Gamma$}
\DisplayProof \and
\AxiomC{$\Delta \vdash \Gamma, B, B', \Gamma'$}
\LeftLabel{\textsc{(XR)}}
\UnaryInfC{$\Delta \vdash \Gamma, B', B, \Gamma'$}
\DisplayProof \and
\AxiomC{$\Delta \vdash \Gamma$}
\LeftLabel{\textsc{($\oc$W)}}
\UnaryInfC{$\Delta, \oc A \vdash \Gamma$}
\DisplayProof \and
\AxiomC{$\Delta\vdash \Gamma$}
\LeftLabel{\textsc{(WR)}}
\UnaryInfC{$\Delta \vdash B, \Gamma$}
\DisplayProof \and
\AxiomC{$\Delta, \oc A, \oc A \vdash \Gamma$}
\LeftLabel{\textsc{($\oc$C)}}
\UnaryInfC{$\Delta, \oc A \vdash \Gamma$}
\DisplayProof \and
\AxiomC{$\Delta \vdash B, B, \Gamma$}
\LeftLabel{\textsc{(CR)}}
\UnaryInfC{$\Delta \vdash B, \Gamma$}
\DisplayProof \and
\AxiomC{$\oc \Delta, A \vdash \Gamma$}
\LeftLabel{\textsc{($\oc$D)}}
\UnaryInfC{$\oc \Delta, \oc A \vdash \Gamma$}
\DisplayProof
\and
\AxiomC{$\oc \Delta \vdash B, \Gamma$}
\LeftLabel{\textsc{($\oc$R$^\oc$)}}
\UnaryInfC{$\oc \Delta \vdash \oc B, \Gamma$}
\DisplayProof \\
\AxiomC{}
\LeftLabel{\textsc{(Id)}}
\UnaryInfC{$A \vdash A$}
\DisplayProof \and
\AxiomC{$\oc \Delta \vdash B, \Gamma$}
	\AxiomC{$\oc \Delta', \oc B \vdash \Gamma'$}
	\LeftLabel{\textsc{(Cut$^\oc$)}}
\BinaryInfC{$\oc \Delta, \oc \Delta' \vdash \Gamma, \Gamma'$}
\DisplayProof 
\\
\AxiomC{$\oc \Delta \vdash \Gamma$}
\LeftLabel{\textsc{($\mathrm{tt}$L$^\oc$)}}
\UnaryInfC{$\oc \Delta, \mathrm{tt} \vdash \Gamma$}
\DisplayProof \and
\AxiomC{}
\LeftLabel{\textsc{($\mathrm{tt}$R)}}
\UnaryInfC{$\vdash \mathrm{tt}$}
\DisplayProof \and
\AxiomC{}
\LeftLabel{\textsc{($\bot$L)}}
\UnaryInfC{$\bot \vdash$}
\DisplayProof \and
\AxiomC{$\Delta \vdash \Gamma$}
\LeftLabel{\textsc{($\bot$R)}}
\UnaryInfC{$\Delta \vdash \bot, \Gamma$}
\DisplayProof \\
\AxiomC{$\oc \Delta, A_i \vdash \Gamma$}
	\LeftLabel{\textsc{($\wedge$L$^\oc$)}}
	\RightLabel{($i \in \overline{2}$)}
\UnaryInfC{$\oc \Delta, A_1 \wedge A_2 \vdash \Gamma$}
\DisplayProof \and
\AxiomC{$\Delta \vdash B_1, \Gamma$}
		\AxiomC{$\Delta \vdash B_2, \Gamma$}
	\LeftLabel{\textsc{($\wedge$R)}}
\BinaryInfC{$\Delta \vdash B_1 \wedge B_2, \Gamma$}
\DisplayProof \\
\AxiomC{$\oc \Delta, A_1 \vdash \Gamma$}
		\AxiomC{$\oc \Delta, A_2 \vdash \Gamma$}
	\LeftLabel{\textsc{($\oplus$L$^\oc$)}}
\BinaryInfC{$\oc \Delta, A_1 \oplus A_2 \vdash \Gamma$}
\DisplayProof \and
\AxiomC{$\Delta \vdash B_i, \Gamma$}
	\LeftLabel{\textsc{($\oplus$R)}}
	\RightLabel{($i \in \overline{2}$)}
\UnaryInfC{$\Delta \vdash B_1 \oplus B_2, \Gamma$}
\DisplayProof \\
\AxiomC{$\oc \Delta, B \vdash \Gamma$}
\AxiomC{$\Theta \vdash A, \Xi$}
\LeftLabel{\textsc{($\looparrowright$L$^\oc$)}}
\BinaryInfC{$\oc \Delta, \Theta, \oc (A \looparrowright B) \vdash \Gamma, \Xi$}
\DisplayProof \and
\AxiomC{$\oc \Delta, A \vdash B, \Gamma$}
\LeftLabel{\textsc{($\looparrowright$R$^\oc$)}}
\UnaryInfC{$\oc \Delta \vdash A \looparrowright B, \Gamma$}
\DisplayProof
\end{mathpar}
\caption{Sequent calculus \textsf{CLC} for CLL$^-$}
\label{FigCLLK_mu}
\end{figure}
\end{definition}

\begin{lemma}[Cut-elimination for \textsf{CLC}]
\label{PropCutEliminationForCLC}
Given a formal proof of a sequent in \textsf{CLC}, there is a formal proof of this sequent in \textsf{CLC} without Cut$^\oc$.
\end{lemma}
\begin{proof}
Again, similar to the case of Theorem~\ref{ThmCutEliminationForILC}.
\end{proof}

By the evident analogue (or dual) of the explanation for the design of IL$^{\text{e}}$ and \textsf{INC} given at the beginning of \S\ref{ConservativeExtensionOfIntuitionisticLogic}, we can similarly explain the design of CLC$^-$ and \textsf{CLC}. 
Let us leave the details to the reader.

There is a formal proof of LEM with respect to plus $\oplus$ in \textsf{CLC}, which is analogous to the formal proof of LEM with respect to non-linear disjunction $\vee$ in \textsf{LK} given in \S\ref{SequentCalculiForClassicalAndIntuitionisticLogics}. 
Hence, CLL$^-$ is \emph{classical} in this conventional sense.

Also, CLL$^-$ is \emph{linear} in the sense that the structural rules on the left-hand side of sequents in \textsf{CLC} coincide with those on linear logics (\S\ref{LinearLogics}). 
However, this point does not hold on the right-hand side of sequents in \textsf{CLC}, which departs from the linearity in the sense of CLL (Remark~\ref{RemarkOnLinearity}).
Our proposal on linearity and classicality is motivated by game semantics (\S\ref{MainResults}); see Remark~\ref{RemarkOnBacktracks}.

On the other hand, CLL$^-$ dispenses with linear negation $(\_)^\bot$ unlike CLL, by which the game semantics mentioned in \S\ref{OurContributionsAndRelatedWork} interprets \textsf{CLC} by \emph{negative} games only. 
Accordingly, we call CLL$^-$ \emph{classical}, \emph{linear} and \emph{negative}.

Let us then decompose the translation $\mathscr{T}_{\oc\wn}$ by the following two lemmata:
\begin{lemma}[Translation $\mathscr{T}_{\oc}$ of \textsf{LK} into $\textsf{CLC}$]
\label{LemTranslationOfLKintoCLLK}
There is a map $\mathscr{T}_{\oc}$ that assigns, to each formal proof $p$ of a sequent $\Delta \vdash \Gamma$ in \textsf{LK}, a formal proof $\mathscr{T}_{\oc}(p)$ of the sequent $\oc \mathscr{T}_{\oc}^\ast(\Delta) \vdash \mathscr{T}_{\oc}^\ast(\Gamma)$ in \textsf{CLC}, where 
\begin{mathpar}
\mathscr{T}_\oc(\mathrm{tt}) \colonequals \mathrm{tt}
\and
\mathscr{T}_\oc(\mathrm{ff}) \colonequals \oc \bot
\and
\mathscr{T}_\oc(A \wedge B) \colonequals \mathscr{T}_\oc(A) \wedge \mathscr{T}_\oc(B)
\and
\mathscr{T}_\oc(A \vee B) \colonequals \oc \mathscr{T}_\oc(A) \oplus \oc \mathscr{T}_\oc(B)
\and
\mathscr{T}_\oc(A \Rrightarrow B) \colonequals \oc \mathscr{T}_\oc(A) \looparrowright \mathscr{T}_\oc(B).
\end{mathpar}

Moreover, it is conservative: If $\Delta, \Gamma$ has only formulas of CL, and $\oc \mathscr{T}_\oc^\ast(\Delta) \vdash \mathscr{T}_\oc^\ast(\Gamma)$ is provable in \textsf{CLC}, then $\Delta \vdash \Gamma$ is provable in \textsf{LK}.
\end{lemma}
\begin{proof}
Let us first translate the axioms and the rules of \textsf{LK} into derived ones in \textsf{CLC}.
Let us leave it to the reader to translate XL, XR, WL, WR, CL, CR, Id, Cut, $\mathrm{tt}$L, $\mathrm{tt}$R, $\mathrm{ff}$L, $\mathrm{ff}$R, $\wedge$R and $\Rrightarrow$R in \textsf{LK} into \textsf{CLC} since it is straightforward.

$\wedge$L in \textsf{LK} is translated into \textsf{CLC} by
\begin{mathpar}
\AxiomC{}
\LeftLabel{\textsc{(Id)}}
\RightLabel{($i \in \overline{2}$)}
\UnaryInfC{$A_i \vdash A_i$}
\LeftLabel{\textsc{($\wedge$L$^\oc$)}}
\UnaryInfC{$A_1 \wedge A_2 \vdash A_i$}
\LeftLabel{\textsc{($\oc$D)}}
\UnaryInfC{$\oc (A_1 \wedge A_2) \vdash A_i$}
\AxiomC{$\oc \Delta, \oc A_i \vdash \Gamma$}
\LeftLabel{\textsc{(Cut$^\oc$)}}
\BinaryInfC{$\oc (A_1 \wedge A_2), \oc \Delta \vdash \Gamma$}
\LeftLabel{\textsc{(XL$^\ast$)}}
\doubleLine
\UnaryInfC{$\oc \Delta, \oc (A_1 \wedge A_2) \vdash \Gamma$}
\DisplayProof
\end{mathpar}

$\vee$L in \textsf{LK} is translated into \textsf{CLC} by
\begin{mathpar}
\AxiomC{$\oc \Delta, \oc A_1 \vdash \Gamma$}
\AxiomC{$\oc \Delta, \oc A_2 \vdash \Gamma$}
\LeftLabel{\textsc{($\oplus$L$^\oc$)}}
\BinaryInfC{$\oc \Delta, \oc A_1 \oplus \oc A_2 \vdash \Gamma$}
\LeftLabel{\textsc{($\oc$D)}}
\UnaryInfC{$\oc \Delta, \oc (\oc A_1 \oplus \oc A_2) \vdash \Gamma$}
\DisplayProof
\end{mathpar}
and $\vee$R by
\begin{mathpar}
\AxiomC{$\oc \Delta \vdash B_i, \Gamma$}
\LeftLabel{\textsc{($\oc$R$^\oc$)}}
\UnaryInfC{$\oc \Delta \vdash \oc B_i, \Gamma$}
\LeftLabel{\textsc{($\oplus$R)}}
\UnaryInfC{$\oc \Delta \vdash \oc B_1 \oplus \oc B_2, \Gamma$}
\DisplayProof
\end{mathpar}

Finally, $\Rrightarrow$L in \textsf{LK} is translated into \textsf{CLC} by
\begin{mathpar}
\AxiomC{}
\LeftLabel{\textsc{(Id)}}
\UnaryInfC{$B \vdash B$}
\AxiomC{}
\LeftLabel{\textsc{(Id)}}
\UnaryInfC{$\oc A \vdash \oc A$}
\LeftLabel{\textsc{($\looparrowright$L$^\oc$)}}
\BinaryInfC{$\oc A, \oc (\oc A \looparrowright B) \vdash B$}
\LeftLabel{\textsc{($\oc$R$^\oc$)}}
\UnaryInfC{$\oc A, \oc (\oc A \looparrowright B) \vdash \oc B$}
\LeftLabel{\textsc{(XL)}}
\UnaryInfC{$\oc (\oc A \looparrowright B), \oc A \vdash \oc B$}
\LeftLabel{\textsc{($\looparrowright$R$^\oc$)}}
\UnaryInfC{$\oc (\oc A \looparrowright B) \vdash \oc A \looparrowright \oc B$}
\AxiomC{$\oc \Delta, \oc B \vdash \Gamma$}
\AxiomC{$\oc \Theta \vdash A, \Xi$}
\LeftLabel{\textsc{($\oc$R$^\oc$)}}
\UnaryInfC{$\oc \Theta \vdash \oc A, \Xi$}
\LeftLabel{\textsc{($\looparrowright$L$^\oc$)}}
\BinaryInfC{$\oc \Delta, \oc \Theta, \oc (\oc A \looparrowright \oc B) \vdash \Gamma, \Xi$}
\LeftLabel{\textsc{(Cut$^\oc$)}}
\BinaryInfC{$\oc (\oc A \looparrowright B), \oc \Delta, \oc \Theta \vdash \Gamma, \Xi$}
\LeftLabel{\textsc{(XL$^\ast$)}}
\doubleLine
\UnaryInfC{$\oc \Delta, \oc \Theta, \oc (\oc A \looparrowright B) \vdash \Gamma, \Xi$}
\DisplayProof
\end{mathpar}

As in the case of the proof of Theorem~\ref{LemTranslationOfLKIntoILK_mu}, these translations of axioms and rules induce the required translation $\mathscr{T}_{\oc}$ on formal proofs. 

Finally, we show that this translation $\mathscr{T}_{\oc}$ is conservative by the method symmetric to the one given in the last part of the proof of Lemma~\ref{LemTranslationOfLKIntoILK_mu}.
\end{proof}

\begin{lemma}[Translation $\mathscr{T}_{\wn}$ of \textsf{CLC} into \textsf{ILC$_{\boldsymbol{\iota}}$}]
\label{LemTranslationOfCLLKintoILLK}
There is a map $\mathscr{T}_{\wn}$ that assigns, to each formal proof $p$ of a sequent $\Delta \vdash \Gamma$ in \textsf{CLC}, a formal proof $\mathscr{T}_{\wn}(p)$ of the sequent $\mathscr{T}_{\wn}^\ast(\Delta) \vdash \wn \mathscr{T}_{\wn}^\ast(\Gamma)$ in \textsf{ILC$_\iota$}, where 
\begin{mathpar}
\mathscr{T}_\wn(\mathrm{tt}) \colonequals \wn \top
\and
\mathscr{T}_\wn(\bot) \colonequals \bot
\\
\mathscr{T}_\wn(A \wedge B) \colonequals \wn \mathscr{T}_\wn(A) \mathbin{\&} \wn \mathscr{T}_\wn(B)
\and
\mathscr{T}_\wn(A \oplus B) \colonequals \mathscr{T}_\wn(A) \oplus \mathscr{T}_\wn(B)
\and
\mathscr{T}_\wn(A \looparrowright B) \colonequals \mathscr{T}_\wn(A) \rightarrowtriangle \wn \mathscr{T}_\wn(B)
\and
\mathscr{T}_\wn(\oc A) \colonequals \oc \mathscr{T}_\wn(A).
\end{mathpar}
\end{lemma}
\begin{proof}
Let us first translate the axioms and the rules of \textsf{CLC} into derived ones in \textsf{ILC$_\iota$}.
It is trivial to translate XL, XR, $\oc$W, WR, $\oc$C, CR, $\oc$D, Id, $\mathrm{tt}$L$^\oc$, $\mathrm{tt}$R, $\bot$L and $\bot$R; let us leave the details to the reader.

$\wedge$L$^\oc$ in \textsf{CLC} is translated into \textsf{ILC$_\iota$} by
\begin{mathpar}
\AxiomC{$\oc \Delta, A_i \vdash \wn \Gamma$}
\RightLabel{\textsc{($i \in \overline{2}$)}}
\LeftLabel{\textsc{($\wn$L$^{\oc\wn}$)}}
\UnaryInfC{$\oc \Delta, \wn A_i \vdash \wn \Gamma$}
\LeftLabel{\textsc{($\&$L)}}
\UnaryInfC{$\oc \Delta, \wn A_1 \mathbin{\&} \wn A_2 \vdash \wn \Gamma$}
\DisplayProof
\end{mathpar}
and $\wedge$R by
\begin{mathpar}
\AxiomC{$\Delta \vdash \wn B_1, \wn \Gamma$}
\AxiomC{$\Delta \vdash \wn B_2, \wn \Gamma$}
\LeftLabel{\textsc{($\&$R)}}
\BinaryInfC{$\Delta \vdash \wn B_1 \mathbin{\&} \wn B_2, \wn \Gamma$}
\LeftLabel{\textsc{($\wn$D)}}
\UnaryInfC{$\Delta \vdash \wn (\wn B_1 \mathbin{\&} \wn B_2), \wn \Gamma$}
\DisplayProof
\end{mathpar}

Dually, $\oplus$L$^\oc$ in \textsf{CLC} is translated into \textsf{ILC$_\iota$} simply by
\begin{mathpar}
\AxiomC{$\oc \Delta, A_1 \vdash \wn \Gamma$}
\AxiomC{$\oc \Delta, A_2 \vdash \wn \Gamma$}
\LeftLabel{\textsc{($\oplus$L)}}
\BinaryInfC{$\oc \Delta, A_1 \oplus A_2 \vdash \wn \Gamma$}
\DisplayProof
\end{mathpar}
and $\oplus$R by
\begin{mathpar}
\AxiomC{$\Delta \vdash \wn B_i, \wn \Gamma$}
\AxiomC{}
\LeftLabel{\textsc{(Id)}}
\RightLabel{\textsc{($i \in \overline{2}$)}}
\UnaryInfC{$B_i \vdash B_i$}
\LeftLabel{\textsc{($\oplus$R)}}
\UnaryInfC{$B_i \vdash B_1 \oplus B_2$}
\LeftLabel{\textsc{($\wn$D)}}
\UnaryInfC{$B_i \vdash \wn (B_1 \oplus B_2)$}
\LeftLabel{\textsc{($\wn$L$^{\oc\wn}$)}}
\UnaryInfC{$\wn B_i \vdash \wn (B_1 \oplus B_2)$}
\LeftLabel{\textsc{(Cut)}}
\BinaryInfC{$\Delta \vdash \wn \Gamma, \wn (B_1 \oplus B_2)$}
\LeftLabel{\textsc{(XR$^\ast$)}}
\doubleLine
\UnaryInfC{$\Delta \vdash \wn (B_1 \oplus B_2), \wn \Gamma$}
\DisplayProof
\end{mathpar}

Next, $\looparrowright$L$^\oc$ in \textsf{CLC} is translated into \textsf{ILC$_\iota$} by
\begin{mathpar}
\AxiomC{}
\LeftLabel{\textsc{(Id)}}
\UnaryInfC{$A \vdash A$}
\LeftLabel{\textsc{($\neg$L)}}
\UnaryInfC{$A, \neg A \vdash$}
\AxiomC{}
\LeftLabel{\textsc{(Id)}}
\UnaryInfC{$\wn B \vdash \wn B$}
\LeftLabel{\textsc{($\invamp$L)}}
\BinaryInfC{$A, A \rightarrowtriangle \wn B \vdash \wn B$}
\LeftLabel{\textsc{($\oc$D)}}
\UnaryInfC{$A, \oc (A \rightarrowtriangle \wn B) \vdash \wn B$}
\LeftLabel{\textsc{(XL)}}
\UnaryInfC{$\oc (A \rightarrowtriangle \wn B), A \vdash \wn B$}
\LeftLabel{\textsc{($\wn$L$^{\oc\wn}$)}}
\UnaryInfC{$\oc (A \rightarrowtriangle \wn B), \wn A \vdash \wn B$}
\LeftLabel{\textsc{($\neg$R)}}
\UnaryInfC{$\oc (A \rightarrowtriangle \wn B) \vdash \neg \wn A, \wn B$}
\LeftLabel{\textsc{($\invamp$R)}}
\UnaryInfC{$\oc (A \rightarrowtriangle \wn B) \vdash \wn A \rightarrowtriangle \wn B$}
\AxiomC{$\Theta \vdash \wn A, \wn \Xi$}
\LeftLabel{\textsc{($\neg$L)}}
\UnaryInfC{$\Theta, \neg \wn A \vdash \wn \Xi$}
\AxiomC{$\oc \Delta, B \vdash \wn \Gamma$}
\LeftLabel{\textsc{($\wn$L$^{\oc\wn}$)}}
\UnaryInfC{$\oc \Delta, \wn B \vdash \wn \Gamma$}
\LeftLabel{\textsc{($\invamp$L)}}
\BinaryInfC{$\Theta, \oc \Delta, \wn A \rightarrowtriangle \wn B \vdash \wn \Xi, \wn \Gamma$}
\LeftLabel{\textsc{(Cut)}}
\BinaryInfC{$\oc (A \rightarrowtriangle \wn B), \Theta, \oc \Delta \vdash \wn \Xi, \wn \Gamma$}
\LeftLabel{\textsc{(XL$^\ast$)}}
\doubleLine
\UnaryInfC{$\oc \Delta, \Theta, \oc (A \rightarrowtriangle \wn B) \vdash \wn \Xi, \wn \Gamma$}
\DisplayProof
\end{mathpar}
where of-course $\oc$ on $A \looparrowright B$ in $\looparrowright$L$^\oc$ is vital for this translation, and $\looparrowright$R$^\oc$ by
\begin{mathpar}
\AxiomC{$\oc \Delta, A \vdash \wn B, \wn \Gamma$}
\LeftLabel{\textsc{($\neg$R)}}
\UnaryInfC{$\oc \Delta \vdash \neg A, \wn B, \wn \Gamma$}
\LeftLabel{\textsc{($\invamp$R)}}
\UnaryInfC{$\oc \Delta \vdash A \rightarrowtriangle \wn B, \wn \Gamma$}
\LeftLabel{\textsc{($\wn$D)}}
\UnaryInfC{$\oc \Delta \vdash \wn (A \rightarrowtriangle \wn B), \wn \Gamma$}
\DisplayProof
\end{mathpar}

Finally, $\oc$R$^\oc$ in \textsf{CLC} is translated into \textsf{ILC$_\iota$} by
\begin{mathpar}
\AxiomC{$\oc \Delta \vdash \wn B, \wn \Gamma$}
\LeftLabel{\textsc{($\oc$R$^{\oc\wn}$)}}
\UnaryInfC{$\oc \Delta \vdash \oc \wn B, \wn \Gamma$}
\AxiomC{}
\UnaryInfC{$\oc B \vdash \oc B$}
\LeftLabel{\textsc{($\wn$D)}}
\UnaryInfC{$\oc B \vdash \wn \oc B$}
\LeftLabel{\textsc{($\oc\wn$L$^{\oc\wn}$)}}
\UnaryInfC{$\oc \wn B \vdash \wn \oc B$}
\LeftLabel{\textsc{(Cut)}}
\BinaryInfC{$\oc \Delta \vdash \wn \Gamma, \wn \oc B$}
\LeftLabel{\textsc{(XR$^\ast$)}}
\doubleLine
\UnaryInfC{$\oc \Delta \vdash \wn \oc B, \wn \Gamma$}
\DisplayProof
\end{mathpar}
and Cut$^\oc$ in \textsf{CLC} by
\begin{mathpar}
\AxiomC{$\oc \Delta, \vdash \wn B, \wn \Gamma$}
\LeftLabel{\textsc{($\oc$R$^{\oc\wn}$)}}
\UnaryInfC{$\oc \Delta, \vdash \oc \wn B, \wn \Gamma$}
\AxiomC{$\oc \Delta', \oc B \vdash \wn \Gamma'$}
\LeftLabel{\textsc{($\oc\wn$L$^{\oc\wn}$)}}
\UnaryInfC{$\oc \Delta', \oc \wn B \vdash \wn \Gamma'$}
\LeftLabel{\textsc{(Cut)}}
\BinaryInfC{$\oc \Delta, \oc \Delta' \vdash \wn \Gamma, \wn \Gamma'$}
\DisplayProof
\end{mathpar}

As in the case of the proof of Lemma~\ref{LemTranslationOfILK_muIntoILLK_mu}, these translations of axioms and rules induce the required translation $\mathscr{T}_{\wn}$ of formal proofs. 
\end{proof}

\begin{remark}
As in the case of the translation $\mathscr{T}_\oc$ of \textsf{INC} by \textsf{ILC$_\iota$} (Lemma~\ref{LemTranslationOfILK_muIntoILLK_mu}), we adopt $\oc\wn$L$^{\wn\oc}$ rather than $\wn\oc$R$^{\oc\wn}$ in the translation $\mathscr{T}_\wn$ of \textsf{CLC} into \textsf{ILC$_\iota$}, even when it makes the translation more complex, mainly for \S\ref{ConservativeTranslations} (as we shall see).
\end{remark}

\begin{corollary}[Translation $\mathscr{T}_{\wn\oc}$ of \textsf{LK} into \textsf{ILC}$_{\boldsymbol{\iota}}$]
\label{CoroAnotherTranslationOfLKintoILLK_mu}
The composition $\mathscr{T}_{\wn\oc} \colonequals \mathscr{T}_\wn \circ \mathscr{T}_\oc$ assigns, to each formal proof $p$ of a sequent $\Delta \vdash \Gamma$ in \textsf{LK}, a formal proof $\mathscr{T}_{\wn\oc}(p)$ of the sequent $\oc \mathscr{T}_{\wn\oc}^\ast(\Delta) \vdash \wn \mathscr{T}_{\wn\oc}^\ast(\Gamma)$ in \textsf{ILC$_\iota$}, where 
\begin{mathpar}
\mathscr{T}_{\wn\oc}(\mathrm{tt}) \colonequals \wn \top
\and
\mathscr{T}_{\wn\oc}(\mathrm{ff}) \colonequals \oc \bot
\and
\mathscr{T}_{\wn\oc}(A \wedge B) \colonequals \wn \mathscr{T}_{\wn\oc}(A) \mathbin{\&} \wn \mathscr{T}_{\wn\oc}(B)
\and 
\mathscr{T}_{\wn\oc} (A \vee B) \colonequals \oc \mathscr{T}_{\wn\oc} (A) \oplus \oc \mathscr{T}_{\wn\oc}(B)
\and 
\mathscr{T}_{\wn\oc}(A \Rrightarrow B) \colonequals \oc \mathscr{T}_{\wn\oc}(A) \rightarrowtriangle \wn \mathscr{T}_{\wn\oc}(B).
\end{mathpar}
\end{corollary}
\begin{proof}
By Lemmata~\ref{LemTranslationOfLKintoCLLK} and \ref{LemTranslationOfCLLKintoILLK}.
\end{proof}

\subsection{Commutative unity of logic}
\label{SubsectionCommutativeUnityOfLogic}
In the previous two subsections, we have presented two translations of CL into ILL$^{\text{e}}_\iota$, where one is through IL$^{\text{e}}$, and the other through CLL$^-$.
Let us then establish the theorem proposed in \S\ref{MainResults}, i.e., these two routes \emph{commute}:
\begin{theorem}[Commutative unity of logic]
\label{ThmCommutativeUnityOfLogic}
Given a formal proof $p$ of a sequent $\Delta \vdash \Gamma$ in \textsf{LK}, the formal proofs $\mathscr{T}_{\oc\wn}(p)$ and $\mathscr{T}_{\wn\oc}(p)$ of the sequent $\oc \mathscr{T}_{\oc\wn}^\ast(\Delta) \vdash \wn \mathscr{T}_{\oc\wn}^\ast(\Gamma)$ in \textsf{ILC$_\iota$}, where $\mathscr{T}_{\oc\wn}^\ast(\Delta) = \mathscr{T}_{\wn\oc}^\ast(\Delta)$ and $\mathscr{T}_{\oc\wn}^\ast(\Gamma) = \mathscr{T}_{\wn\oc}^\ast(\Gamma)$, coincide modulo permuting axioms and rules in formal proofs in \textsf{ILC$_\iota$}. 
\end{theorem}
\begin{proof}
By induction on $p$.
\end{proof}

The permutation of axioms and rules stated in Theorem~\ref{ThmCommutativeUnityOfLogic} is exclusively on the order of introducing of-course $\oc$ and why-not $\wn$ on sequents, e.g., the order of applying the rules $\oc$D and $\wn$D on a sequent.
This trivial permutation is completely ignored by the categorical and the game semantics mentioned in \S\ref{OurContributionsAndRelatedWork}. 
In other words, the compromise of the commutativity modulo permutation is only due to the inessential syntactic details of sequent calculi.


\subsection{Conservative translations}
\label{ConservativeTranslations}
The remaining problem is the difficulty in showing the conservativity of the translations of \textsf{INC} into \textsf{ILC$_\iota$} (Lemma~\ref{LemTranslationOfILK_muIntoILLK_mu}) and \textsf{CLC} into \textsf{ILC$_\iota$} (Lemma~\ref{LemTranslationOfCLLKintoILLK}).
Thus, in this last subsection of \S\ref{CommutativeUnityOfLogic}, we prove the corollary articulated in \S\ref{MainResults}, i.e., refine Theorem~\ref{ThmCommutativeUnityOfLogic} in such a way that the translations become all \emph{conservative}. 
Since the conservativity proofs of Lemmata~\ref{LemTranslationOfLKIntoILK_mu} and \ref{LemTranslationOfLKintoCLLK} rely on the cut-elimination of the codomain sequent calculi, our  idea on the corollary is to carve out a substructural sequent calculus of \textsf{ILC$_\iota$} that enjoys cut-elimination yet powerful enough to accommodate our unity of logic (Theorem~\ref{ThmCommutativeUnityOfLogic}).

Concretely, we define a substructural logic of ILL$^{\text{e}}_\iota$, which let us call \emph{intuitionistic linear logic $\rho$-extended (ILL$^{\text{e}}_\rho$)}, and a sequent calculus \emph{\textsf{ILC$_\rho$}} for ILL$^{\text{e}}_\rho$ that \emph{enjoys cut-elimination}.
We also carve out substructural sequent calculi \emph{\textsf{LK$_\rho$}} of \textsf{LK}, \emph{\textsf{INC$_\rho$}} of \textsf{INC}, and \emph{\textsf{CLC$_\rho$}} of \textsf{CLC}, respectively, for CL, IL$^{\text{e}}$ and CLL$^-$, so that the corresponding restrictions of the translations given so far induce conservative translations of \textsf{LK$_\rho$} into \textsf{INC$_\rho$}, \textsf{INC$_\rho$} into \textsf{ILC$_\rho$}, \textsf{LK$_\rho$} into \textsf{CLC$_\rho$}, and \textsf{CLC$_\rho$} into \textsf{ILC$_\rho$}.
In particular, we can show that the translations of \textsf{INC$_\rho$} into \textsf{ILC$_\rho$}, and \textsf{CLC$_\rho$} into \textsf{ILC$_\rho$} are \emph{conservative} thanks to the cut-elimination of \textsf{ILC$_\rho$}, which overcomes the deficiency of Theorem~\ref{ThmCommutativeUnityOfLogic}.

Consequently, only the sequent calculus \textsf{ILC$_\rho$} suffices to capture ILL$^{\text{e}}_\rho$, IL$^{\text{e}}$, CLL$^-$ and CL precisely at the level of provability, preserving the commutativity of Theorem~\ref{ThmCommutativeUnityOfLogic}. 
This result is summarised in Figure~\ref{FigConservativeTranslationsForCommutativeUnityOfLogic}.

\begin{figure}
\begin{mathpar}
\begin{tikzcd}
\arrow[d, hook, "\text{conservative extension}"'] \text{ILL} \arrow[rrrr, "\text{Girard's translation}"] &&&& \text{IL} \arrow[d, hook, "\text{conservative extension}"] \\
\arrow[d, "\text{classicalisation $(\_)_\wn$}"'] \text{ILL$^{\text{e}}_\rho$ (\textsf{ILC$_\rho$})} \arrow[rrrr, "\text{unlinearisation $(\_)_\oc$}"'] &&&& \arrow[d, "\text{classicalisation $(\_)_\wn$}"] \text{IL$^{\text{e}}$ (\textsf{INC}$_\rho$)} \\
\text{CLL$^-$ (\textsf{CLC}$_\rho$)} \arrow[rrrr, "\text{unlinearisation $(\_)_\oc$}"'] &&&& \text{CL (\textsf{LK}$_\rho$)}
\end{tikzcd}
\end{mathpar}
\caption{Conservative translations for a commutative unity of logic}
\label{FigConservativeTranslationsForCommutativeUnityOfLogic}
\end{figure}
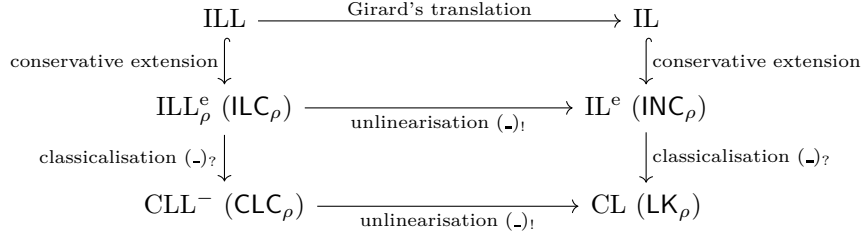

Towards this result, we first define ILL$^{\text{e}}_\rho$ and \textsf{ILC}$_\rho$.
Let us emphasise that the novel calculus \textsf{ILC$_\rho$}, which is powerful enough to embody ILL$^{\text{e}}_\rho$, IL$^{\text{e}}$, CLL$^-$ and CL, yet enjoys cut-elimination, is a technical highlight of this work.

\begin{definition}[Pure occurrences in \textsf{ILC}$_{\boldsymbol{\iota}}$]
\label{DefPureOccurrences}
An occurrence of a formula $A$ in a sequent $\Delta \vdash \Gamma$ in a formal proof $p$ in \textsf{ILC}$_\iota$ is \emph{(hereditarily) pure} if each sequent $\Delta' \vdash \Gamma'$ occurring in the subtree of $p$ whose root is $\Delta \vdash \Gamma$ satisfies

\begin{enumerate}

\item There is no application of $\oc$W, $\oc$C, $\wn$W or $\wn$C on subformulas of $A$ occurring in $\Delta' \vdash \Gamma'$ that constitute the occurrence $A$ in $\Delta \vdash \Gamma$;

\item An occurrence of each subformula $A'$ of $A$ in $\Delta'$ or $\Gamma'$ is at most one (i.e., if $A'$ occurs in $\Delta'$ (resp. $\Gamma'$), then there is no other $A'$ in $\Delta'$ (resp. $\Gamma'$)).

\end{enumerate}
\end{definition}

One may think of the purity condition as the \emph{two-sided generalisation} of the intuitionistic restriction since every element on the right-hand side of a sequent in the intuitionistic restriction of a sequent calculus is pure. 
Because every subformula occurrence of a pure occurrence in a formal proof is also pure, we sometimes emphasise this property by adding the adverb \emph{hereditarily}.

\begin{definition}[\textsf{ILC$_{\boldsymbol{\rho}}$} for ILL$^{\text{e}}_{\boldsymbol{\rho}}$]
\label{DefILC_Delta}
A formal proof $p$ in the sequent calculus \textsf{ILC}$_\iota$ is \emph{tractable} if it satisfies 
\begin{enumerate}

\item There is no occurrence of the rule $\wn\oc$R$^{\oc\wn}$ in $p$;

\item For each occurrence of the rule
\begin{mathpar}
\AxiomC{$q$}
\noLine
\UnaryInfC{$\Delta \vdash A, \Gamma$}
\AxiomC{$q'$}
\noLine
\UnaryInfC{$\Delta', A \vdash \Gamma'$}
\LeftLabel{\textsc{(Cut)}}
\BinaryInfC{$\Delta, \Delta' \vdash \Gamma, \Gamma'$}
\DisplayProof
\end{mathpar}
in $p$, the last occurrence $A$ in $q$ or the last occurrence $A$ in $q'$ is pure.

\end{enumerate}

\emph{\textsf{ILC}$_\rho$} is the sequent calculus obtained from \textsf{ILC}$_\iota$ by restricting formal proofs to tractable ones, and \emph{intuitionistic linear logic $\rho$-extended (ILL$^{\text{e}}_\rho$)} is the substructural logic of ILL$^{\text{e}}_\iota$ embodied by \textsf{ILC}$_\rho$, where the formal language of ILL$^{\text{e}}_\iota$ remains unchanged. 
\end{definition}

\begin{remark}
We can drop the first axiom of Definition~\ref{DefILC_Delta} without any problem, but this axiom makes the rest of this article simpler. 
\end{remark}

By Theorem~\ref{ThmCutEliminationForILC}, \textsf{ILC}$_\rho$ is at least as powerful as \textsf{ILC}. 
And crucially, we can restrict ILL$^{\text{e}}_\iota$ to ILL$^{\text{e}}_\rho$ for the translations of the cut-free fragments of \textsf{INC} and \textsf{CLC} given by Lemmata~\ref{LemTranslationOfILK_muIntoILLK_mu} and \ref{LemTranslationOfCLLKintoILLK}, respectively.

The way of restricting the cut-rule in \textsf{ILC}$_\rho$ is similar to the case of \textsf{LU} \cite[Figure~1]{girard1993unity}. 
This restriction enables us to prove the main point of \textsf{ILC}$_\rho$:
\begin{theorem}[Cut-elimination for \textsf{ILC}$_{\boldsymbol{\rho}}$]
\label{ThmCutEliminationForILC_Delta}
Given a formal proof of a sequent in \textsf{ILC}$_\rho$, there is a formal proof of this sequent in \textsf{ILC}$_\rho$ without Cut.
\end{theorem}
\begin{proof}
By Theorem~\ref{ThmCutEliminationForILC} and the first axiom of Definition~\ref{DefILC_Delta}, it suffices to extend the cut-elimination for \textsf{ILC} (\S\ref{CutElimination}) to $\oc\wn$L$^{\oc\wn}$, and show that the resulting, extended cut-elimination preserves tractable formal proofs in \textsf{ILC$_\iota$}.  

Before going into details, note that the cut-elimination for \textsf{ILC} eliminates occurrences of Cut in a formal proof in the order from the topmost occurrences with respect to their \emph{depths} among the ones with the highest \emph{rank}. 
Roughly, the depth of an occurrence of Cut in a formal proof is the length of the longest branch from the occurrence to a leaf, and the rank of the occurrence is the syntactic complexity of its cut formula.
See \S\ref{CutElimination} for their precise definitions.

Then, for extending the cut-elimination for \textsf{ILC} to $\oc\wn$L$^{\oc\wn}$, the only nontrivial case is an occurrence of Cut in the following form (since the other cases can be handled as in the case of the cut-elimination for \textsf{ILC} given in \S\ref{RightMinorCuts}--\ref{IdentityCuts}): 
\begin{mathpar}
\AxiomC{$p$}
\noLine
\UnaryInfC{$\oc \Delta \vdash \wn A, \wn \Gamma$}
\LeftLabel{\textsc{($\oc$R$^{\oc\wn}$)}}
\UnaryInfC{$\oc \Delta \vdash \oc \wn A, \wn \Gamma$}
\AxiomC{$p'$}
\noLine
\UnaryInfC{$\oc \Delta', \oc \wn A^n, \oc A \vdash \wn \Gamma'$}
\LeftLabel{\textsc{($\oc\wn$L$^{\oc\wn}$)}}
\UnaryInfC{$\oc \Delta', \oc \wn A^{n+1} \vdash \wn \Gamma'$}
\LeftLabel{\textsc{(CutL$^{n+1}$)}}
\RightLabel{\textsc{($n \in \mathbb{N}$)}}
\BinaryInfC{$\oc \Delta^{n+1}, \oc \Delta' \vdash \wn \Gamma^{n+1}, \wn \Gamma'$}
\DisplayProof
\end{mathpar}
where let us write $f$ for the entire formal proof, and CutL$^{m}$ for each $m \in \mathbb{N}$ is an auxiliary rule
\begin{mathpar}
\AxiomC{$\Sigma \vdash B, \Pi$}
\AxiomC{$\Sigma', B^{m} \vdash \Pi'$}
\LeftLabel{\textsc{(CutL$^{m}$)}}
\BinaryInfC{$\Sigma^{m}, \Sigma' \vdash \Pi^{m}, \Pi'$}
\DisplayProof
\end{mathpar}
which is derivable in \textsf{ILC}$_\rho$ by consecutive $m$ times applications of Cut (n.b., we `do nothing' on the right hypothesis if $m = 0$).
This integration of multiple occurrences of Cut into CutL$^{n+1}$ is a standard technique for the elimination of Cut whose cut formula is the principal formula of contraction \cite{gentzen1935untersuchungen,troelstra2000basic}.

We proceed by a case analysis on if the last occurrence $\oc \wn A$ in $p$ is pure.
Let us first assume that the occurrence $\oc \wn A$ is pure.
In this case, we transform $f$ into an auxiliary derivation
\begin{mathpar}
\AxiomC{$p$}
\noLine
\UnaryInfC{$\oc \Delta \vdash \wn A, \wn \Gamma$}
\AxiomC{$p$}
\noLine
\UnaryInfC{$\oc \Delta \vdash \wn A, \wn \Gamma$}
\LeftLabel{\textsc{($\oc$R$^{\oc\wn}$)}}
\UnaryInfC{$\oc \Delta \vdash \oc \wn A, \wn \Gamma$}
\AxiomC{$p'$}
\noLine
\UnaryInfC{$\oc \Delta', \oc \wn A^{n}, \oc A \vdash \wn \Gamma'$}
\LeftLabel{\textsc{(CutL$^n$)}}
\BinaryInfC{$\oc \Delta^n, \oc \Delta', \oc A \vdash \wn \Gamma^n, \wn \Gamma'$}
\LeftLabel{\textsc{(Cut$_{\wn\oc}$)}}
\BinaryInfC{$\oc \Delta^{n+1}, \oc \Delta' \vdash \wn \Gamma^{n+1}, \wn \Gamma'$}
\DisplayProof
\end{mathpar}
which let us write $\tilde{f}$, where Cut$_{\wn\oc}$ is an auxiliary rule 
\begin{mathpar}
\AxiomC{$\Sigma \vdash \wn B, \Pi$}
\AxiomC{$\Sigma', \oc B \vdash \Pi'$}
\LeftLabel{\textsc{(Cut$_{\wn\oc}$)}}
\BinaryInfC{$\Sigma, \Sigma' \vdash \Pi, \Pi'$}
\DisplayProof
\end{mathpar}

We regard this rule as Cut whose rank is that of $\wn B$, which is equal to that of $\oc B$ (\S\ref{CutElimination}), so that we extend the cut-elimination procedure for \textsf{ILC} to this rule accordingly.
By the order of occurrences of Cut that the cut-elimination procedure eliminates sketched above, this extended cut-elimination procedure next proceeds to the right hypothesis of the last occurrence of Cut$_{\wn\oc}$ in $\tilde{f}$ (since $\oc \wn A$ is more complex than $\wn A$, i.e., the rank of $\oc \wn A$ is higher than that of $\wn A$). 

By the inductive argument of the cut-elimination for \textsf{ILC} (\S\ref{CutElimination}), we may assume that the cut-elimination on the right hypothesis has eliminated the last occurrence of CutL$^n$, and the cut formula $\oc A$ on the right-hand side of the occurrence of Cut$_{\wn\oc}$ is now the principal formula.
We then eliminate the occurrence of Cut$_{\wn\oc}$ by the following case analysis on the principal formula $\oc A$:
\begin{itemize}

\item If the principal formula $\oc A$ is given by the rule $\oc$D, then, because the last occurrence $\oc \wn A$ in $p$ is pure, $\tilde{f}$ is reduced to the only nontrivial case
\begin{mathpar}
\AxiomC{$\Theta \vdash A, \Xi$}
\LeftLabel{\textsc{($\wn$D)}}
\UnaryInfC{$\Theta \vdash \wn A, \Xi$}
\AxiomC{$\Theta', \oc A^n, A \vdash \Xi'$}
\LeftLabel{\textsc{($\oc$D)}}
\UnaryInfC{$\Theta', \oc A^{n+1} \vdash \Xi'$}
\LeftLabel{\textsc{(CutL$_{\wn\oc}^{n+1}$)}}
\BinaryInfC{$\Theta^{n+1}, \Theta' \vdash \Xi^{n+1}, \Xi'$}
\DisplayProof
\end{mathpar}
where CutL$_{\wn\oc}^{m}$ is analogous to CutL$^{m}$ ($m \in \mathbb{N}$), which we transform into
\begin{mathpar}
\AxiomC{$\Theta \vdash A, \Xi$}
\AxiomC{$\Theta \vdash A, \Xi$}
\LeftLabel{\textsc{($\wn$D)}}
\UnaryInfC{$\Theta \vdash \wn A, \Xi$}
\AxiomC{$\Theta', \oc A^n, A \vdash \Xi'$}
\LeftLabel{\textsc{(CutL$_{\wn\oc}^{n}$)}}
\BinaryInfC{$\Theta^n, \Theta', A \vdash \Xi^n, \Xi'$}
\LeftLabel{\textsc{(Cut)}}
\BinaryInfC{$\Theta^{n+1}, \Theta' \vdash \Xi^{n+1}, \Xi'$}
\DisplayProof
\end{mathpar}
Note that this derived rule for $n = 0$ is
\begin{mathpar}
\AxiomC{$\Theta \vdash A, \Xi$}
\AxiomC{$\Theta', A \vdash \Xi'$}
\LeftLabel{\textsc{(Cut)}}
\BinaryInfC{$\Theta, \Theta' \vdash \Xi, \Xi'$}
\DisplayProof
\end{mathpar}

\item If the principal formula $\oc A$ is given by the rule $\oc$W, then $\tilde{f}$ is reduced to the only nontrivial case
\begin{mathpar}
\AxiomC{$\oc \Delta \vdash \wn A, \wn \Gamma$}
\AxiomC{$\Theta', \oc A^n \vdash \Xi'$}
\LeftLabel{\textsc{($\oc$W)}}
\UnaryInfC{$\Theta', \oc A^{n+1} \vdash \Xi'$}
\LeftLabel{\textsc{(CutL$_{\wn\oc}^{n+1}$)}}
\BinaryInfC{$\oc \Delta^{n+1}, \Theta' \vdash \wn \Gamma^{n+1}, \Xi'$}
\DisplayProof
\end{mathpar}
which we transform into
\begin{mathpar}
\AxiomC{$\oc \Delta \vdash \wn A, \wn \Gamma$}
\AxiomC{$\Theta', \oc A^n \vdash \Xi'$}
\LeftLabel{\textsc{(CutL$_{\wn\oc}^{n}$)}}
\BinaryInfC{$\oc \Delta^{n}, \Theta' \vdash \wn \Gamma^n, \Xi'$}
\LeftLabel{\textsc{($\oc$W$^\ast$, XL$^\ast$)}}
\doubleLine
\UnaryInfC{$\oc \Delta^{n+1}, \Theta' \vdash \wn \Gamma^n, \Xi'$}
\LeftLabel{\textsc{($\wn$W$^\ast$)}}
\doubleLine
\UnaryInfC{$\oc \Delta^{n+1}, \Theta' \vdash \wn \Gamma^{n+1}, \Xi'$}
\DisplayProof
\end{mathpar}

\item If the principal formula $\oc A$ is given by the rule $\oc$C, then $\tilde{f}$ is reduced to the only nontrivial case
\begin{mathpar}
\AxiomC{$\oc \Delta \vdash \wn A, \wn \Gamma$}
\AxiomC{$\Theta', \oc A^n, \oc A, \oc A \vdash \Xi'$}
\LeftLabel{\textsc{($\oc$C)}}
\UnaryInfC{$\Theta', \oc A^{n+1} \vdash \Xi'$}
\LeftLabel{\textsc{(CutL$_{\wn\oc}^{n+1}$)}}
\BinaryInfC{$\oc \Delta^{n+1}, \Theta' \vdash \wn \Gamma^{n+1}, \Xi'$}
\DisplayProof
\end{mathpar}
which we transform into
\begin{mathpar}
\AxiomC{$\oc \Delta \vdash \wn A, \wn \Gamma$}
\AxiomC{$\Theta', \oc A^n, \oc A, \oc A \vdash \Xi'$}
\LeftLabel{\textsc{(CutL$_{\wn\oc}^{n+2}$)}}
\BinaryInfC{$\oc \Delta^{n+2}, \Theta' \vdash \wn \Gamma^{n+2}, \Xi'$}
\LeftLabel{\textsc{($\oc$C$^\ast$, XL$^\ast$)}}
\doubleLine
\UnaryInfC{$\oc \Delta^{n+1}, \Theta' \vdash \wn \Gamma^{n+2}, \Xi'$}
\LeftLabel{\textsc{($\wn$C$^\ast$)}}
\doubleLine
\UnaryInfC{$\oc \Delta^{n+1}, \Theta' \vdash \wn \Gamma^{n+1}, \Xi'$}
\DisplayProof
\end{mathpar}

\end{itemize}

In this way, the additional rule Cut$_{\wn\oc}$ is incorporated into the cut-elimination for \textsf{ILC} just like another instance of Cut, and the same inductive argument as that of \S\ref{CutElimination} is applicable to this extended cut-elimination procedure. 
We have completed the case where the last occurrence $\oc\wn A$ in $p$ is pure.

Next, let us consider the other case, i.e., when the last occurrence $\oc \wn A$ in $p$ is not pure. 
In this case, $f$ must be of the form
\begin{mathpar}
\AxiomC{$p$}
\noLine
\UnaryInfC{$\oc \Delta \vdash \wn A, \wn \Gamma$}
\LeftLabel{\textsc{($\oc$R$^{\oc\wn}$)}}
\UnaryInfC{$\oc \Delta \vdash \oc \wn A, \wn \Gamma$}
\AxiomC{$p'$}
\noLine
\UnaryInfC{$\oc \Delta', \oc A \vdash \wn \Gamma'$}
\LeftLabel{\textsc{($\oc\wn$L$^{\oc\wn}$)}}
\UnaryInfC{$\oc \Delta', \oc \wn A \vdash \wn \Gamma'$}
\LeftLabel{\textsc{(Cut)}}
\BinaryInfC{$\oc \Delta \vdash \wn \Gamma, \wn \Gamma'$}
\DisplayProof
\end{mathpar}
where the last occurrence $\oc \wn A$ in $p'$ is pure, by the second axiom of Definition~\ref{DefILC_Delta}.
Then, we transform $f$ into
\begin{mathpar}
\AxiomC{$p$}
\noLine
\UnaryInfC{$\oc \Delta \vdash \wn A, \wn \Gamma$}
\AxiomC{$p'$}
\noLine
\UnaryInfC{$\oc \Delta', \oc A \vdash \wn \Gamma'$}
\LeftLabel{\textsc{(Cut$_{\wn\oc}$)}}
\BinaryInfC{$\oc \Delta, \oc \Delta' \vdash \wn \Gamma, \wn \Gamma'$}
\DisplayProof
\end{mathpar}
By the argument symmetric to the elimination of CutL$^n$ and CutL$^n_{\wn\oc}$ described above, we eliminate this occurrence of Cut$_{\wn\oc}$ with the help of the auxiliary rules to be denoted by CutR$^n$ and CutR$^n_{\wn\oc}$. 
Dually, it is crucial that the last occurrence $\oc A$ in $p'$ is pure, so that this cut elimination completes too.

Finally, the resulting, extended cut-elimination procedure clearly preserves tractable formal proofs, completing the proof.
\end{proof}

Note that the second axiom of Definition~\ref{DefILC_Delta} plays crucial roles in the proof of Theorem~\ref{ThmCutEliminationForILC_Delta}.
For instance, the formal proof of the sequent $\oc (X \invamp X) \vdash \wn (X \otimes X)$ in \textsf{ILC}$_\iota$ given in \S\ref{ConservativeExtensionOfIntuitionisticLinearLogic}, for which cut-elimination is impossible, is not tractable or valid in \textsf{ILC}$_\rho$ because it does not satisfy the second axiom. 

Next, let us similarly carve out a substructural sequent calculus of \textsf{INC}, which we can translate into \textsf{ILC$_\rho$}:
\begin{definition}[Pure occurrences in \textsf{INC}]
\label{DefPureOccurrencesII}
An occurrence of a formula $A$ in a sequent $\Delta \vdash \Gamma$ in a formal proof $p$ in \textsf{INC} is \emph{(hereditarily) pure} if each sequent $\Delta' \vdash \Gamma'$ occurring in the subtree of $p$ whose root is $\Delta \vdash \Gamma$ satisfies

\begin{enumerate}

\item There is no application of WL, CL, $\wn$W or $\wn$C on subformulas of $A$ occurring in $\Delta' \vdash \Gamma'$ that constitute the occurrence $A$ in $\Delta \vdash \Gamma$;

\item An occurrence of each subformula $A'$ of $A$ in $\Delta'$ or $\Gamma'$ is at most one (i.e., if $A'$ occurs in $\Delta'$ (resp. $\Gamma'$), then there is no other $A'$ in $\Delta'$ (resp. $\Gamma'$)).

\end{enumerate}
\end{definition}

\begin{definition}[\textsf{INC$_{\boldsymbol{\rho}}$} for IL$^{\text{e}}$]
\label{DefINC_Gamma}
A formal proof $p$ in the sequent calculus \textsf{INC} is \emph{tractable} if, for each occurrence of the rule
\begin{mathpar}
\AxiomC{$q$}
\noLine
\UnaryInfC{$\Delta \vdash \wn A, \wn \Gamma$}
\AxiomC{$q'$}
\noLine
\UnaryInfC{$\Delta', A \vdash \wn \Gamma'$}
\LeftLabel{\textsc{(Cut$^\wn$)}}
\BinaryInfC{$\Delta, \Delta' \vdash \wn \Gamma, \wn \Gamma'$}
\DisplayProof
\end{mathpar}
in $p$, the last occurrence $\wn A$ in $q$ or the last occurrence $A$ in $q'$ is pure.

\emph{\textsf{INC}$_\rho$} is the sequent calculus obtained out of \textsf{INC} by restricting formal proofs to tractable ones.
\end{definition}

\textsf{INC}$_\rho$ embodies IL$^{\text{e}}$ since \textsf{INC} enjoys cut-elimination (Lemma~\ref{LemCutEliminationForINC_mu}).
Also, \textsf{INC}$_\rho$ enjoys cut-elimination since the cut-elimination procedure for \textsf{INC} in the proofs of Lemma~\ref{LemCutEliminationForINC_mu} preserves tractable formal proofs.

Moreover, the restriction of the translation $\mathscr{T}_\oc : \textsf{INC} \rightarrow \textsf{ILC}_\iota$ (Lemma~\ref{LemTranslationOfILK_muIntoILLK_mu}) to \textsf{INC}$_\rho$ defines a \emph{conservative} translation $\textsf{INC}_\rho \rightarrow \textsf{ILC}_\rho$ verified by a technique analogous to the one used for Lemma~\ref{LemTranslationOfLKIntoILK_mu} thanks to Theorem~\ref{ThmCutEliminationForILC_Delta} (see the proof of Corollary~\ref{CorMainCorollary}). 
We write $\mathscr{T}_\oc : \textsf{INC}_\rho \rightarrow \textsf{ILC}_\rho$ for this translation.

Similarly, we carve out a substructural sequent calculus of \textsf{CLC}, which we can translate into \textsf{ILC$_\rho$}:
\begin{definition}[Pure occurrences in \textsf{CLC}]
\label{DefPureOccurrencesIII}
An occurrence of a formula $A$ in a sequent $\Delta \vdash \Gamma$ in a formal proof $p$ in \textsf{CLC} is \emph{(hereditarily) pure} if each sequent $\Delta' \vdash \Gamma'$ occurring in the subtree of $p$ whose root is $\Delta \vdash \Gamma$ satisfies

\begin{enumerate}

\item There is no application of $\oc$W, $\oc$C, WR or CR on subformulas of $A$ occurring in $\Delta' \vdash \Gamma'$ that constitute the occurrence $A$ in $\Delta \vdash \Gamma$;

\item An occurrence of each subformula $A'$ of $A$ in $\Delta'$ or $\Gamma'$ is at most one (i.e., if $A'$ occurs in $\Delta'$ (resp. $\Gamma'$), then there is no other $A'$ in $\Delta'$ (resp. $\Gamma'$)).

\end{enumerate}
\end{definition}

\begin{definition}[\textsf{CLC$_{\boldsymbol{\rho}}$} for CLL$^{\boldsymbol{-}}$]
\label{DefLK_Gamma}
A formal proof $p$ in the sequent calculus \textsf{CLC} is \emph{tractable} if, for each occurrence of the rule
\begin{mathpar}
\AxiomC{$q$}
\noLine
\UnaryInfC{$\oc \Delta \vdash A, \Gamma$}
\AxiomC{$q'$}
\noLine
\UnaryInfC{$\oc \Delta', \oc A \vdash \Gamma'$}
\LeftLabel{\textsc{(Cut$^\oc$)}}
\BinaryInfC{$\oc \Delta, \oc \Delta' \vdash \Gamma, \Gamma'$}
\DisplayProof
\end{mathpar}
in $p$, the last occurrence $A$ in $q$ or the last occurrence $\oc A$ in $q'$ is pure.

\emph{\textsf{CLC}$_\rho$} is the sequent calculus obtained out of \textsf{CLC} by restricting formal proofs to tractable ones.
\end{definition}

Similarly to \textsf{INC}$_\rho$, \textsf{CLC}$_\rho$ embodies CLL$^-$ and enjoys cut-elimination.
Also, the restriction of the translation $\mathscr{T}_\wn : \textsf{CLC} \rightarrow \textsf{ILC}_\iota$ (Lemma~\ref{LemTranslationOfCLLKintoILLK}) to \textsf{CLC}$_\rho$ defines a \emph{conservative} translation of \textsf{CLC}$_\rho$ into \textsf{ILC}$_\rho$, which we also write $\mathscr{T}_\wn$.

Finally, we similarly carve out a substructural sequent calculus of \textsf{LK}, which we can translate into \textsf{INC$_\rho$} and \textsf{CLC}$_\rho$:
\begin{definition}[Pure occurrences in \textsf{LK}]
\label{DefPureOccurrencesIII}
An occurrence of a formula $A$ in a sequent $\Delta \vdash \Gamma$ in a formal proof $p$ in \textsf{LK} is \emph{(hereditarily) pure} if each sequent $\Delta' \vdash \Gamma'$ occurring in the subtree of $p$ whose root is $\Delta \vdash \Gamma$ satisfies

\begin{enumerate}

\item There is no application of WL, CL, WR or CR on subformulas of $A$ occurring in $\Delta' \vdash \Gamma'$ that constitute the occurrence $A$ in $\Delta \vdash \Gamma$;

\item An occurrence of each subformula $A'$ of $A$ in $\Delta'$ or $\Gamma'$ is at most one (i.e., if $A'$ occurs in $\Delta'$ (resp. $\Gamma'$), then there is no other $A'$ in $\Delta'$ (resp. $\Gamma'$)).

\end{enumerate}
\end{definition}

\begin{definition}[\textsf{LK$_{\boldsymbol{\rho}}$} for CL]
\label{DefLK_Gamma}
A formal proof $p$ in the sequent calculus \textsf{LK} is \emph{tractable} if, for each occurrence of the rule
\begin{mathpar}
\AxiomC{$q$}
\noLine
\UnaryInfC{$\Delta \vdash A, \Gamma$}
\AxiomC{$q'$}
\noLine
\UnaryInfC{$\Delta', A \vdash \Gamma'$}
\LeftLabel{\textsc{(Cut)}}
\BinaryInfC{$\Delta, \Delta' \vdash \Gamma, \Gamma'$}
\DisplayProof
\end{mathpar}
in $p$, the last occurrence $A$ in $q$ or the last occurrence $A$ in $q'$ is pure.

\emph{\textsf{LK}$_\rho$} is the sequent calculus obtained out of \textsf{LK} by restricting formal proofs to tractable ones.
\end{definition}

Again, \textsf{LK}$_\rho$ embodies CL and enjoys cut-elimination.
Further, the restriction of the conservative translation $\mathscr{T}_\wn : \textsf{LK} \rightarrow \textsf{INC}$ given in Lemma~\ref{LemTranslationOfLKIntoILK_mu} (resp. $\mathscr{T}_\oc : \textsf{LK} \rightarrow \textsf{CLC}$ given in Lemma~\ref{LemTranslationOfLKintoCLLK}) to \textsf{LK}$_\rho$ defines a conservative translation of \textsf{LK}$_\rho$ into \textsf{INC}$_\rho$ (resp. \textsf{LK}$_\rho$ into \textsf{CLC}$_\rho$), which we also write $\mathscr{T}_\wn$ (resp. $\mathscr{T}_\oc$).

Summarising the results obtained so far in this last subsection, we have finally established the corollary articulated in \S\ref{MainResults}:
\begin{corollary}[Conservative translations]
\label{CorMainCorollary}
The following holds:
\begin{enumerate}

\item The restrictions of the translations $\textsf{LK} \stackrel{\mathscr{T}_\wn}{\rightarrow} \textsf{INC} \stackrel{\mathscr{T}_\oc}{\rightarrow} \textsf{ILC}_\iota$ to \textsf{LK}$_\rho$, \textsf{INC}$_\rho$ and \textsf{ILC}$_\rho$ give rise to conservative translations $\textsf{LK}_\rho \stackrel{\mathscr{T}_\wn}{\rightarrow} \textsf{INC}_\rho \stackrel{\mathscr{T}_\oc}{\rightarrow} \textsf{ILC}_\rho$;

\item The restrictions of the translations $\textsf{LK} \stackrel{\mathscr{T}_\oc}{\rightarrow} \textsf{CLC} \stackrel{\mathscr{T}_\wn}{\rightarrow} \textsf{ILC}_\iota$ to \textsf{LK}$_\rho$, \textsf{CLC}$_\rho$ and \textsf{ILC}$_\rho$ give rise to conservative translations $\textsf{LK}_\rho \stackrel{\mathscr{T}_\oc}{\rightarrow} \textsf{CLC}_\rho \stackrel{\mathscr{T}_\wn}{\rightarrow} \textsf{ILC}_\rho$;

\item Given a formal proof $p$ of a sequent $\Delta \vdash \Gamma$ in \textsf{LK}$_\rho$, the formal proofs $\mathscr{T}_{\oc\wn}(p)$ and $\mathscr{T}_{\wn\oc}(p)$ of the sequent $\oc \mathscr{T}_{\oc\wn}^\ast(\Delta) \vdash \wn \mathscr{T}_{\oc\wn}^\ast(\Gamma)$ in \textsf{ILC$_\rho$} coincide modulo permuting axioms and rules in formal proofs in \textsf{ILC$_\rho$}. 

\end{enumerate}
\end{corollary}
\begin{proof}
We focus on the conservativity of $\mathscr{T}_\oc : \textsf{INC}_\rho \rightarrow \textsf{ILC}_\rho$ since the case of $\mathscr{T}_\wn : \textsf{CLC}_\rho \rightarrow \textsf{ILC}_\rho$ is symmetric, and the other points of the corollary follow from our preceding results.
By Theorem~\ref{ThmCutEliminationForILC_Delta}, we can show the conservativity of $\mathscr{T}_\oc$ by the method similar to the last part of the proof of Lemma~\ref{LemTranslationOfLKIntoILK_mu}.
Here, it is crucial to observe that if a sequent $\oc \mathscr{T}_\oc^\ast(\Delta), [\mathscr{T}_\oc(A)] \vdash \mathscr{T}_\oc^\ast(\Gamma)$ is provable in $\textsf{ILC}_\rho$, and only formulas in IL$^{\text{e}}$ occur in $\Delta, [A], \Gamma$, then $\mathscr{T}_\oc^\ast(\Gamma)$ must be of the form $[B], \wn \Gamma'$ since in this case an increment of the number of elements on the right-hand side of sequents in $\textsf{ILC}_\rho$ is possible only by $\wn$W or ($\bot$R, $\oc$R$^{\oc\wn}$).
\end{proof}

The last part of the proof of Corollary~\ref{CorMainCorollary} clarifies how the sequent calculus $(\textsf{ILC}_\rho)_\oc$ \emph{automatically satisfies} the form $[B], \wn \Gamma'$ required on the right-hand side of sequents in the sequent calculi \textsf{INC} and \textsf{INC}$_{\rho}$ (\S\ref{ConservativeExtensionOfIntuitionisticLogic}) if we focus on the formulas of IL$^{\text{e}}$.
As announced in Remark~\ref{RemarkOnIntuitionisity}, this point explains the general mechanism underlying the standard yet \emph{ad-hoc} intuitionistic restriction on CL to obtain IL (Definition~\ref{DefLJ}): Why-not $\wn$ does not occur in the formal language of IL, and hence $[B], \wn \Gamma'$ must be $[B]$ if we focus on the formulas of IL.

\begin{remark}
Our sequent calculi introduced in this section are \emph{undirected} (on cut-elimination), but we can design their \emph{directed} variants as follows.
First, observe that for the present work Cut occurring in formal proofs in \textsf{ILC}$_\rho$ such that the cut-formula on the left hypothesis is not pure is only necessary for the translation of the rule $\oplus$R in \text{CLC} into \textsf{ILC}$_\rho$ in the proof of Lemma~\ref{LemTranslationOfCLLKintoILLK}. 
Next, Corollary~\ref{CorMainCorollary} can dispense with this kind of Cut by advancing occurrences of the right rule on disjunction over those of structural rules of \textsf{LK}$_\rho$, \textsf{CLC}$_\rho$, \textsf{INC}$_\rho$ and \textsf{ILC}$_\rho$.
It follows from these two points that we can focus on Cut whose cut-formula on the left hypothesis is pure, so that we only need the \emph{rightward} cut-elimination, or the \emph{leftward} one by symmetry. 
However, the resulting sequent calculi, in which occurrences of the right rule on disjunction in formal proofs are always advanced over those of structural rules, would be involved. 
Hence, we leave it as future work to pursue this idea more appropriately in terms of \emph{term calculi}, perhaps in a style similar to Wadler's \emph{dual calculus} \cite{wadler2003call}.
\end{remark}

\bibliographystyle{spmpsci}
\bibliography{RecursionTheory,HoTT,GamesAndStrategies,LinearLogic,CategoricalLogic}

\begin{thebibliography}{10}
\providecommand{\url}[1]{{#1}}
\providecommand{\urlprefix}{URL }
\expandafter\ifx\csname urlstyle\endcsname\relax
  \providecommand{\doi}[1]{DOI~\discretionary{}{}{}#1}\else
  \providecommand{\doi}{DOI~\discretionary{}{}{}\begingroup
  \urlstyle{rm}\Url}\fi

\bibitem{abramsky1993computational}
Abramsky, S.: Computational interpretations of linear logic.
\newblock Theor. Comput. Sci. \textbf{111}(1\&2), 3--57 (1993)

\bibitem{backus1959syntax}
Backus, J.W.: The syntax and semantics of the proposed international algebraic
  language of the zurich acm-gamm conference.
\newblock Proceedings of the International Comference on Information
  Processing, 1959  (1959)

\bibitem{bierman1994intuitionistic}
Bierman, G.M.: On intuitionistic linear logic.
\newblock Tech. rep., Citeseer (1994)

\bibitem{blot2017realizability}
Blot, V.: Realizability for {P}eano arithmetic with winning conditions in {HON}
  games.
\newblock Annals of Pure and Applied Logic \textbf{168}(2), 254--277 (2017)

\bibitem{coquand1995semantics}
Coquand, T.: A semantics of evidence for classical arithmetic.
\newblock The Journal of Symbolic Logic \textbf{60}(1), 325--337 (1995)

\bibitem{danos1997new}
Danos, V., Joinet, J.B., Schellinx, H.: A new deconstructive logic: Linear
  logic.
\newblock Journal of Symbolic Logic pp. 755--807 (1997)

\bibitem{filinski1989declarative}
Filinski, A.: Declarative continuations: An investigation of duality in
  programming language semantics.
\newblock In: Category Theory and Computer Science, pp. 224--249. Springer
  (1989)

\bibitem{frege1879begriffsschrift}
Frege, G.: Begriffsschrift, eine der arithmetischen nachgebildete Formelsprache
  des reinen Denkens.
\newblock Nebert (1879)

\bibitem{gentzen1969widerspruchsfreiheit}
Gentzen, G.: Die widerspruchsfreiheit der reinen zahlentheorie, mathematicshe
  annalen 112 (1936), 493--565.
\newblock Translated in [1969] below pp. 132--213

\bibitem{gentzen1935untersuchungen}
Gentzen, G.: Untersuchungen {\"u}ber das logische schlie{\ss}en. i.
\newblock Mathematische zeitschrift \textbf{39}(1), 176--210 (1935)

\bibitem{girard1987linear}
Girard, J.Y.: Linear logic.
\newblock Theoretical Computer Science \textbf{50}(1), 1--101 (1987)

\bibitem{girard1991new}
Girard, J.Y.: A new constructive logic: classic logic.
\newblock Mathematical Structures in Computer Science \textbf{1}(3), 255--296
  (1991)

\bibitem{girard1993unity}
Girard, J.Y.: On the unity of logic.
\newblock Annals of pure and applied logic \textbf{59}(3), 201--217 (1993)

\bibitem{girard1995linear}
Girard, J.Y.: Linear logic: its syntax and semantics.
\newblock London Mathematical Society Lecture Note Series pp. 1--42 (1995)

\bibitem{girard2011blind}
Girard, J.Y.: The Blind Spot: lectures on logic.
\newblock European Mathematical Society (2011)

\bibitem{godel1933intuitionistischen}
G{\"o}del, K.: Zur intuitionistischen arithmetik und zahlentheorie.
\newblock Ergebnisse eines mathematischen Kolloquiums \textbf{4}(1933), 34--38
  (1933)

\bibitem{griffin1989formulae}
Griffin, T.G.: A formulae-as-type notion of control.
\newblock In: Proceedings of the 17th ACM SIGPLAN-SIGACT symposium on
  Principles of programming languages, pp. 47--58 (1989)

\bibitem{harmer2007categorical}
Harmer, R., Hyland, M., Mellies, P.A.: Categorical combinatorics for innocent
  strategies.
\newblock In: 22nd annual IEEE symposium on logic in computer science, pp.
  pp--379 (2007)

\bibitem{herbelin1997games}
Herbelin, H.: Games and weak-head reduction for classical {PCF}.
\newblock In: International Conference on Typed Lambda Calculi and
  Applications, pp. 214--230. Springer (1997)

\bibitem{herbelin2000duality}
Herbelin, H., Curien, P.L.: The duality of computation.
\newblock In: Fifth ACM SIGPLAN International Conference on Functional
  Programming: ICFP'00, vol.~35, pp. 233--243. ACM (2000)

\bibitem{heyting1930formalen}
Heyting, A.: Die formalen regeln der intuitionistischen logik.
\newblock Sitzungsbericht PreuBische Akademie der Wissenschaften Berlin,
  physikalisch-mathematische Klasse II pp. 42--56 (1930)

\bibitem{hyland1997game}
Hyland, M.: Game semantics.
\newblock In: Semantics and Logics of Computation, vol.~14, p. 131. Cambridge
  University Press, New York (1997)

\bibitem{kolmogorov1925principe}
Kolmogorov, A.: O principe tertium non datur. mathematicheskij sbornik 32:
  646--667.
\newblock English trans. in van Heijenoort [1967, 414-437]  (1925)

\bibitem{lafont1988introduction}
Lafont, Y.: Introduction to linear logic.
\newblock Lecture Notes of the School on Constructive Logics and Category
  Theory  (1988)

\bibitem{laird1999semantic}
Laird, J.D.: A semantic analysis of control.
\newblock Ph.D. thesis, University of Edinburgh (1999)

\bibitem{lambek1988introduction}
Lambek, J., Scott, P.J.: Introduction to Higher-order Categorical Logic,
  vol.~7.
\newblock Cambridge University Press (1988)

\bibitem{laurent2002polarized}
Laurent, O.: Polarized games.
\newblock In: Proceedings of the 17th Annual IEEE Symposium on Logic in
  Computer Science, pp. 265--274. IEEE (2002)

\bibitem{laurent2003translations}
Laurent, O., Regnier, L.: About translations of classical logic into polarized
  linear logic.
\newblock In: 18th Annual IEEE Symposium of Logic in Computer Science, 2003.
  Proceedings., pp. 11--20. IEEE (2003)

\bibitem{mccusker1998games}
McCusker, G.: Games and Full Abstraction for a Functional Metalanguage with
  Recursive Types.
\newblock Springer Science \& Business Media, London (1998)

\bibitem{mellies2009categorical}
Mellies, P.A.: Categorical semantics of linear logic.
\newblock Panoramas et syntheses \textbf{27}, 15--215 (2009)

\bibitem{mellies2010resource}
Melli{\`e}s, P.A., Tabareau, N.: Resource modalities in tensor logic.
\newblock Annals of Pure Applied Logic \textbf{161}(5), 632--653 (2010)

\bibitem{murthy1991evaluation}
Murthy, C.: An evaluation semantics for classical proofs.
\newblock In: Proceedings 1991 Sixth Annual IEEE Symposium on Logic in Computer
  Science, pp. 96--97 (1991)

\bibitem{parigot1992lambdamu}
Parigot, M.: $\lambda$$\mu$-calculus: an algorithmic interpretation of
  classical natural deduction.
\newblock In: International Conference on Logic for Programming Artificial
  Intelligence and Reasoning, pp. 190--201. Springer (1992)

\bibitem{rogers1987theory}
Rogers~Jr, H.: Theory of recursive functions and effective computability
  (1987)

\bibitem{schellinx1991some}
Schellinx, H.: Some syntactical observations on linear logic.
\newblock Journal of Logic and Computation \textbf{1}(4), 537--559 (1991)

\bibitem{seely1987linear}
Seely, R.A.: Linear logic,*-autonomous categories and cofree coalgebras.
\newblock Ste. Anne de Bellevue, Quebec: CEGEP John Abbott College (1987)

\bibitem{troelstra1991lectures}
Troelstra, A.S.: Lectures on linear logic  (1991)

\bibitem{troelstra1988constructivism}
Troelstra, A.S., van Dalen, D.: Constructivism in mathematics. {V}ol. {I},
  volume 121 of.
\newblock Studies in Logic and the Foundations of Mathematics p.~26 (1988)

\bibitem{troelstra2000basic}
Troelstra, A.S., Schwichtenberg, H.: Basic proof theory.
\newblock 43. Cambridge University Press (2000)

\bibitem{wadler2003call}
Wadler, P.: Call-by-value is dual to call-by-name.
\newblock In: Proceedings of the eighth ACM SIGPLAN international conference on
  Functional programming, pp. 189--201 (2003)

\end{thebibliography}

\appendix
\section{Proof of Corollary~\ref{ILCAsConservativeExtensionOfLLJ}}
\label{ProofOnMu}
In this appendix, we show in detail how to extend the conservativity of \textsf{ILC} over \textsf{LLJ} to \textsf{ILC$_\iota$}.
For this proof, it is convenient to introduce the following axiom:
\begin{mathpar}
\AxiomC{}
\LeftLabel{\textsc{(Dist)}}
\UnaryInfC{$\oc \wn A \vdash \wn \oc A$}
\DisplayProof
\end{mathpar}

Note that this axiom Dist is \emph{derivable} in \textsf{ILC$_\iota$} in two ways:
\begin{mathpar}
\AxiomC{}
\LeftLabel{\textsc{(Id)}}
\UnaryInfC{$A \vdash A$}
\LeftLabel{\textsc{($\oc$D)}}
\UnaryInfC{$\oc A \vdash A$}
\LeftLabel{\textsc{($\oc$R$^{\oc\wn}$)}}
\UnaryInfC{$\oc A \vdash \oc A$}
\LeftLabel{\textsc{($\wn$D)}}
\UnaryInfC{$\oc A \vdash \wn \oc A$}
\LeftLabel{\textsc{($\oc\wn$L$^{\oc\wn}$)}}
\UnaryInfC{$\oc \wn A \vdash \wn \oc A$}
\DisplayProof
\and
\AxiomC{}
\LeftLabel{\textsc{(Id)}}
\UnaryInfC{$A \vdash A$}
\LeftLabel{\textsc{($\wn$D)}}
\UnaryInfC{$A \vdash \wn A$}
\LeftLabel{\textsc{($\wn$L$^{\oc\wn}$)}}
\UnaryInfC{$\wn A \vdash \wn A$}
\LeftLabel{\textsc{($\oc$D)}}
\UnaryInfC{$\oc \wn A \vdash \wn A$}
\LeftLabel{\textsc{($\wn\oc$R$^{\oc\wn}$)}}
\UnaryInfC{$\oc \wn A \vdash \wn \oc A$}
\DisplayProof
\end{mathpar}

Conversely, the weakly distributive rules $\oc\wn$L$^{\oc\wn}$ and $\wn\oc$R$^{\oc\wn}$ are derivable in \textsf{ILC} augmented with the axiom Dist:
\begin{mathpar}
\AxiomC{}
\LeftLabel{\textsc{(Dist)}}
\UnaryInfC{$\oc \wn A \vdash \wn \oc A$}
\AxiomC{$\oc \Delta, \oc A \vdash \wn \Gamma$}
\LeftLabel{\textsc{($\wn$L$^{\oc\wn}$)}}
\UnaryInfC{$\oc \Delta, \wn \oc A \vdash \wn \Gamma$}
\LeftLabel{\textsc{(Cut)}}
\BinaryInfC{$\oc \wn A, \oc \Delta \vdash \wn \Gamma$}
\LeftLabel{\textsc{(XL$^\ast$)}}
\doubleLine
\UnaryInfC{$\oc \Delta, \oc \wn A \vdash \wn \Gamma$}
\DisplayProof
\and
\AxiomC{$\oc \Delta \vdash \wn A, \wn \Gamma$}
\LeftLabel{\textsc{($\oc$R$^{\oc\wn}$)}}
\UnaryInfC{$\oc \Delta \vdash \oc \wn A, \wn \Gamma$}
\AxiomC{}
\LeftLabel{\textsc{(Dist)}}
\UnaryInfC{$\oc \wn A \vdash \wn \oc A$}
\LeftLabel{\textsc{(Cut)}}
\BinaryInfC{$\oc \Delta \vdash \wn \Gamma, \wn \oc A$}
\LeftLabel{\textsc{(XR$^\ast$)}}
\doubleLine
\UnaryInfC{$\oc \Delta \vdash \wn \oc A, \wn \Gamma$}
\DisplayProof
\end{mathpar}

Hence, at the level of provability, we may replace \textsf{ILC$_\iota$} with \textsf{ILC} augmented with the axiom Dist, which let us call \emph{\textsf{ILC$_\delta$}}.
\textsf{ILC$_\delta$} is more suited than \textsf{ILC$_\rho$} to the following argument. 

Now, assume that a sequent $\Delta \vdash \Gamma$ has a proof $p$ in \textsf{ILC$_\delta$}, and only formulas of ILL occur in $\Delta, \Gamma$.
It is easy to see by induction on formal proofs in \textsf{ILC$_\delta$} that once a formula of the form $\wn B$ occurs on the right-hand side of a sequent in $p$, $\Gamma$ must contain $\wn$.
Note, however, that $\Gamma$ does not contain $\wn$ by the assumption, and thus there is no sequent in $p$ that contains a formula of the form $\wn B$ on the right-hand side.
Hence, there is no application of Dist in $p$, i.e., $p$ is a formal proof in \textsf{ILC}, which completes the proof of Corollary~\ref{ILCAsConservativeExtensionOfLLJ}.

\section{Proof of Theorem~\ref{ThmCutEliminationForILC}}
\label{CutElimination}
In this appendix, we define a cut-elimination procedure for \textsf{ILC} (Definition~\ref{DefILC}), following the method for \textsf{LLJ} (Definition~\ref{DefLLJ}) given in \cite{bierman1994intuitionistic}.
This cut-elimination procedure removes applications of Cut in a formal proof that are topmost with respect to the \emph{depth} of each application of Cut among those with the highest \emph{rank} in the formal proof.

\begin{definition}[Ranks \cite{bierman1994intuitionistic}]
\label{DefRanks}
The \emph{rank} $|A| \in \mathbb{N}$ of each formula $A$ in ILL$^{\text{e}}$ is defined by the following induction on $A$:
\begin{mathpar}
|X| \colonequals | \top | \colonequals | \bot | \colonequals |1| \colonequals |0| \colonequals 0
\and
|\neg A| \colonequals |\oc A| \colonequals |\wn A| \colonequals |A| + 1
\and
|A \otimes B| \colonequals |A \mathbin{\&} B| \colonequals |A \invamp B| \colonequals |A \oplus B| \colonequals |A| + |B| + 1.
\end{mathpar}
\end{definition}

\begin{definition}[Depths \cite{bierman1994intuitionistic}]
\label{DefDepths}
The \emph{depth} $d(p) \in \mathbb{N}$ of a given formal proof $p$ in \textsf{ILC} is defined by the following induction on $p$:
\begin{equation*}
d(p) \colonequals \begin{cases} 0 &\text{if the last rule occurring in $p$ is Id;} \\ d(p')+1 &\text{if $p$ consists of a formal proof $p'$ followed by an unary rule;} \\ \max\{ d(p'_1), d(p'_2) \} + 1 &\text{if $p$ consists of formal proofs $p_1'$ and $p'_2$ followed by a binary rule.} \end{cases}
\end{equation*}
\end{definition}

We may basically follow the proof given in \cite{bierman1994intuitionistic} to show that our cut-elimination procedure  eliminates all the applications of Cut in any formal proof in \textsf{ILC}.
Hence, in the following, we only describe how the cut-elimination procedure locally transforms each application of Cut.

Strictly speaking, the cut-elimination procedure actually deals with a multiple, consecutive applications of Cut at a time, regarding them as the following single rule:
\begin{definition}[Multiple cuts]
\emph{Left multiple Cut} is the rule
\begin{small}
\begin{equation*}
\AxiomC{$\Delta \vdash A, \Gamma$}
\AxiomC{$\Delta', A^n \vdash \Gamma'$}
\LeftLabel{\textsc{(CutL$^{n}$)}}
\RightLabel{\textsc{($n \in \mathbb{N}$)}}
\BinaryInfC{$\Delta^n, \Delta' \vdash \Gamma^n, \Gamma'$}
\DisplayProof
\end{equation*}
\end{small}and symmetrically, \emph{right multiple Cut} is the rule
\begin{small}
\begin{equation*}
\AxiomC{$\Delta' \vdash A^n, \Gamma'$}
\AxiomC{$\Delta, A \vdash \Gamma$}
\LeftLabel{\textsc{(CutR$^{n}$)}}
\RightLabel{\textsc{($n \in \mathbb{N}$)}}
\BinaryInfC{$\Delta', \Delta^n \vdash \Gamma', \Gamma^n$}
\DisplayProof
\end{equation*}
\end{small}
\end{definition}

The left and the right multiple Cut rules are \emph{derivable} in \textsf{ILC} by consecutive applications of Cut from left and right, respectively, in the evident manner.
The cut-elimination procedure has to take multiple Cut as a single rule in order to handle applications of Cut whose cut formula is the principal formula of a contraction rule; see \cite[\S 4.1.9]{troelstra2000basic} for the details.

Let us now list all the cases of an application of Cut in \textsf{ILC}, where the cases are divided into four patterns, and describe how the cut-elimination procedure transforms them.

\subsection{Principal cuts}
\label{PrincipalCuts}
The first pattern is an application of Cut such that the principal formulas of the rules at the end of the two hypotheses of the application of Cut are both the cut formula. 

The cases of the first pattern are the following:
\begin{itemize}

\item ($\top$R, $\top$L)-Cut. A formal proof of the form
\begin{scriptsize}
\begin{equation*}
\AxiomC{}
\LeftLabel{\textsc{($\top$R)}}
\UnaryInfC{$\vdash \top$}
\AxiomC{$p$}
\noLine
\UnaryInfC{$\top^n, \Delta \vdash \Gamma$}
\LeftLabel{\textsc{($\top$L)}}
\UnaryInfC{$\top^{n+1}, \Delta \vdash \Gamma$}
\LeftLabel{\textsc{(CutL$^{n+1}$)}}
\BinaryInfC{$\Delta \vdash \Gamma$}
\DisplayProof
\end{equation*}
\end{scriptsize}is transformed into
\begin{scriptsize}
\begin{equation*}
\AxiomC{}
\LeftLabel{\textsc{($\top$R)}}
\UnaryInfC{$\vdash \top$}
\AxiomC{$p$}
\noLine
\UnaryInfC{$\top^n, \Delta \vdash \Gamma$}
\LeftLabel{\textsc{(CutL$^{n}$)}}
\BinaryInfC{$\Delta \vdash \Gamma$}
\DisplayProof
\end{equation*}
\end{scriptsize}

\item ($\bot$R, $\bot$L)-Cut. A formal proof of the form
\begin{scriptsize}
\begin{equation*}
\AxiomC{$p$}
\noLine
\UnaryInfC{$\Delta \vdash \bot^n, \Gamma$}
\LeftLabel{\textsc{($\bot$R)}}
\UnaryInfC{$\Delta \vdash \bot^{n+1}, \Gamma$}
\AxiomC{}
\LeftLabel{\textsc{($\bot$L)}}
\UnaryInfC{$\bot \vdash$}
\LeftLabel{\textsc{(CutR$^{n+1}$)}}
\BinaryInfC{$\Delta \vdash \Gamma$}
\DisplayProof
\end{equation*}
\end{scriptsize}is transformed into
\begin{scriptsize}
\begin{equation*}
\AxiomC{$p$}
\noLine
\UnaryInfC{$\Delta \vdash \bot^n, \Gamma$}
\AxiomC{}
\LeftLabel{\textsc{($\bot$L)}}
\UnaryInfC{$\bot \vdash$}
\LeftLabel{\textsc{(CutR$^{n}$)}}
\BinaryInfC{$\Delta \vdash \Gamma$}
\DisplayProof
\end{equation*}
\end{scriptsize}

\item ($\neg$R, $\neg$L)-Cut. A formal proof of the form
\begin{scriptsize}
\begin{equation*}
\AxiomC{$p_1$}
\noLine
\UnaryInfC{$\Delta_1, A \vdash \Gamma_1$}
\LeftLabel{\textsc{($\neg$R)}}
\UnaryInfC{$\Delta_1 \vdash \neg A, \Gamma_1$}
\AxiomC{$p_2$}
\noLine
\UnaryInfC{$\Delta_2, \neg A^n \vdash A, \Gamma_2$}
\LeftLabel{\textsc{($\neg$L)}}
\UnaryInfC{$\Delta_2, \neg A^{n+1} \vdash \Gamma_2$}
\LeftLabel{\textsc{(CutL$^{n+1}$)}}
\BinaryInfC{$\Delta_1^{n+1}, \Delta_2 \vdash \Gamma_1^{n+1}, \Gamma_2$}
\DisplayProof
\end{equation*}
\end{scriptsize}is transformed into
\begin{scriptsize}
\begin{equation*}
\AxiomC{$p_1$}
\noLine
\UnaryInfC{$\Delta_1, A \vdash \Gamma_1$}
\LeftLabel{\textsc{($\neg$R)}}
\UnaryInfC{$\Delta_1 \vdash \neg A, \Gamma_1$}
\AxiomC{$p_2$}
\noLine
\UnaryInfC{$\Delta_2, \neg A^n \vdash A, \Gamma_2$}
\LeftLabel{\textsc{(CutL$^{n}$)}}
\BinaryInfC{$\Delta_1^n, \Delta_2 \vdash A, \Gamma_1^n, \Gamma_2$}
\AxiomC{$p_1$}
\noLine
\UnaryInfC{$\Delta_1, A \vdash \Gamma_1$}
\LeftLabel{\textsc{(Cut)}}
\BinaryInfC{$\Delta_1^{n}, \Delta_2, \Delta_1 \vdash \Gamma_1^{n}, \Gamma_2, \Gamma_1$}
\LeftLabel{\textsc{(XL$^\ast$, XR$^\ast$)}}
\doubleLine
\UnaryInfC{$\Delta_1^{n+1}, \Delta_2 \vdash \Gamma_1^{n+1}, \Gamma_2$}
\DisplayProof
\end{equation*}
\end{scriptsize}The case of the right multiple Cut is just symmetric, and thus we omit it.

\item ($\otimes$R, $\otimes$L)-Cut. A formal proof of the form
\begin{scriptsize}
\begin{equation*}
\AxiomC{$p_1$}
\noLine
\UnaryInfC{$\Delta_1 \vdash A_1, \Gamma_1$}
\AxiomC{$p_2$}
\noLine
\UnaryInfC{$\Delta_2 \vdash A_2, \Gamma_2$}
\LeftLabel{\textsc{($\otimes$R)}}
\BinaryInfC{$\Delta_1, \Delta_2 \vdash A_1 \otimes A_2, \Gamma_1, \Gamma_2$}
\AxiomC{$p_3$}
\noLine
\UnaryInfC{$\Delta_3, (A_1 \otimes A_2)^n, A_1, A_2 \vdash \Gamma_3$}
\LeftLabel{\textsc{($\otimes$L)}}
\UnaryInfC{$\Delta_3, (A_1 \otimes A_2)^{n+1} \vdash \Gamma_3$}
\LeftLabel{\textsc{(CutL$^{n+1}$)}}
\BinaryInfC{$\Delta_1^{n+1}, \Delta_2^{n+1}, \Delta_3 \vdash \Gamma_1^{n+1}, \Gamma_2^{n+1}, \Gamma_3$}
\DisplayProof
\end{equation*}
\end{scriptsize}is transformed into
\begin{scriptsize}
\begin{equation*}
\AxiomC{$p_1$}
\noLine
\UnaryInfC{$\Delta_1 \vdash A_1, \Gamma_1$}
\AxiomC{$p_2$}
\noLine
\UnaryInfC{$\Delta_2 \vdash A_2, \Gamma_2$}
\AxiomC{$p_1$}
\noLine
\UnaryInfC{$\Delta_1 \vdash A_1, \Gamma_1$}
\AxiomC{$p_2$}
\noLine
\UnaryInfC{$\Delta_2 \vdash A_2, \Gamma_2$}
\LeftLabel{\textsc{($\otimes$R)}}
\BinaryInfC{$\Delta_1, \Delta_2 \vdash A_1 \otimes A_2, \Gamma_1, \Gamma_2$}
\AxiomC{$p_3$}
\noLine
\UnaryInfC{$\Delta_3, (A_1 \otimes A_2)^n, A_1, A_2 \vdash \Gamma_3$}
\LeftLabel{\textsc{(CutL$^{n}$)}}
\BinaryInfC{$\Delta_1^{n}, \Delta_2^{n}, \Delta_3, A_1, A_2 \vdash \Gamma_1^n, \Gamma_2^n, \Gamma_3$}
\LeftLabel{\textsc{(Cut)}}
\BinaryInfC{$\Delta_2, \Delta_1^{n}, \Delta_2^{n}, \Delta_3, A_1 \vdash \Gamma_2, \Gamma_1^n, \Gamma_2^{n}, \Gamma_3$}
\LeftLabel{\textsc{(XL$^\ast$, XR$^\ast$)}}
\doubleLine
\UnaryInfC{$\Delta_1^{n}, \Delta_2^{n+1}, \Delta_3, A_1 \vdash \Gamma_1^n, \Gamma_2^{n+1}, \Gamma_3$}
\LeftLabel{\textsc{(Cut)}}
\BinaryInfC{$\Delta_1^{n+1}, \Delta_2^{n+1}, \Delta_3 \vdash \Gamma_1^{n+1}, \Gamma_2^{n+1}, \Gamma_3$}
\DisplayProof
\end{equation*}
\end{scriptsize}The case of the right multiple Cut is just symmetric, and thus we omit it.

\item ($\invamp$R, $\invamp$L)-Cut. A formal proof of the form
\begin{scriptsize}
\begin{equation*}
\AxiomC{$p$}
\noLine
\UnaryInfC{$\Delta \vdash A_1, A_2, (A_1 \invamp A_2)^n, \Gamma$}
\LeftLabel{\textsc{($\invamp$R)}}
\UnaryInfC{$\Delta \vdash (A_1 \invamp A_2)^{n+1}, \Gamma$}
\AxiomC{$p_1$}
\noLine
\UnaryInfC{$\Delta_1, A_1 \vdash \Gamma_1$}
\AxiomC{$p_2$}
\noLine
\UnaryInfC{$\Delta_2, A_2 \vdash \Gamma_2$}
\LeftLabel{\textsc{($\invamp$L)}}
\BinaryInfC{$\Delta_1, \Delta_2, A_1 \invamp A_2 \vdash \Gamma_1, \Gamma_2$}
\LeftLabel{\textsc{(CutR$^{n+1}$)}}
\BinaryInfC{$\Delta, \Delta_1^{n+1}, \Delta_2^{n+1} \vdash \Gamma, \Gamma_1^{n+1}, \Gamma_2^{n+1}$}
\DisplayProof
\end{equation*}
\end{scriptsize}is transformed into
\begin{scriptsize}
\begin{equation*}
\AxiomC{$p$}
\noLine
\UnaryInfC{$\Delta \vdash A_1, A_2, (A_1 \invamp A_2)^n, \Gamma$}
\AxiomC{$p_1$}
\noLine
\UnaryInfC{$\Delta_1, A_1 \vdash \Gamma_1$}
\AxiomC{$p_2$}
\noLine
\UnaryInfC{$\Delta_2, A_2 \vdash \Gamma_2$}
\LeftLabel{\textsc{($\invamp$L)}}
\BinaryInfC{$\Delta_1, \Delta_2, A_1 \invamp A_2 \vdash \Gamma_1, \Gamma_2$}
\LeftLabel{\textsc{(CutR$^{n}$)}}
\BinaryInfC{$\Delta, \Delta_1^n, \Delta_2^n \vdash A_1, A_2, \Gamma, \Gamma_1^{n}, \Gamma_2^{n}$}
\AxiomC{$p_1$}
\noLine
\UnaryInfC{$\Delta_1, A_1 \vdash \Gamma_1$}
\LeftLabel{\textsc{(Cut)}}
\BinaryInfC{$\Delta, \Delta_1^{n}, \Delta_2^n, \Delta_1 \vdash A_2, \Gamma, \Gamma_1^{n}, \Gamma_2^{n}, \Gamma_1$}
\LeftLabel{\textsc{(XL$^\ast$, XR$^\ast$)}}
\doubleLine
\UnaryInfC{$\Delta, \Delta_1^{n+1}, \Delta_2^n \vdash A_2, \Gamma, \Gamma_1^{n+1}, \Gamma_2^{n}$}
\AxiomC{$p_2$}
\noLine
\UnaryInfC{$\Delta_2, A_2 \vdash \Gamma_2$}
\LeftLabel{\textsc{(Cut)}}
\BinaryInfC{$\Delta, \Delta_1^{n+1}, \Delta_2^{n+1} \vdash \Gamma, \Gamma_1^{n+1}, \Gamma_2^{n+1}$}
\DisplayProof
\end{equation*}
\end{scriptsize}The case of the left multiple Cut is just symmetric, and thus we omit it.

\item ($\&$R, $\&$L)-Cut. A formal proof of the form
\begin{scriptsize}
\begin{equation*}
\AxiomC{$p_1$}
\noLine
\UnaryInfC{$\Delta \vdash A_1, \Gamma$}
\AxiomC{$p_2$}
\noLine
\UnaryInfC{$\Delta \vdash A_2, \Gamma$}
\LeftLabel{\textsc{($\&$R)}}
\BinaryInfC{$\Delta \vdash A_1 \& A_2, \Gamma$}
\AxiomC{$p'$}
\noLine
\UnaryInfC{$\Delta', (A_1 \& A_2)^n, A_i \vdash \Gamma'$}
\LeftLabel{\textsc{($\&$L)}}
\RightLabel{($i \in \overline{2}$)}
\UnaryInfC{$\Delta', (A_1 \& A_2)^{n+1} \vdash \Gamma'$}
\LeftLabel{\textsc{(CutL$^{n+1}$)}}
\BinaryInfC{$\Delta^{n+1}, \Delta' \vdash \Gamma^{n+1}, \Gamma'$}
\DisplayProof
\end{equation*}
\end{scriptsize}is transformed into
\begin{scriptsize}
\begin{equation*}
\AxiomC{$p_i$}
\noLine
\UnaryInfC{$\Delta \vdash A_i, \Gamma$}
\AxiomC{$p_1$}
\noLine
\UnaryInfC{$\Delta \vdash A_1, \Gamma$}
\AxiomC{$p_2$}
\noLine
\UnaryInfC{$\Delta \vdash A_2, \Gamma$}
\LeftLabel{\textsc{($\&$R)}}
\BinaryInfC{$\Delta \vdash A_1 \& A_2, \Gamma$}
\AxiomC{$p'$}
\noLine
\UnaryInfC{$\Delta', (A_1 \& A_2)^n, A_i \vdash \Gamma'$}
\LeftLabel{\textsc{(CutL$^{n}$)}}
\BinaryInfC{$\Delta^{n}, \Delta', A_i \vdash \Gamma^n, \Gamma'$}
\LeftLabel{\textsc{(Cut)}}
\BinaryInfC{$\Delta^{n+1}, \Delta' \vdash \Gamma^{n+1}, \Gamma'$}
\DisplayProof
\end{equation*}
\end{scriptsize}The case of the right multiple Cut is just symmetric, and thus we omit it.

\item ($\oplus$R, $\oplus$L)-Cut. A formal proof of the form
\begin{scriptsize}
\begin{equation*}
\AxiomC{$p'$}
\noLine
\UnaryInfC{$\Delta' \vdash A_i, (A_1 \oplus A_2)^n, \Gamma'$}
\LeftLabel{\textsc{($\oplus$R)}}
\RightLabel{($i \in \overline{2}$)}
\UnaryInfC{$\Delta' \vdash (A_1 \oplus A_2)^{n+1}, \Gamma'$}
\AxiomC{$p_1$}
\noLine
\UnaryInfC{$\Delta ,A_1 \vdash \Gamma$}
\AxiomC{$p_2$}
\noLine
\UnaryInfC{$\Delta, A_2 \vdash \Gamma$}
\LeftLabel{\textsc{($\oplus$L)}}
\BinaryInfC{$\Delta, A_1 \oplus A_2 \vdash \Gamma$}
\LeftLabel{\textsc{(CutR$^{n+1}$)}}
\BinaryInfC{$\Delta', \Delta^{n+1} \vdash \Gamma', \Gamma^{n+1}$}
\DisplayProof
\end{equation*}
\end{scriptsize}is transformed into
\begin{scriptsize}
\begin{equation*}
\AxiomC{$p'$}
\noLine
\UnaryInfC{$\Delta' \vdash A_i, (A_1 \oplus A_2)^n, \Gamma'$}
\AxiomC{$p_1$}
\noLine
\UnaryInfC{$\Delta, A_1 \vdash \Gamma$}
\AxiomC{$p_2$}
\noLine
\UnaryInfC{$\Delta, A_2 \vdash \Gamma$}
\LeftLabel{\textsc{($\oplus$L)}}
\BinaryInfC{$\Delta, A_1 \oplus A_2 \vdash \Gamma$}
\LeftLabel{\textsc{(CutR$^{n}$)}}
\BinaryInfC{$\Delta', \Delta^n \vdash A_i, \Gamma', \Gamma^{n}$}
\AxiomC{$p_i$}
\noLine
\UnaryInfC{$\Delta, A_i \vdash \Gamma$}
\LeftLabel{\textsc{(Cut)}}
\BinaryInfC{$\Delta', \Delta^{n+1} \vdash \Gamma', \Gamma^{n+1}$}
\DisplayProof
\end{equation*}
\end{scriptsize}The case of the left multiple cut is just symmetric, and thus we omit it.

\item ($\oc$R$^{\oc\wn}$, $\oc$D)-Cut. A formal proof of the form
\begin{scriptsize}
\begin{equation*}
\AxiomC{$p$}
\noLine
\UnaryInfC{$\oc \Delta \vdash A, \wn \Gamma$}
\LeftLabel{\textsc{($\oc$R$^{\oc\wn}$)}}
\UnaryInfC{$\oc \Delta \vdash \oc A, \wn \Gamma$}
\AxiomC{$p'$}
\noLine
\UnaryInfC{$\Delta', \oc A^n, A \vdash \Gamma'$}
\LeftLabel{\textsc{($\oc$D)}}
\UnaryInfC{$\Delta', \oc A^{n+1} \vdash \Gamma'$}
\LeftLabel{\textsc{(CutL$^{n+1}$)}}
\BinaryInfC{$\oc \Delta^{n+1}, \Delta' \vdash \wn \Gamma^{n+1}, \Gamma'$}
\DisplayProof
\end{equation*}
\end{scriptsize}is transformed into
\begin{scriptsize}
\begin{equation*}
\AxiomC{$p$}
\noLine
\UnaryInfC{$\oc \Delta \vdash A, \wn \Gamma$}
\AxiomC{$p$}
\noLine
\UnaryInfC{$\oc \Delta \vdash A, \wn \Gamma$}
\LeftLabel{\textsc{($\oc$R$^{\oc\wn}$)}}
\UnaryInfC{$\oc \Delta \vdash \oc A, \wn \Gamma$}
\AxiomC{$p'$}
\noLine
\UnaryInfC{$\Delta', \oc A^n, A \vdash \Gamma'$}
\LeftLabel{\textsc{(CutL$^{n}$)}}
\BinaryInfC{$\oc \Delta^{n}, \Delta', A \vdash \wn \Gamma^{n}, \Gamma'$}
\LeftLabel{\textsc{(Cut)}}
\BinaryInfC{$\oc \Delta^{n+1}, \Delta' \vdash \wn \Gamma^{n+1}, \Gamma'$}
\DisplayProof
\end{equation*}
\end{scriptsize}

\begin{remark}
\label{RemarkOC}
The case of the right multiple Cut, i.e., a formal proof of the form
\begin{scriptsize}
\begin{equation*}
\AxiomC{$p'$}
\noLine
\UnaryInfC{$\oc \Delta' \vdash A, \oc A^n, \wn \Gamma'$}
\LeftLabel{\textsc{($\oc$R$^{\oc\wn}$)}}
\UnaryInfC{$\oc \Delta' \vdash \oc A^{n+1}, \wn \Gamma'$}
\AxiomC{$p$}
\noLine
\UnaryInfC{$\Delta, A \vdash \Gamma$}
\LeftLabel{\textsc{($\oc$D)}}
\UnaryInfC{$\Delta, \oc A \vdash \Gamma$}
\LeftLabel{\textsc{(CutR$^{n+1}$)}}
\BinaryInfC{$\oc \Delta', \Delta^{n+1} \vdash \wn \Gamma', \Gamma^{n+1}$}
\DisplayProof
\end{equation*}
\end{scriptsize}for $n > 0$, cannot occur since otherwise the application of the rule $\oc$R$^{\oc\wn}$ would be invalid.
This remark is applied to the following two cases as well.
\end{remark}

\item ($\oc$R$^{\oc\wn}$, $\oc$W)-Cut. A formal proof of the form
\begin{scriptsize}
\begin{equation*}
\AxiomC{$p$}
\noLine
\UnaryInfC{$\oc \Delta \vdash A, \wn \Gamma$}
\LeftLabel{\textsc{($\oc$R$^{\oc\wn}$)}}
\UnaryInfC{$\oc \Delta \vdash \oc A, \wn \Gamma$}
\AxiomC{$p'$}
\noLine
\UnaryInfC{$\Delta', \oc A^n \vdash \Gamma'$}
\LeftLabel{\textsc{($\oc$W)}}
\UnaryInfC{$\Delta', \oc A^{n+1} \vdash \Gamma'$}
\LeftLabel{\textsc{(CutL$^{n+1}$)}}
\BinaryInfC{$\oc \Delta^{n+1}, \Delta' \vdash \wn \Gamma^{n+1}, \Gamma'$}
\DisplayProof
\end{equation*}
\end{scriptsize}is transformed into
\begin{scriptsize}
\begin{equation*}
\AxiomC{$p$}
\noLine
\UnaryInfC{$\oc \Delta \vdash A, \wn \Gamma$}
\LeftLabel{\textsc{($\oc$R$^{\oc\wn}$)}}
\UnaryInfC{$\oc \Delta \vdash \oc A, \wn \Gamma$}
\AxiomC{$p'$}
\noLine
\UnaryInfC{$\Delta', \oc A^n \vdash \Gamma'$}
\LeftLabel{\textsc{(CutL$^{n}$)}}
\BinaryInfC{$\oc \Delta^{n}, \Delta' \vdash \wn \Gamma^{n}, \Gamma'$}
\doubleLine
\LeftLabel{\textsc{($\oc$W$^\ast$, XL$^\ast$)}}
\UnaryInfC{$\oc \Delta^{n+1}, \Delta' \vdash \wn \Gamma^{n}, \Gamma'$}
\LeftLabel{\textsc{($\wn$W$^\ast$)}}
\doubleLine
\UnaryInfC{$\oc \Delta^{n+1}, \Delta' \vdash \wn \Gamma^{n+1}, \Gamma'$}
\DisplayProof
\end{equation*}
\end{scriptsize}

\item ($\oc$R$^{\oc\wn}$, $\oc$C)-Cut. A formal proof of the form
\begin{scriptsize}
\begin{equation*}
\AxiomC{$p$}
\noLine
\UnaryInfC{$\oc \Delta \vdash A, \wn \Gamma$}
\LeftLabel{\textsc{($\oc$R$^{\oc\wn}$)}}
\UnaryInfC{$\oc \Delta \vdash \oc A, \wn \Gamma$}
\AxiomC{$p'$}
\noLine
\UnaryInfC{$\Delta', \oc A^n, \oc A, \oc A \vdash \Gamma'$}
\LeftLabel{\textsc{($\oc$C)}}
\UnaryInfC{$\Delta', \oc A^{n+1} \vdash \Gamma'$}
\LeftLabel{\textsc{(CutL$^{n+1}$)}}
\BinaryInfC{$\oc \Delta^{n+1}, \Delta' \vdash \wn \Gamma^{n+1}, \Gamma'$}
\DisplayProof
\end{equation*}
\end{scriptsize}is transformed into
\begin{scriptsize}
\begin{equation*}
\AxiomC{$p$}
\noLine
\UnaryInfC{$\oc \Delta \vdash A, \wn \Gamma$}
\LeftLabel{\textsc{($\oc$R$^{\oc\wn}$)}}
\UnaryInfC{$\oc \Delta \vdash \oc A, \wn \Gamma$}
\AxiomC{$p'$}
\noLine
\UnaryInfC{$\Delta', \oc A^n, \oc A, \oc A \vdash \Gamma'$}
\LeftLabel{\textsc{(CutL$^{n+2}$)}}
\BinaryInfC{$\oc \Delta^{n}, \oc \Delta, \oc \Delta, \Delta' \vdash \wn \Gamma^{n+2}, \Gamma'$}
\LeftLabel{\textsc{($\oc$C$^\ast$, XL$^\ast$)}}
\doubleLine
\UnaryInfC{$\oc \Delta^{n+1}, \Delta' \vdash \wn \Gamma^{n+2}, \Gamma'$}
\LeftLabel{\textsc{($\wn$C$^\ast$)}}
\doubleLine
\UnaryInfC{$\oc \Delta^{n+1}, \Delta' \vdash \wn \Gamma^{n}, \Gamma'$}
\DisplayProof
\end{equation*}
\end{scriptsize}

\item ($\wn$D, $\wn$L$^{\oc\wn}$)-Cut. A formal proof of the form
\begin{scriptsize}
\begin{equation*}
\AxiomC{$p'$}
\noLine
\UnaryInfC{$\Delta' \vdash A, \wn A^n, \Gamma'$}
\LeftLabel{\textsc{($\wn$D)}}
\UnaryInfC{$\Delta' \vdash \wn A^{n+1}, \Gamma'$}
\AxiomC{$p$}
\noLine
\UnaryInfC{$\oc \Delta, A \vdash \wn \Gamma$}
\LeftLabel{\textsc{($\wn$L$^{\oc\wn}$)}}
\UnaryInfC{$\oc \Delta, \wn A \vdash \wn \Gamma$}
\LeftLabel{\textsc{(CutR$^{n+1}$)}}
\BinaryInfC{$\Delta', \oc \Delta^{n+1} \vdash \Gamma', \wn \Gamma^{n+1}$}
\DisplayProof
\end{equation*}
\end{scriptsize}is transformed into
\begin{scriptsize}
\begin{equation*}
\AxiomC{$p'$}
\noLine
\UnaryInfC{$\Delta' \vdash A, \wn A^n, \Gamma'$}
\AxiomC{$p$}
\noLine
\UnaryInfC{$\oc \Delta, A \vdash \wn \Gamma$}
\LeftLabel{\textsc{($\wn$L$^{\oc\wn}$)}}
\UnaryInfC{$\oc \Delta, \wn A \vdash \wn \Gamma$}
\LeftLabel{\textsc{(CutR$^{n}$)}}
\BinaryInfC{$\Delta', \oc \Delta^{n} \vdash A, \Gamma', \wn \Gamma^{n}$}
\AxiomC{$p$}
\noLine
\UnaryInfC{$\oc \Delta, A \vdash \wn \Gamma$}
\LeftLabel{\textsc{(Cut)}}
\BinaryInfC{$\Delta', \oc \Delta^{n+1} \vdash \Gamma', \wn \Gamma^{n+1}$}
\DisplayProof
\end{equation*}
\end{scriptsize}

\begin{remark}
Dually to Remark~\ref{RemarkOC}, the case of the left multiple Cut cannot occur for the application of the rule $\wn$L$^{\oc\wn}$.
This remark is applied to the following two cases too.
\end{remark}

\item ($\wn$W, $\wn$L$^{\oc\wn}$)-Cut. A formal proof of the form
\begin{scriptsize}
\begin{equation*}
\AxiomC{$p'$}
\noLine
\UnaryInfC{$\Delta' \vdash \wn A^n, \Gamma'$}
\LeftLabel{\textsc{($\wn$W)}}
\UnaryInfC{$\Delta' \vdash \wn A^{n+1}, \Gamma'$}
\AxiomC{$p$}
\noLine
\UnaryInfC{$\oc \Delta, A \vdash \wn \Gamma$}
\LeftLabel{\textsc{($\wn$L$^{\oc\wn}$)}}
\UnaryInfC{$\oc \Delta, \wn A \vdash \wn \Gamma$}
\LeftLabel{\textsc{(CutR$^{n+1}$)}}
\BinaryInfC{$\Delta', \oc \Delta^{n+1} \vdash \Gamma', \wn \Gamma^{n+1}$}
\DisplayProof
\end{equation*}
\end{scriptsize}is transformed into
\begin{scriptsize}
\begin{equation*}
\AxiomC{$p'$}
\noLine
\UnaryInfC{$\Delta' \vdash \wn A^n, \Gamma'$}
\AxiomC{$p$}
\noLine
\UnaryInfC{$\oc \Delta, A \vdash \wn \Gamma$}
\LeftLabel{\textsc{($\wn$L$^{\oc\wn}$)}}
\UnaryInfC{$\oc \Delta, \wn A \vdash \wn \Gamma$}
\LeftLabel{\textsc{(CutR$^n$)}}
\BinaryInfC{$\Delta', \oc \Delta^n \vdash \Gamma', \wn \Gamma^n$}
\LeftLabel{\textsc{($\oc$W$^\ast$)}}
\doubleLine
\UnaryInfC{$\Delta', \oc \Delta^{n+1} \vdash \Gamma', \wn \Gamma^n$}
\LeftLabel{\textsc{($\wn$W$^\ast$, XR$^\ast$)}}
\doubleLine
\UnaryInfC{$\Delta', \oc \Delta^{n+1} \vdash \Gamma', \wn \Gamma^{n+1}$}
\DisplayProof
\end{equation*}
\end{scriptsize}

\item ($\wn$C, $\wn$L$^{\oc\wn}$)-Cut. A formal proof of the form
\begin{scriptsize}
\begin{equation*}
\AxiomC{$p'$}
\noLine
\UnaryInfC{$\Delta' \vdash \wn A, \wn A, \wn A^n, \Gamma'$}
\LeftLabel{\textsc{($\wn$C)}}
\UnaryInfC{$\Delta' \vdash \wn A^{n+1}, \Gamma'$}
\AxiomC{$p$}
\noLine
\UnaryInfC{$\oc \Delta, A \vdash \wn \Gamma$}
\LeftLabel{\textsc{($\wn$L$^{\oc\wn}$)}}
\UnaryInfC{$\oc \Delta, \wn A \vdash \wn \Gamma$}
\LeftLabel{\textsc{(CutR$^{n+1}$)}}
\BinaryInfC{$\Delta', \oc \Delta^{n+1} \vdash \Gamma', \wn \Gamma^{n+1}$}
\DisplayProof
\end{equation*}
\end{scriptsize}is transformed into
\begin{scriptsize}
\begin{equation*}
\AxiomC{$p'$}
\noLine
\UnaryInfC{$\Delta' \vdash \wn A, \wn A, \wn A^n, \Gamma'$}
\AxiomC{$p$}
\noLine
\UnaryInfC{$\oc \Delta, A \vdash \wn \Gamma$}
\LeftLabel{\textsc{($\wn$L$^{\oc\wn}$)}}
\UnaryInfC{$\oc \Delta, \wn A \vdash \wn \Gamma$}
\LeftLabel{\textsc{(CutR$^{n+2}$)}}
\BinaryInfC{$\Delta', \oc \Delta^{n}, \oc \Delta, \oc \Delta \vdash \Gamma', \wn \Gamma^{n}, \wn \Gamma, \wn \Gamma$}
\LeftLabel{\textsc{($\oc$C$^\ast$)}}
\doubleLine
\UnaryInfC{$\Delta', \oc \Delta^{n+1} \vdash \Gamma', \wn \Gamma^n, \wn \Gamma, \wn \Gamma$}
\LeftLabel{\textsc{($\wn$C$^\ast$, XR$^\ast$)}}
\doubleLine
\UnaryInfC{$\oc \Delta^{n+1}, \Delta' \vdash \wn \Gamma^{n+1}, \Gamma'$}
\DisplayProof
\end{equation*}
\end{scriptsize}

\subsection{Right minor cuts}
\label{RightMinorCuts}
The second pattern is an application of Cut such that the principal formula of the rule at the end of the right hypothesis of the application of Cut is not the cut formula. 

The cases of the second pattern are the following:
\item Right-minor $1$R-Cut. A formal proof of the form
\begin{scriptsize}
\begin{equation*}
\AxiomC{$p$}
\noLine
\UnaryInfC{$\Delta \vdash A, \Gamma$}
\AxiomC{}
\LeftLabel{\textsc{($1$L)}}
\UnaryInfC{$\Delta', A^n \vdash 1, \Gamma'$}
\LeftLabel{\textsc{(CutL$^{n}$)}}
\BinaryInfC{$\Delta^n, \Delta' \vdash \Gamma^n, 1, \Gamma'$}
\DisplayProof
\end{equation*}
\end{scriptsize}
is transformed into
\begin{scriptsize}
\begin{equation*}
\AxiomC{}
\LeftLabel{\textsc{($1$L)}}
\UnaryInfC{$\Delta^n, \Delta' \vdash 1, \Gamma^n, \Gamma'$}
\LeftLabel{\textsc{(XR$^\ast$)}}
\doubleLine
\UnaryInfC{$\Delta^n, \Delta' \vdash \Gamma^n, 1, \Gamma'$}
\DisplayProof
\end{equation*}
\end{scriptsize}The case of the right multiple Cut is similar, and thus we omit it.

\item Right-minor $0$L-Cut. 
A formal proof of the form
\begin{scriptsize}
\begin{equation*}
\AxiomC{$p$}
\noLine
\UnaryInfC{$\Delta \vdash A, \Gamma$}
\AxiomC{}
\LeftLabel{\textsc{($0$L)}}
\UnaryInfC{$\Delta', A^n, 0 \vdash \Gamma'$}
\LeftLabel{\textsc{(CutL$^{n}$)}}
\BinaryInfC{$\Delta^n, \Delta', 0 \vdash \Gamma^n, \Gamma'$}
\DisplayProof
\end{equation*}
\end{scriptsize}is transformed into
\begin{scriptsize}
\begin{equation*}
\AxiomC{}
\LeftLabel{\textsc{($0$L)}}
\UnaryInfC{$\Delta^n, \Delta', 0 \vdash \Gamma^n, \Gamma'$}
\DisplayProof
\end{equation*}
\end{scriptsize}The case of the right multiple Cut is similar, and thus we omit it.

\item Right-minor $\top$L-Cut. A formal proof of the form
\begin{scriptsize}
\begin{equation*}
\AxiomC{$p_1$}
\noLine
\UnaryInfC{$\Delta_1 \vdash A, \Gamma_1$}
\AxiomC{$p_2$}
\noLine
\UnaryInfC{$\Delta_2, A^n \vdash \Gamma_2$}
\LeftLabel{\textsc{($\top$L)}}
\UnaryInfC{$\Delta_2, A^n, \top \vdash \Gamma_2$}
\LeftLabel{\textsc{(CutL$^{n}$)}}
\BinaryInfC{$\Delta_1^n, \Delta_2, \top \vdash \Gamma_1^n, \Gamma_2$}
\DisplayProof
\end{equation*}
\end{scriptsize}is transformed into
\begin{scriptsize}
\begin{equation*}
\AxiomC{$p_1$}
\noLine
\UnaryInfC{$\Delta_1 \vdash A, \Gamma_1$}
\AxiomC{$p_2$}
\noLine
\UnaryInfC{$\Delta_2, A^n \vdash \Gamma_2$}
\LeftLabel{\textsc{(CutL$^{n}$)}}
\BinaryInfC{$\Delta_1^n, \Delta_2 \vdash \Gamma_1^n, \Gamma_2$}
\LeftLabel{\textsc{($\top$L)}}
\UnaryInfC{$\Delta_1^n, \Delta_2, \top \vdash \Gamma_1^n, \Gamma_2$}
\DisplayProof
\end{equation*}
\end{scriptsize}The case of the right multiple Cut is similar, and thus we omit it.

\item Right-minor $\top$R-Cut. This case is impossible as there is no occurrence of a formula on the left-hand side of the conclusion of $\top$R.

\item Right-minor $\bot$L-Cut. This case is impossible as there is only one occurrence of bottom $\bot$ on the left-hand side of the conclusion of $\bot$L.

\item Right-minor $\bot$R-Cut. A formal proof of the form
\begin{scriptsize}
\begin{equation*}
\AxiomC{$p_1$}
\noLine
\UnaryInfC{$\Delta_1 \vdash A, \Gamma_1$}
\AxiomC{$p_2$}
\noLine
\UnaryInfC{$\Delta_2, A^n \vdash \Gamma_2$}
\LeftLabel{\textsc{($\bot$R)}}
\UnaryInfC{$\Delta_2, A^n \vdash \bot, \Gamma_2$}
\LeftLabel{\textsc{(CutL$^{n}$)}}
\BinaryInfC{$\Delta_1^n, \Delta_2 \vdash \Gamma_1^n, \bot, \Gamma_2$}
\DisplayProof
\end{equation*}
\end{scriptsize}is transformed into
\begin{scriptsize}
\begin{equation*}
\AxiomC{$p_1$}
\noLine
\UnaryInfC{$\Delta_1 \vdash A, \Gamma_1$}
\AxiomC{$p_2$}
\noLine
\UnaryInfC{$\Delta_2, A^n \vdash \Gamma_2$}
\LeftLabel{\textsc{(CutL$^{n}$)}}
\BinaryInfC{$\Delta_1^n, \Delta_2 \vdash \Gamma_1^n, \Gamma_2$}
\LeftLabel{\textsc{($\bot$R)}}
\UnaryInfC{$\Delta_1^n, \Delta_2 \vdash \bot, \Gamma_1^n, \Gamma_2$}
\LeftLabel{\textsc{(XR$^\ast$)}}
\doubleLine
\UnaryInfC{$\Delta_1^n, \Delta_2 \vdash \Gamma_1^n, \bot, \Gamma_2$}
\DisplayProof
\end{equation*}
\end{scriptsize}The case of the right multiple Cut is similar, and thus we omit it.

\item Right-minor $\otimes$L-Cut. A formal proof of the form
\begin{scriptsize}
\begin{equation*}
\AxiomC{$p_1$}
\noLine
\UnaryInfC{$\Delta_1 \vdash A, \Gamma_1$}
\AxiomC{$p_2$}
\noLine
\UnaryInfC{$\Delta_2, A^n, B, C \vdash \Gamma_2$}
\LeftLabel{\textsc{($\otimes$L)}}
\UnaryInfC{$\Delta_2, A^n, B \otimes C \vdash \Gamma_2$}
\LeftLabel{\textsc{(CutL$^{n}$)}}
\BinaryInfC{$\Delta_1^n, \Delta_2, B \otimes C \vdash \Gamma_1^n, \Gamma_2$}
\DisplayProof
\end{equation*}
\end{scriptsize}is transformed into
\begin{scriptsize}
\begin{equation*}
\AxiomC{$p_1$}
\noLine
\UnaryInfC{$\Delta_1 \vdash A, \Gamma_1$}
\AxiomC{$p_2$}
\noLine
\UnaryInfC{$\Delta_2, A^n, B, C \vdash \Gamma_2$}
\LeftLabel{\textsc{(CutL$^{n}$)}}
\BinaryInfC{$\Delta_1^n, \Delta_2, B, C \vdash \Gamma_1^n, \Gamma_2$}
\LeftLabel{\textsc{($\otimes$L)}}
\UnaryInfC{$\Delta_1^n, \Delta_2, B \otimes C \vdash \Gamma_1^n, \Gamma_2$}
\DisplayProof
\end{equation*}
\end{scriptsize}The case of the right multiple Cut is similar, and thus we omit it.

\item Right-minor $\otimes$R-Cut. A formal proof of the form
\begin{scriptsize}
\begin{equation*}
\AxiomC{$p$}
\noLine
\UnaryInfC{$\Delta \vdash A, \Gamma$}
\AxiomC{$p_1$}
\noLine
\UnaryInfC{$\Delta_1, A^{n_1} \vdash B_1, \Gamma_1$}
\AxiomC{$p_2$}
\noLine
\UnaryInfC{$\Delta_2, A^{n_2} \vdash B_2, \Gamma_2$}
\LeftLabel{\textsc{($\otimes$R)}}
\BinaryInfC{$\Delta_1, \Delta_2, A^{n_1 + n_2} \vdash B_1 \otimes B_2, \Gamma_1, \Gamma_2$}
\LeftLabel{\textsc{(CutL$^{n_1 + n_2}$)}}
\BinaryInfC{$\Delta^{n_1 + n_2}, \Delta_1, \Delta_2 \vdash \Gamma^{n_1 + n_2}, B_1 \otimes B_2, \Gamma_1, \Gamma_2$}
\DisplayProof
\end{equation*}
\end{scriptsize}is transformed into
\begin{scriptsize}
\begin{equation*}
\AxiomC{$p$}
\noLine
\UnaryInfC{$\Delta \vdash A, \Gamma$}
\AxiomC{$p_1$}
\noLine
\UnaryInfC{$\Delta_1, A^{n_1} \vdash B_1, \Gamma_1$}
\LeftLabel{\textsc{(CutL$^{n_1}$)}}
\BinaryInfC{$\Delta^{n_1}, \Delta_1 \vdash \Gamma^{n_1}, B_1, \Gamma_1$}
\LeftLabel{\textsc{(XR$^\ast$)}}
\doubleLine
\UnaryInfC{$\Delta^{n_1}, \Delta_1 \vdash B_1, \Gamma^{n_1}, \Gamma_1$}
\AxiomC{$p$}
\noLine
\UnaryInfC{$\Delta \vdash A, \Gamma$}
\AxiomC{$p_2$}
\noLine
\UnaryInfC{$\Delta_2, A^{n_2} \vdash B_2, \Gamma_2$}
\LeftLabel{\textsc{(CutL$^{n_2}$)}}
\BinaryInfC{$\Delta^{n_2}, \Delta_2 \vdash \Gamma^{n_2}, B_2, \Gamma_2$}
\LeftLabel{\textsc{(XR$^\ast$)}}
\doubleLine
\UnaryInfC{$\Delta^{n_2}, \Delta_2 \vdash B_2, \Gamma^{n_2}, \Gamma_2$}
\LeftLabel{\textsc{($\otimes$R)}}
\BinaryInfC{$\Delta^{n_1}, \Delta_1, \Delta^{n_2}, \Delta_2 \vdash B_1 \otimes B_2, \Gamma^{n_1}, \Gamma_1, \Gamma^{n_2}, \Gamma_2$}
\LeftLabel{\textsc{(XL$^\ast$)}}
\doubleLine
\UnaryInfC{$\Delta^{n_1 + n_2}, \Delta_1, \Delta_2 \vdash B_1 \otimes B_2, \Gamma^{n_1}, \Gamma_1, \Gamma^{n_2}, \Gamma_2$}
\LeftLabel{\textsc{(XR$^\ast$)}}
\doubleLine
\UnaryInfC{$\Delta^{n_1 + n_2}, \Delta_1, \Delta_2 \vdash \Gamma^{n_1 + n_2}, B_1 \otimes B_2, \Gamma_1, \Gamma_2$}
\DisplayProof
\end{equation*}
\end{scriptsize}The case of the right multiple Cut is similar, and thus we omit it.

\item Right-minor $\invamp$L-Cut. A formal proof of the form
\begin{scriptsize}
\begin{equation*}
\AxiomC{$p$}
\noLine
\UnaryInfC{$\Delta \vdash A, \Gamma$}
\AxiomC{$p_1$}
\noLine
\UnaryInfC{$\Delta_1, A^{n_1}, B_1 \vdash \Gamma_1$}
\AxiomC{$p_2$}
\noLine
\UnaryInfC{$\Delta_2, A^{n_2}, B_2 \vdash \Gamma_2$}
\LeftLabel{\textsc{($\otimes$R)}}
\BinaryInfC{$\Delta_1, \Delta_2, A^{n_1 + n_2}, B_1 \invamp B_2 \vdash \Gamma_1, \Gamma_2$}
\LeftLabel{\textsc{(CutL$^{n_1 + n_2}$)}}
\BinaryInfC{$\Delta^{n_1 + n_2}, \Delta_1, \Delta_2, B_1 \invamp B_2 \vdash \Gamma^{n_1 + n_2}, \Gamma_1, \Gamma_2$}
\DisplayProof
\end{equation*}
\end{scriptsize}is transformed into
\begin{scriptsize}
\begin{equation*}
\AxiomC{$p$}
\noLine
\UnaryInfC{$\Delta \vdash A, \Gamma$}
\AxiomC{$p_1$}
\noLine
\UnaryInfC{$\Delta_1, A^{n_1}, B_1 \vdash \Gamma_1$}
\LeftLabel{\textsc{(CutL$^{n_1}$)}}
\BinaryInfC{$\Delta^{n_1}, \Delta_1, B_1 \vdash \Gamma^{n_1}, \Gamma_1$}
\AxiomC{$p$}
\noLine
\UnaryInfC{$\Delta \vdash A, \Gamma$}
\AxiomC{$p_2$}
\noLine
\UnaryInfC{$\Delta_2, A^{n_2}, B_2 \vdash \Gamma_2$}
\LeftLabel{\textsc{(CutL$^{n_2}$)}}
\BinaryInfC{$\Delta^{n_2}, \Delta_2, B_2 \vdash \Gamma^{n_2}, \Gamma_2$}
\LeftLabel{\textsc{($\invamp$L)}}
\BinaryInfC{$\Delta^{n_1}, \Delta_1, \Delta^{n_2}, \Delta_2, B_1 \invamp B_2 \vdash \Gamma^{n_1}, \Gamma_1, \Gamma^{n_2}, \Gamma_2$}
\LeftLabel{\textsc{(XL$^\ast$)}}
\doubleLine
\UnaryInfC{$\Delta^{n_1 + n_2}, \Delta_1, \Delta_2, B_1 \invamp B_2 \vdash \Gamma^{n_1}, \Gamma_1, \Gamma^{n_2}, \Gamma_2$}
\LeftLabel{\textsc{(XR$^\ast$)}}
\doubleLine
\UnaryInfC{$\Delta^{n_1 + n_2}, \Delta_1, \Delta_2, B_1 \invamp B_2 \vdash \Gamma^{n_1 + n_2}, \Gamma_1, \Gamma_2$}
\DisplayProof
\end{equation*}
\end{scriptsize}The case of the right multiple Cut is similar, and thus we omit it.

\item Right-minor $\invamp$R-Cut. A formal proof of the form
\begin{scriptsize}
\begin{equation*}
\AxiomC{$p$}
\noLine
\UnaryInfC{$\Delta \vdash A, \Gamma$}
\AxiomC{$p'$}
\noLine
\UnaryInfC{$\Delta', A^{n} \vdash B_1, B_2, \Gamma'$}
\LeftLabel{\textsc{($\invamp$R)}}
\UnaryInfC{$\Delta', A^{n} \vdash B_1 \invamp B_2, \Gamma'$}
\LeftLabel{\textsc{(CutL$^{n}$)}}
\BinaryInfC{$\Delta^{n}, \Delta' \vdash \Gamma^n, B_1 \invamp B_2, \Gamma'$}
\DisplayProof
\end{equation*}
\end{scriptsize}is transformed into
\begin{scriptsize}
\begin{equation*}
\AxiomC{$p$}
\noLine
\UnaryInfC{$\Delta \vdash A, \Gamma$}
\AxiomC{$p'$}
\noLine
\UnaryInfC{$\Delta', A^{n} \vdash B_1, B_2, \Gamma'$}
\LeftLabel{\textsc{(CutL$^{n}$)}}
\BinaryInfC{$\Delta^{n}, \Delta' \vdash \Gamma^n, B_1, B_2, \Gamma'$}
\LeftLabel{\textsc{($\invamp$R, XR$^\ast$)}}
\doubleLine
\UnaryInfC{$\Delta^{n}, \Delta' \vdash \Gamma^n, B_1 \invamp B_2, \Gamma'$}
\DisplayProof
\end{equation*}
\end{scriptsize}The case of the right multiple Cut is similar, and thus we omit it.

\item Right-minor $\&$L-Cut. A formal proof of the form
\begin{scriptsize}
\begin{equation*}
\AxiomC{$p$}
\noLine
\UnaryInfC{$\Delta \vdash A, \Gamma$}
\AxiomC{$p'$}
\noLine
\UnaryInfC{$\Delta', A^n, B_i \vdash \Gamma'$}
\RightLabel{($i \in \overline{2}$)}
\LeftLabel{\textsc{($\&$L)}}
\UnaryInfC{$\Delta', A^n, B_1 \& B_2 \vdash \Gamma'$}
\LeftLabel{\textsc{(CutL$^{n}$)}}
\BinaryInfC{$\Delta^n, \Delta', B_1 \& B_2 \vdash \Gamma^n, \Gamma'$}
\DisplayProof
\end{equation*}
\end{scriptsize}is transformed into
\begin{scriptsize}
\begin{equation*}
\AxiomC{$p$}
\noLine
\UnaryInfC{$\Delta \vdash A, \Gamma$}
\AxiomC{$p'$}
\noLine
\UnaryInfC{$\Delta', A^n, B_i \vdash \Gamma'$}
\LeftLabel{\textsc{(CutL$^{n}$)}}
\BinaryInfC{$\Delta^n, \Delta', B_i \vdash \Gamma^n, \Gamma'$}
\LeftLabel{\textsc{($\&$L)}}
\UnaryInfC{$\Delta^n, \Delta', B_1 \& B_2 \vdash \Gamma^n, \Gamma'$}
\DisplayProof
\end{equation*}
\end{scriptsize}The case of the right multiple Cut is similar, and thus we omit it.

\item Right-minor $\&$R-Cut. A formal proof of the form
\begin{scriptsize}
\begin{equation*}
\AxiomC{$p$}
\noLine
\UnaryInfC{$\Delta \vdash A, \Gamma$}
\AxiomC{$p'_1$}
\noLine
\UnaryInfC{$\Delta', A^n \vdash B_1, \Gamma'$}
\AxiomC{$p'_2$}
\noLine
\UnaryInfC{$\Delta', A^n \vdash B_1, \Gamma'$}
\LeftLabel{\textsc{($\&$R)}}
\BinaryInfC{$\Delta', A^n \vdash B_1 \& B_2, \Gamma'$}
\LeftLabel{\textsc{(CutL$^{n}$)}}
\BinaryInfC{$\Delta^n, \Delta' \vdash \Gamma^n, B_1 \& B_2, \Gamma'$}
\DisplayProof
\end{equation*}
\end{scriptsize}is transformed into
\begin{scriptsize}
\begin{equation*}
\AxiomC{$p$}
\noLine
\UnaryInfC{$\Delta \vdash A, \Gamma$}
\AxiomC{$p'_1$}
\noLine
\UnaryInfC{$\Delta', A^n \vdash B_1, \Gamma'$}
\LeftLabel{\textsc{(CutL$^{n}$)}}
\BinaryInfC{$\Delta^n, \Delta' \vdash \Gamma^n, B_1, \Gamma'$}
\LeftLabel{\textsc{(XR$^\ast$)}}
\doubleLine
\UnaryInfC{$\Delta^n, \Delta' \vdash B_1, \Gamma^n, \Gamma'$}
\AxiomC{$p$}
\noLine
\UnaryInfC{$\Delta \vdash A, \Gamma$}
\AxiomC{$p'_2$}
\noLine
\UnaryInfC{$\Delta', A^n \vdash B_2, \Gamma'$}
\LeftLabel{\textsc{(Cut$^{n}$)}}
\BinaryInfC{$\Delta^n, \Delta' \vdash \Gamma^n, B_2, \Gamma'$}
\LeftLabel{\textsc{(XR$^\ast$)}}
\doubleLine
\UnaryInfC{$\Delta^n, \Delta' \vdash B_2, \Gamma^n, \Gamma'$}
\LeftLabel{\textsc{($\&$R)}}
\BinaryInfC{$\Delta^n, \Delta' \vdash B_1 \& B_2, \Gamma^n, \Gamma'$}
\LeftLabel{\textsc{(XR$^\ast$)}}
\doubleLine
\UnaryInfC{$\Delta^n, \Delta' \vdash \Gamma^n, B_1 \& B_2, \Gamma'$}
\DisplayProof
\end{equation*}
\end{scriptsize}The case of the right multiple Cut is similar, and thus we omit it.

\item Right-minor $\oplus$L-Cut. A formal proof of the form
\begin{scriptsize}
\begin{equation*}
\AxiomC{$p$}
\noLine
\UnaryInfC{$\Delta \vdash A, \Gamma$}
\AxiomC{$p'_1$}
\noLine
\UnaryInfC{$\Delta', A^n, B_1 \vdash \Gamma'$}
\AxiomC{$p'_2$}
\noLine
\UnaryInfC{$\Delta', A^n, B_2 \vdash \Gamma'$}
\LeftLabel{\textsc{($\&$R)}}
\BinaryInfC{$\Delta', A^n, B_1 \oplus B_2 \vdash \Gamma'$}
\LeftLabel{\textsc{(CutL$^{n}$)}}
\BinaryInfC{$\Delta^n, \Delta', B_1 \oplus B_2 \vdash \Gamma^n, \Gamma'$}
\DisplayProof
\end{equation*}
\end{scriptsize}
is transformed into
\begin{scriptsize}
\begin{equation*}
\AxiomC{$p$}
\noLine
\UnaryInfC{$\Delta \vdash A, \Gamma$}
\AxiomC{$p'_1$}
\noLine
\UnaryInfC{$\Delta', A^n, B_1 \vdash \Gamma'$}
\LeftLabel{\textsc{(CutL$^{n}$)}}
\BinaryInfC{$\Delta^n, \Delta', B_1 \vdash \Gamma^n, \Gamma'$}
\AxiomC{$p$}
\noLine
\UnaryInfC{$\Delta \vdash A, \Gamma$}
\AxiomC{$p'_2$}
\noLine
\UnaryInfC{$\Delta', A^n, B_2 \vdash \Gamma'$}
\LeftLabel{\textsc{(Cut$^{n}$)}}
\BinaryInfC{$\Delta^n, \Delta', B_2 \vdash \Gamma^n, \Gamma'$}
\LeftLabel{\textsc{($\oplus$L)}}
\BinaryInfC{$\Delta^n, \Delta', B_1 \oplus B_2 \vdash \Gamma^n, \Gamma'$}
\DisplayProof
\end{equation*}
\end{scriptsize}The case of the right multiple Cut is similar, and thus we omit it.

\item Right-minor $\oplus$R-Cut. A formal proof of the form
\begin{scriptsize}
\begin{equation*}
\AxiomC{$p$}
\noLine
\UnaryInfC{$\Delta \vdash A, \Gamma$}
\AxiomC{$p'_i$}
\noLine
\UnaryInfC{$\Delta', A^n \vdash B_i, \Gamma'$}
\LeftLabel{\textsc{($\oplus$R)}}
\RightLabel{$(i \in \overline{2})$}
\UnaryInfC{$\Delta', A^n \vdash B_1 \oplus B_2, \Gamma'$}
\LeftLabel{\textsc{(CutL$^{n}$)}}
\BinaryInfC{$\Delta^n, \Delta' \vdash \Gamma^n, B_1 \oplus B_2, \Gamma'$}
\DisplayProof
\end{equation*}
\end{scriptsize}is transformed into
\begin{scriptsize}
\begin{equation*}
\AxiomC{$p$}
\noLine
\UnaryInfC{$\Delta \vdash A, \Gamma$}
\AxiomC{$p'_i$}
\noLine
\UnaryInfC{$\Delta', A^n \vdash B_i, \Gamma'$}
\LeftLabel{\textsc{(CutL$^{n}$)}}
\BinaryInfC{$\Delta^n, \Delta' \vdash \Gamma^n, B_i, \Gamma'$}
\LeftLabel{\textsc{($\oplus$R, XR$^\ast$)}}
\doubleLine
\RightLabel{$(i \in \overline{2})$}
\UnaryInfC{$\Delta^n, \Delta' \vdash \Gamma^n, B_1 \oplus B_2, \Gamma'$}
\DisplayProof
\end{equation*}
\end{scriptsize}The case of the right multiple Cut is similar, and thus we omit it.

\item Right-minor $\neg$L-Cut. A formal proof of the form
\begin{scriptsize}
\begin{equation*}
\AxiomC{$p$}
\noLine
\UnaryInfC{$\Delta \vdash A, \Gamma$}
\AxiomC{$p'$}
\noLine
\UnaryInfC{$\Delta', A^n \vdash B, \Gamma'$}
\LeftLabel{\textsc{($\neg$L)}}
\UnaryInfC{$\Delta', A^n, \neg B \vdash \Gamma'$}
\LeftLabel{\textsc{(CutL$^{n}$)}}
\BinaryInfC{$\Delta^n, \Delta', \neg B \vdash \Gamma^n, \Gamma'$}
\DisplayProof
\end{equation*}
\end{scriptsize}
is transformed into
\begin{scriptsize}
\begin{equation*}
\AxiomC{$p$}
\noLine
\UnaryInfC{$\Delta \vdash A, \Gamma$}
\AxiomC{$p'$}
\noLine
\UnaryInfC{$\Delta', A^n \vdash B, \Gamma'$}
\LeftLabel{\textsc{(CutL$^{n}$)}}
\BinaryInfC{$\Delta^n, \Delta' \vdash \Gamma^n, B, \Gamma'$}
\LeftLabel{\textsc{($\neg$L, XR$^\ast$)}}
\doubleLine
\UnaryInfC{$\Delta^n, \Delta', \neg B \vdash \Gamma^n, \Gamma'$}
\DisplayProof
\end{equation*}
\end{scriptsize}The case of the right multiple Cut is similar, and thus we omit it.

\item Right-minor $\neg$R-Cut. A formal proof of the form
\begin{scriptsize}
\begin{equation*}
\AxiomC{$p$}
\noLine
\UnaryInfC{$\Delta \vdash A, \Gamma$}
\AxiomC{$p'$}
\noLine
\UnaryInfC{$\Delta', A^n, B \vdash \Gamma'$}
\LeftLabel{\textsc{($\neg$R)}}
\UnaryInfC{$\Delta', A^n \vdash \neg B, \Gamma'$}
\LeftLabel{\textsc{(CutL$^{n}$)}}
\BinaryInfC{$\Delta^n, \Delta' \vdash \Gamma^n, \neg B, \Gamma'$}
\DisplayProof
\end{equation*}
\end{scriptsize}is transformed into
\begin{scriptsize}
\begin{equation*}
\AxiomC{$p$}
\noLine
\UnaryInfC{$\Delta \vdash A, \Gamma$}
\AxiomC{$p'$}
\noLine
\UnaryInfC{$\Delta', A^n, B \vdash \Gamma'$}
\LeftLabel{\textsc{(CutL$^{n}$)}}
\BinaryInfC{$\Delta^n, \Delta', B \vdash \Gamma^n, \Gamma'$}
\LeftLabel{\textsc{($\neg$R, XR$^\ast$)}}
\doubleLine
\UnaryInfC{$\Delta^n, \Delta' \vdash \Gamma^n, \neg B, \Gamma'$}
\DisplayProof
\end{equation*}
\end{scriptsize}The case of the right multiple Cut is similar, and thus we omit it.

\item Right-minor $\oc$D-Cut. A formal proof of the form
\begin{scriptsize}
\begin{equation*}
\AxiomC{$p$}
\noLine
\UnaryInfC{$\Delta \vdash A, \Gamma$}
\AxiomC{$p'$}
\noLine
\UnaryInfC{$\Delta', A^n, B \vdash \Gamma'$}
\LeftLabel{\textsc{($\oc$D)}}
\UnaryInfC{$\Delta', A^n, \oc B \vdash \Gamma'$}
\LeftLabel{\textsc{(CutL$^{n}$)}}
\BinaryInfC{$\Delta^n, \Delta', \oc B \vdash \Gamma^n, \Gamma'$}
\DisplayProof
\end{equation*}
\end{scriptsize}is transformed into
\begin{scriptsize}
\begin{equation*}
\AxiomC{$p$}
\noLine
\UnaryInfC{$\Delta \vdash A, \Gamma$}
\AxiomC{$p'$}
\noLine
\UnaryInfC{$\Delta', A^n, B \vdash \Gamma'$}
\LeftLabel{\textsc{(CutL$^{n}$)}}
\BinaryInfC{$\Delta^n, \Delta', B \vdash \Gamma^n, \Gamma'$}
\LeftLabel{\textsc{($\oc$D)}}
\UnaryInfC{$\Delta^n, \Delta', \oc B \vdash \Gamma^n, \Gamma'$}
\DisplayProof
\end{equation*}
\end{scriptsize}The case of the right multiple Cut is similar, and thus we omit it.

\item Right-minor $\oc$W-Cut. A formal proof of the form
\begin{scriptsize}
\begin{equation*}
\AxiomC{$p$}
\noLine
\UnaryInfC{$\Delta \vdash A, \Gamma$}
\AxiomC{$p'$}
\noLine
\UnaryInfC{$\Delta', A^n \vdash \Gamma'$}
\LeftLabel{\textsc{($\oc$W)}}
\UnaryInfC{$\Delta', A^n, \oc B \vdash \Gamma'$}
\LeftLabel{\textsc{(CutL$^{n}$)}}
\BinaryInfC{$\Delta^n, \Delta', \oc B \vdash \Gamma^n, \Gamma'$}
\DisplayProof
\end{equation*}
\end{scriptsize}is transformed into
\begin{scriptsize}
\begin{equation*}
\AxiomC{$p$}
\noLine
\UnaryInfC{$\Delta \vdash A, \Gamma$}
\AxiomC{$p'$}
\noLine
\UnaryInfC{$\Delta', A^n \vdash \Gamma'$}
\LeftLabel{\textsc{(CutL$^{n}$)}}
\BinaryInfC{$\Delta^n, \Delta' \vdash \Gamma^n, \Gamma'$}
\LeftLabel{\textsc{($\oc$W)}}
\UnaryInfC{$\Delta^n, \Delta', \oc B \vdash \Gamma^n, \Gamma'$}
\DisplayProof
\end{equation*}
\end{scriptsize}The case of the right multiple Cut is similar, and thus we omit it.

\item Right-minor $\oc$C-Cut. A formal proof of the form
\begin{scriptsize}
\begin{equation*}
\AxiomC{$p$}
\noLine
\UnaryInfC{$\Delta \vdash A, \Gamma$}
\AxiomC{$p'$}
\noLine
\UnaryInfC{$\Delta', A^n, \oc B, \oc B \vdash \Gamma'$}
\LeftLabel{\textsc{($\oc$C)}}
\UnaryInfC{$\Delta', A^n, \oc B \vdash \Gamma'$}
\LeftLabel{\textsc{(CutL$^{n}$)}}
\BinaryInfC{$\Delta^n, \Delta', \oc B \vdash \Gamma^n, \Gamma'$}
\DisplayProof
\end{equation*}
\end{scriptsize}is transformed into
\begin{scriptsize}
\begin{equation*}
\AxiomC{$p$}
\noLine
\UnaryInfC{$\Delta \vdash A, \Gamma$}
\AxiomC{$p'$}
\noLine
\UnaryInfC{$\Delta', A^n, \oc B, \oc B \vdash \Gamma'$}
\LeftLabel{\textsc{(CutL$^{n}$)}}
\BinaryInfC{$\Delta^n, \Delta', \oc B, \oc B \vdash \Gamma^n, \Gamma'$}
\LeftLabel{\textsc{($\oc$C)}}
\UnaryInfC{$\Delta^n, \Delta', \oc B \vdash \Gamma^n, \Gamma'$}
\DisplayProof
\end{equation*}
\end{scriptsize}The case of the right multiple Cut is similar, and thus we omit it.

\item Right-minor $\oc$R$^{\oc\wn}$-Cut. A formal proof of the form
\begin{scriptsize}
\begin{equation*}
\AxiomC{$p$}
\noLine
\UnaryInfC{$\oc \Delta \vdash \oc A, \wn \Gamma$}
\AxiomC{$p'$}
\noLine
\UnaryInfC{$\oc \Delta', \oc A^n \vdash B, \wn \Gamma'$}
\LeftLabel{\textsc{($\oc$R$^{\oc\wn}$)}}
\UnaryInfC{$\oc \Delta', \oc A^n \vdash \oc B, \wn \Gamma'$}
\LeftLabel{\textsc{(CutL$^{n}$)}}
\BinaryInfC{$\oc \Delta^n, \oc \Delta' \vdash \wn \Gamma^n, \oc B, \wn \Gamma'$}
\DisplayProof
\end{equation*}
\end{scriptsize}is transformed into
\begin{scriptsize}
\begin{equation*}
\AxiomC{$p$}
\noLine
\UnaryInfC{$\oc \Delta \vdash \oc A, \wn \Gamma$}
\AxiomC{$p'$}
\noLine
\UnaryInfC{$\oc \Delta', \oc A^n \vdash B, \wn \Gamma'$}
\LeftLabel{\textsc{(CutL$^{n}$)}}
\BinaryInfC{$\oc \Delta^n, \oc \Delta' \vdash \wn \Gamma^n, B, \wn \Gamma'$}
\LeftLabel{\textsc{($\oc$R$^{\oc\wn}$, XR$^\ast$)}}
\doubleLine
\UnaryInfC{$\oc \Delta^n, \Delta' \vdash \wn \Gamma^n, \oc B, \Gamma'$}
\DisplayProof
\end{equation*}
\end{scriptsize}

\begin{remark}
Similarly to Remark~\ref{RemarkOC}, the case of the right multiple Cut cannot occur. 
\end{remark}

\item Right-minor $\wn$D-Cut. A formal proof of the form
\begin{scriptsize}
\begin{equation*}
\AxiomC{$p$}
\noLine
\UnaryInfC{$\Delta \vdash A, \Gamma$}
\AxiomC{$p'$}
\noLine
\UnaryInfC{$\Delta', A^n \vdash B, \Gamma'$}
\LeftLabel{\textsc{($\wn$D)}}
\UnaryInfC{$\Delta', A^n \vdash \wn B, \Gamma'$}
\LeftLabel{\textsc{(CutL$^{n}$)}}
\BinaryInfC{$\Delta^n, \Delta' \vdash \Gamma^n, \wn B, \Gamma'$}
\DisplayProof
\end{equation*}
\end{scriptsize}is transformed into
\begin{scriptsize}
\begin{equation*}
\AxiomC{$p$}
\noLine
\UnaryInfC{$\Delta \vdash A, \Gamma$}
\AxiomC{$p'$}
\noLine
\UnaryInfC{$\Delta', A^n \vdash B, \Gamma'$}
\LeftLabel{\textsc{(CutL$^{n}$)}}
\BinaryInfC{$\Delta^n, \Delta' \vdash \Gamma^n, B, \Gamma'$}
\LeftLabel{\textsc{($\wn$D, XR$^\ast$)}}
\doubleLine
\UnaryInfC{$\Delta^n, \Delta' \vdash \Gamma^n, \wn B, \Gamma'$}
\DisplayProof
\end{equation*}
\end{scriptsize}The case of the right multiple Cut is similar, and thus we omit it.

\item Right-minor $\wn$W-Cut. A formal proof of the form
\begin{scriptsize}
\begin{equation*}
\AxiomC{$p$}
\noLine
\UnaryInfC{$\Delta \vdash A, \Gamma$}
\AxiomC{$p'$}
\noLine
\UnaryInfC{$\Delta', A^n \vdash \Gamma'$}
\LeftLabel{\textsc{($\wn$W)}}
\UnaryInfC{$\Delta', A^n \vdash \wn B, \Gamma'$}
\LeftLabel{\textsc{(CutL$^{n}$)}}
\BinaryInfC{$\Delta^n, \Delta' \vdash \Gamma^n, \wn B, \Gamma'$}
\DisplayProof
\end{equation*}
\end{scriptsize}is transformed into
\begin{scriptsize}
\begin{equation*}
\AxiomC{$p$}
\noLine
\UnaryInfC{$\Delta \vdash A, \Gamma$}
\AxiomC{$p'$}
\noLine
\UnaryInfC{$\Delta', A^n \vdash \Gamma'$}
\LeftLabel{\textsc{(CutL$^{n}$)}}
\BinaryInfC{$\Delta^n, \Delta' \vdash \Gamma^n, \Gamma'$}
\LeftLabel{\textsc{($\wn$W, XR$^\ast$)}}
\doubleLine
\UnaryInfC{$\Delta^n, \Delta' \vdash \Gamma^n, \wn B, \Gamma'$}
\DisplayProof
\end{equation*}
\end{scriptsize}The case of the right multiple Cut is similar, and thus we omit it.

\item Right-minor $\wn$C-Cut. A formal proof of the form
\begin{scriptsize}
\begin{equation*}
\AxiomC{$p$}
\noLine
\UnaryInfC{$\Delta \vdash A, \Gamma$}
\AxiomC{$p'$}
\noLine
\UnaryInfC{$\Delta', A^n \vdash \wn B, \wn B, \Gamma'$}
\LeftLabel{\textsc{($\wn$C)}}
\UnaryInfC{$\Delta', A^n \vdash \wn B, \Gamma'$}
\LeftLabel{\textsc{(CutL$^{n}$)}}
\BinaryInfC{$\Delta^n, \Delta' \vdash \Gamma^n, \wn B, \Gamma'$}
\DisplayProof
\end{equation*}
\end{scriptsize}is transformed into
\begin{scriptsize}
\begin{equation*}
\AxiomC{$p$}
\noLine
\UnaryInfC{$\Delta \vdash A, \Gamma$}
\AxiomC{$p'$}
\noLine
\UnaryInfC{$\Delta', A^n \vdash \wn B, \wn B, \Gamma'$}
\LeftLabel{\textsc{(CutL$^{n}$)}}
\BinaryInfC{$\Delta^n, \Delta' \vdash \Gamma^n, \wn B, \wn B, \Gamma'$}
\LeftLabel{\textsc{($\wn$C, XR$^\ast$)}}
\doubleLine
\UnaryInfC{$\Delta^n, \Delta' \vdash \Gamma^n, \wn B, \Gamma'$}
\DisplayProof
\end{equation*}
\end{scriptsize}The case of the right multiple Cut is similar, and thus we omit it.

\item Right-minor $\wn$L$^{\oc\wn}$-Cut. A formal proof of the form
\begin{scriptsize}
\begin{equation*}
\AxiomC{$p$}
\noLine
\UnaryInfC{$\oc \Delta \vdash \oc A, \wn \Gamma$}
\AxiomC{$p'$}
\noLine
\UnaryInfC{$\oc \Delta', \oc A^n, B \vdash \wn \Gamma'$}
\LeftLabel{\textsc{($\wn$L$^{\oc\wn}$)}}
\UnaryInfC{$\oc \Delta', \oc A^n, \wn B \vdash \wn \Gamma'$}
\LeftLabel{\textsc{(CutL$^{n}$)}}
\BinaryInfC{$\oc \Delta^n, \oc \Delta', \wn B \vdash \wn \Gamma^n, \wn \Gamma'$}
\DisplayProof
\end{equation*}
\end{scriptsize}is transformed into
\begin{scriptsize}
\begin{equation*}
\AxiomC{$p$}
\noLine
\UnaryInfC{$\oc \Delta \vdash \oc A, \wn \Gamma$}
\AxiomC{$p'$}
\noLine
\UnaryInfC{$\oc \Delta', \oc A^n, B \vdash \wn \Gamma'$}
\LeftLabel{\textsc{(CutL$^{n}$)}}
\BinaryInfC{$\oc \Delta^n, \oc \Delta', B \vdash \wn \Gamma^n, \wn \Gamma'$}
\LeftLabel{\textsc{($\wn$L$^{\oc\wn}$)}}
\UnaryInfC{$\oc \Delta^n, \oc \Delta', \wn B \vdash \wn \Gamma^n, \wn \Gamma'$}
\DisplayProof
\end{equation*}
\end{scriptsize}

\begin{remark}
Again, the case of the right multiple Cut cannot occur. 
\end{remark}

\subsection{Left minor cuts}
\label{LeftMinorCuts}
The third pattern is an application of Cut such that the principal formula of the rule at the end of the left hypothesis of the application of Cut is not the cut formula. 
Because it is just symmetric to the second pattern, we omit the cases of the third pattern.

\subsection{Identity cuts}
\label{IdentityCuts}
Finally, the fourth pattern is an application of Cut such that at least one of the two hypotheses is a singleton formal proof of the axiom Id.

The cases of the fourth pattern are the following:

\item Left Id-Cut. A formal proof of the form
\begin{scriptsize}
\begin{equation*}
\AxiomC{}
\LeftLabel{\textsc{(Id)}}
\UnaryInfC{$A \vdash A$}
\AxiomC{$p$}
\noLine
\UnaryInfC{$\Delta, A^n \vdash \Gamma$}
\LeftLabel{\textsc{(CutL$^{n}$)}}
\BinaryInfC{$A^n, \Delta \vdash \Gamma$}
\DisplayProof
\end{equation*}
\end{scriptsize}is transformed into
\begin{scriptsize}
\begin{equation*}
\AxiomC{$p$}
\noLine
\UnaryInfC{$\Delta, A^n \vdash \Gamma$}
\LeftLabel{\textsc{(XL$^\ast$)}}
\doubleLine
\UnaryInfC{$A^n, \Delta \vdash \Gamma$}
\DisplayProof
\end{equation*}
\end{scriptsize}

\item Right Id-Cut. A formal proof of the form
\begin{scriptsize}
\begin{equation*}
\AxiomC{$p$}
\noLine
\UnaryInfC{$\Delta \vdash A^n, \Gamma$}
\AxiomC{}
\LeftLabel{\textsc{(Id)}}
\UnaryInfC{$A \vdash A$}
\LeftLabel{\textsc{(CutR$^{n}$)}}
\BinaryInfC{$\Delta \vdash \Gamma, A^n$}
\DisplayProof
\end{equation*}
\end{scriptsize}is transformed into
\begin{scriptsize}
\begin{equation*}
\AxiomC{$p$}
\noLine
\UnaryInfC{$\Delta \vdash A^n, \Gamma$}
\LeftLabel{\textsc{(XR$^\ast$)}}
\doubleLine
\UnaryInfC{$\Delta \vdash \Gamma, A^n$}
\DisplayProof
\end{equation*}
\end{scriptsize}

\end{itemize}

We have considered all the cases and thus completed the description of our cut-elimination procedure for the sequent calculus \textsf{ILC}.

\end{document}